\documentclass[a4paper,12pt]{preprint}
\usepackage[left=1.3in,right=1.3in,bottom=1.4in]{geometry}  
\usepackage[T1]{fontenc}
\usepackage[full]{textcomp}
\usepackage[osf]{newtxtext}
\usepackage{mathalfa} 
\usepackage{microtype}
\usepackage{macros}

\usepackage{booktabs}
\usepackage{enumitem}
\usepackage{amsmath}
\usepackage{amsfonts} 
\usepackage{amssymb}             
\usepackage{stmaryrd}
\usepackage{comment}
\usepackage{ulem}
\usepackage{mathrsfs}
\usepackage{booktabs}
\usepackage{mathtools}
\usepackage{tikz}
\usepackage{chngcntr}
\usepackage{amsthm}                 
\usepackage{mhequ}
\usepackage{orcidlink}
\usepackage{cprotect}
\usepackage{hyperref}

\definecolor{labelkey}{gray}{.8}
\definecolor{refkey}{gray}{.8}

\hypersetup{pdfstartview=}
\addtocontents{toc}{\protect\hypersetup{hidelinks}}
\DeclareSymbolFont{sfoperators}{OT1}{ptm}{m}{n}
\DeclareSymbolFontAlphabet{\mathsf}{sfoperators}

\makeatletter
\def\operator@font{\mathgroup\symsfoperators}
\makeatother
\newcommand{\eqlaw}{\stackrel{\mbox{\tiny\rm law}}{=}}
\newcommand{\eqdef}{\stackrel{\mbox{\tiny\rm def}}{=}}

\newtheorem{thm}{Theorem}[section]
\newtheorem{defn}[thm]{Definition}
\newtheorem{lem}[thm]{Lemma}
\newtheorem{prop}[thm]{Proposition}

\newtheorem{cor}[thm]{Corollary}

\newtheorem{assumption}[thm]{Assumption}
\theoremstyle{remark}
\newtheorem{remark}[thm]{Remark}
\newtheorem{rmk}[thm]{Remark}

\DeclareMathOperator{\Cov}{Cov}
\DeclareMathOperator{\Span}{Span}
\DeclareMathOperator{\id}{id}

\DeclareMathOperator{\supp}{supp}

\makeatletter
\def\th@newremark{\th@remark\thm@headfont{\bfseries}}

\def\bdiamond{\mathop{\mathpalette\bdi@mond\relax}}
\newcommand\bdi@mond[2]{%
\vcenter{\hbox{\m@th
\scalebox{\ifx#1\displaystyle 2.6\else1.8\fi}{$#1\diamond$}%
}}%
}

\def\bDiamond{\mathop{\mathpalette\bDi@mond\relax}}
\newcommand\bDi@mond[2]{%
\vcenter{\hbox{\m@th
\scalebox{\ifx#1\displaystyle 2.6\else1.2\fi}{$#1\DDDiamond$}%
}}%
}

\newcommand*{\bigcdot}{}
\DeclareRobustCommand*{\bigcdot}{%
  \mathord{\hspace{0.1em}\mathpalette\bigcdot@{}\hspace{0.1em}}%
}
\newcommand*{\bigcdot@scalefactor}{.5}
\newcommand*{\bigcdot@widthfactor}{1.15}
\newcommand*{\bigcdot@}[2]{%
  \sbox0{$#1\vcenter{}$}
  \sbox2{$#1\cdot\m@th$}%
  \hbox to \bigcdot@widthfactor\wd2{%
    \hfil
    \raise\ht0\hbox{%
      \scalebox{\bigcdot@scalefactor}{%
        \lower\ht0\hbox{$#1\bullet\m@th$}%
      }%
    }%
    \hfil
  }%
}
\makeatother

\newcommand{\EE}{\mathbb{E}}     

\newcommand{\PP}{\mathbb{P}}     

\newcommand{\aA}{\mathcal{A}}
\newcommand{\bB}{\mathcal{B}}
\newcommand{\cC}{\mathcal{C}}
\newcommand{\dD}{\mathcal{D}}

\newcommand{\fF}{\mathcal{F}}
\newcommand{\gG}{\mathcal{G}}

\newcommand{\iI}{\mathcal{I}}

\newcommand{\kK}{\mathcal{K}}
\newcommand{\lL}{\mathcal{L}}
\newcommand{\mM}{\mathcal{M}}

\newcommand{\oO}{\mathcal{O}}
\newcommand{\pP}{\mathcal{P}}
\newcommand{\qQ}{\mathcal{Q}}
\newcommand{\rR}{\mathcal{R}}
\newcommand{\sS}{\mathcal{S}}
\newcommand{\tT}{\mathcal{T}}

\newcommand{\xX}{\mathcal{X}}

\def\CO{\mathcal{O}}

\newcommand{\fa}{\mathfrak{a}}

\newcommand{\fb}{\mathfrak{b}}
\newcommand{\fc}{\nu}
\newcommand{\fX}{\mathfrak{X}}

\newcommand{\fK}{\mathfrak{K}}

\renewcommand{\C}{\mathbf{C}}
\newcommand{\E}{\mathbf{E}}

\renewcommand{\P}{\mathbf{P}}

\newcommand{\N}{\mathbf{N}}

\newcommand{\R}{\mathbf{R}}
\newcommand{\T}{\mathbf{T}}

\newcommand{\Z}{\mathbf{Z}}
\newcommand{\0}{\mathbf{0}}

\newcommand{\one}{\mathbf{1}}

\newcommand{\s}{\mathbf{a}}

\renewcommand{\d}{\partial}

\newcommand{\z}{\mathbf{z}}

\newcommand{\eps}{\varepsilon}

\newcommand{\Laplace}{\Delta}

\newcommand{\md}{\mathrm{d}}

\newcommand{\zzz}{z}

\newcommand{\VVoneB}{V^{(1\mathrm{b})}_{s}}
\newcommand{\VVoneC}{V^{(1\mathrm{c})}_{s}}

\newcommand{\VVmodone}{V^{(1)}_{s}}
\newcommand{\VVmodoneA}{V^{(1\mathrm{a})}_{s}}
\newcommand{\VVmodoneB}{V^{(1\mathrm{b})}_{s}}
\newcommand{\VVmodoneC}{V^{(1\mathrm{c})}_{s}}
\newcommand{\VVmodtwo}{V^{(2)}_{s}}
\newcommand{\DDD}{D_{\eps,\ell}}

\newcommand{\DDDmodone}{\hat D^{(1)}_{s}}
\newcommand{\PHI}{\Phi_{\eps,\ell}}
\newcommand{\PHImod}{\Phi^{(H)}_{\eps,\ell}}
\newcommand{\PSI}{\Psi_{\eps,\ell,s}}
\newcommand{\PSImod}{\tilde\Psi_{\eps,\ell,s}}
\newcommand{\weight}{\mathtt{w}}
\def\rmw{\weight}

\newcommand{\VERT}[1]{{\lvert\kern-0.28ex\lvert\kern-0.28ex\lvert #1
    \rvert\kern-0.28ex\rvert\kern-0.28ex\rvert}}

\let\f\frac

\newcommand{\ou}{\<1_black> }
\newcommand{\oud}{\<2_black> }
\newcommand{\out}{\<3_black>}
\newcommand{\outz}{\<K*3_black>}
\newcommand{\outu}{\<1K*3_black>}
\newcommand{\oudz}{\<K*2_black>}

\newcommand{\outd}{\<2K*3_black>}
\newcommand{\oudd}{\<2K*2_black>}

\newcommand{\XXX}{\mathbf{X}}

\newcommand{\bou}{\<1_blue> }
\newcommand{\boud}{\<2_blue> }
\newcommand{\bout}{\<3_blue>}
\newcommand{\boutz}{\<K*3_blue>}
\newcommand{\boutu}{\<1K*3_blue>}
\newcommand{\boudz}{\<K*2_blue>}

\newcommand{\boutd}{\<2K*3_blue>}
\newcommand{\boudd}{\<2K*2_blue>}

\newcommand{\bXXX}{\blue{\mathbf{X}}}

\newcommand{\oneb}{\blue{\mathbf{1}}}
\newcommand{\blue}[1]{{\color{blue}#1}}
\newcommand{\Cone}{C_{\eps,\ell}^{(1)}}
\newcommand{\Ctwo}{C_{\eps,\ell}^{(2)}}
\newcommand{\Cones}{C_{\eps,\ell,s}^{(1)}}
\newcommand{\Ctwos}{C_{\eps,\ell,s}^{(2)}}

\newcommand{\Besovt}[2][e_{\delta+t,\zzz}]{\bB^{#2}_{p, p}(#1)}

\newcommand{\grad}{\nabla}

\newcommand{\Besov}[2]{\bB^{#1}_{#2}}
\newcommand{\bracket}[1]{\left\langle{#1}\right\rangle}

\newcommand{\para}{\olessthan}
\newcommand{\rpara}{\ogreaterthan}
\newcommand{\reso}{\odot}

\newcommand{\pareq}{\tikzsetnextfilename{pareq}\mathrel{\rlap{\kern0.07em\tikz[x=0.1em,y=0.1em,baseline=0.04em] \draw[line width=0.3pt,black] (0,1.93) -- (4.5,0);}\olessthan}}
\newcommand{\rpareq}{\tikzsetnextfilename{rpareq}\mathrel{\rlap{\kern0.23em\tikz[x=0.1em,y=0.1em,baseline=0.04em] \draw[line width=0.3pt,black] (0,0) -- (4.5,1.93);}\ogreaterthan}}

\def\dash{\leavevmode\unskip\kern0.18em--\penalty\exhyphenpenalty\kern0.18em}
\def\slash{\leavevmode\unskip\kern0.15em/\penalty\exhyphenpenalty\kern0.15em}

\makeatletter
\DeclareRobustCommand{\TitleEquation}[2]{\texorpdfstring{\StrLeft{\f@series}{1}[\@firstchar]$\if%
b\@firstchar\boldsymbol{#1}\else#1\fi$}{#2}}

\makeatother

\DeclarePairedDelimiter\abs{\lvert}{\rvert}

\newcounter{step}
\renewcommand*{\thestep}{(\Alph{step})}
\counterwithin*{step}{section}
\counterwithin*{step}{subsection}
\def\step{\refstepcounter{step}\smallskip\noindent\textit{\thestep} \textit}

\colorlet{darkblue}{blue!90!black}
\colorlet{darkred}{red!90!black}
\colorlet{darkgreen}{green!50!black}
\definecolor{darkpurple}{rgb}{0.65,0.1,0.7}

\begin{document}
\title{Ergodicity of infinite volume \TitleEquation{\Phi^4_3}{Phi 43} at high temperature}

\author{Pawe\l\ Duch\inst1 \orcidlink{0000-0002-2445-2189}, Martin Hairer\inst{1,2}\orcidlink{0000-0002-2141-6561}, 
Jaeyun Yi\inst1 \orcidlink{0000-0002-2940-9307}, and
Wenhao Zhao\inst1 \orcidlink{0000-0002-7489-1884}}

\institute{EPFL, Switzerland, \email{pawel.duch@epfl.ch, martin.hairer@epfl.ch, stork3827@gmail.com, wenhao.zhao@epfl.ch}
\and Imperial College London, UK, \email{m.hairer@imperial.ac.uk}  
}

\maketitle

\begin{abstract}
We consider the infinite volume $\Phi^4_3$ dynamic and show that it is globally 
well-posed in a suitable weighted Besov space of distributions. At 
high temperatures \slash small coupling, we furthermore show that the difference between
any two solutions driven by the same realisation of the noise converges to zero exponentially fast.
This allows us to characterise the infinite-volume $\Phi^4_3$ measure at high temperature as the unique invariant measure of the dynamic, and to prove that it satisfies all Osterwalder--Schrader axioms, including invariance under translations, rotations, and reflections, as well as exponential decay of correlations.

\vspace{1em}

\end{abstract}

\tableofcontents

\setcounter{tocdepth}{2}
\microtypesetup{protrusion=false}
\microtypesetup{protrusion=true}

\section{Introduction}

The simplest interacting bosonic field theory is the so-called $\Phi^4$ theory
with (formal) Lagrangian given by 
\begin{equ}
H_{m,\lambda}(\Phi) = \int \Bigl(\f12 |\nabla \Phi(x)|^2 + \f m2|\Phi(x)|^2 + \f\lambda4  |\Phi(x)|^4\Bigr)\,\md x\;.
\end{equ}
One major achievement of the programme of constructive field theory 
that took place in the late 70's was the construction of a family (parametrised by $m$ and $\lambda$) 
of  non-Gaussian probability measures on the
space of Schwartz distributions $\dD'(\R^d)$ with $d < 4$ that exhibits all the properties
one would expect from the measures formally given by 
$Z_{m,\lambda}^{-1}\exp(-2H_{m,\lambda}(\Phi))\,\md\Phi$, with $\md\Phi$ denoting the 
(non-existent) Lebesgue measure on $\dD'(\R^d)$ and $Z_{m,\lambda}$ denoting the normalisation constant enforcing 
that the measures are probability measures. 
See for example \cite{Fel74,FO76,MS76,GJ87} and references therein for the original construction.
When $d \ge 4$, there is strong evidence \cite{Aiz82,Fro82,AD21} that no such measures exist in the sense that limit points of their natural approximations
all turn out to be Gaussian.

In the present article we will always consider the case $d=3$, with mass $m=1$ and coupling constant $\lambda > 0$ small. 
When $\lambda = 0$, the measure described above can unambiguously be defined (in any dimension) as 
the Gaussian measure
with covariance function given by the Green function of the selfadjoint operator $1-\Delta$.
For $\eps \ll 1$ and $\ell \gg 1$ with $\ell \in \eps \N$, let  $\T_{\eps,\ell}^d = (\eps \Z / \ell \Z)^d$ be the 
discrete torus of size $\ell$ and let $\P_{\eps,\ell}$ be the Gaussian measure 
on $\R^{\T_{\eps,\ell}^d}$ with covariance given by the inverse of the 
matrix $\id -\Delta_\eps$, where $\Delta_\eps$ is the discrete Laplacian. In order
to have any chance of obtaining a nontrivial limiting measure, one needs to
``renormalise'' the mass $m$ in $H$ by considering the approximation
\begin{equ}[e:muepsell]
\hat \mu_{\eps,\ell}(\md\Phi) = Z_{\eps,\ell}^{-1} \exp \Biggl(-2\lambda \int_{\T_{\eps,\ell}^d} \Bigl( \frac{|\Phi(x)|^4}{4}- \bigl(3 C_\eps^{(1)} - 9\lambda C_\eps^{(2)}\bigr)\frac{|\Phi(x)|^2}{2}\Bigr)\,dx\biggr)\,\P_{\eps,\ell}(\md\Phi)\;,
\end{equ}
where $dx$ denotes $\eps^d$ times the counting measure,
$C_\eps^{(1)}$ denotes the variance of $\Phi(x)$ under $\P_{\eps,\ell}$
(which is asymptotically independent of $\ell$ as $\ell \to \infty$), and 
$C_\eps^{(2)}$ is an additional correction that, in dimension $3$, diverges like $\abs{\log\eps}$
as $\eps \to 0$.

\begin{remark}\label{rem:mass}
For any fixed $\lambda$, the value of the ``mass'' $m$ (including its sign) can be adjusted simply 
by changing $C_\eps^{(2)}$ by some $\CO(1)$ quantity. In this article however,
we consider the renormalisation constants as fixed functions of $\eps$ and then choose 
$\lambda$ sufficiently small, so that the sign of the mass is well defined.
See also Section~\ref{sec:regimes} below for a discussion on how ``large mass'', ``small coupling'', and ``high temperature''
are essentially equivalent notions in our context, so the focus on $\lambda$ as our free parameter
is arbitrary and just made for convenience. 
\end{remark}

The stochastic quantisation procedure originally proposed by Parisi and Wu \cite{PW81}
is based on the observation that $\hat \mu_{\eps,\ell}$ is the (unique) invariant measure
for the stochastic differential equation
\begin{equ}[e:Phi4approx]
\md\Phi_{\eps,\ell} = \bigl(\Delta_\eps \Phi_{\eps,\ell} - \Phi_{\eps,\ell} - \lambda \Phi_{\eps,\ell}^3 + (3 \lambda C_\eps^{(1)} - 9\lambda^2 C_\eps^{(2)}) \Phi_{\eps,\ell}\bigr)\,\md t + \md W_{\eps,\ell}\;,
\end{equ} 
where $W_{\eps,\ell}$ denotes the cylindrical Wiener process on $L^2(\T_{\eps,\ell}^d)$.
The theories of regularity structures \cite{Hai14} and paracontrolled distributions \cite{GIP15}
were developed in part in order to provide a 
meaning to the limit of $\Phi_{\eps,\ell}$ as $\eps \to 0$. 
The idea is to consider the mild form of the equation
\begin{equ}[e:Phi4]
\md\Phi = (\Delta \Phi - \Phi - \lambda \Phi^3)\,\md t + \md W\;,
\end{equ}
as a fixed point problem in a space of \textit{modelled distributions} that are 
locally described by a linear combination of elements of a \textit{model}, similarly
to the way in which smooth functions can locally be described by a Taylor polynomial.
The interpretation of the term $\Phi^3$ (and in particular the appearance of the renormalisation
constants that are apparent in \eqref{e:Phi4approx}) is then encoded in the construction of the model,
which is where renormalisation takes place. One advantage of this perspective is that
it provides an \textit{intrinsic} meaning to solutions to \eqref{e:Phi4} which can
then be shown to coincide with the limits of a large number of different regularisations. 
One a priori obtains a well-posed local solution theory for \eqref{e:Phi4} in finite volume,
but it was shown in \cite{MW17a,MW17b,GH19,AK20,MW20} that this solution theory is global in time
with very strong a priori bounds. In particular, the size of the solutions remains bounded as the size 
of the domain tends to infinity.

Our first main result is that one has an intrinsic solution theory for \eqref{e:Phi4}
on all of $\R^3$. A loose formulation of this result is as follows, where $E$ denotes
some weighted space of distributions in $\cC^{-1/2-\kappa}$ (for $\kappa$ small)
that allows for some slow algebraic growth. The precise formulation of this result is
provided in Theorem~\ref{thm:global_solution} below.

\begin{thm}\label{thm:A}
For the regularity structure associated to \eqref{e:Phi4} as in \cite{Hai14}, 
consider the model given by the BPHZ lift of space-time white noise as in \cite{BHZ19,HS24}
as well as an initial condition belonging to $E$.
Then, the mild form of \eqref{e:Phi4} posed on all of $\R^3$ admits a unique solution in some suitable
weighted space of modelled distributions. Furthermore, the reconstruction of this 
solution coincides with the limit
$\lim_{\ell \to \infty}\lim_{\eps \to 0} \Phi_{\eps,\ell}$, belongs to $E$ for all 
times, depends continuously on the initial data, and admits an invariant measure.
\end{thm}

Note that this result holds for all (strictly positive) values of the constant $\lambda$ and, 
as already hinted at in Remark~\ref{rem:mass}, it consequently also holds for 
all values of $m$ (not just positive ones).

Our second main result is then that, 
when $\lambda$ is small enough, solutions to \eqref{e:Phi4} not only admit
a unique invariant measure, but they satisfy a ``one force, one solution''
principle or, in other words, they admit a unique global random fixed point.
This can be formulated as follows, 
see Theorem~\ref{thm:linearised} for the precise statement.

\begin{thm}
	There exist $\lambda_\star,\gamma > 0$ such that, for all $\lambda \in (0,\lambda_\star]$ the Markov process constructed in Theorem~\ref{thm:A} admits a unique invariant measure in $E$, and $\E \|\Phi_t - \tilde \Phi_t\|_E \lesssim e^{-\gamma t}$, uniformly over $t \ge 1$ and  $\Phi,\tilde\Phi$ solving~\eqref{e:Phi4}.
\end{thm}

\begin{remark}
One (almost) immediate consequence of these results is that the $\Phi^4_3$ measure
is translation, rotation, and reflection invariant. Combining this with a coupling method, one also obtains exponential decay of correlations, see Theorem~\ref{thm:uniqueness}.
\end{remark}

\begin{rmk}	
In the recent work \cite{RolandHendrik}, the authors used the log-Sobolev inequality established in \cite{RolandLogSob} to prove exponential ergodicity in the whole high temperature regime for $\Phi^4_2$. Since a log-Sobolev inequality has also been established for the $\Phi^4_3$ model in~\cite{RolandLogSob}, it is conceivable that their strategy could be adapted to the $\Phi^4_3$ setting. A key feature of our approach is that it relies solely on PDE techniques and does not require any prior information about the invariant measure. In particular, our method is quite robust and can be extended to the $\oO(N)$ vector-valued $\Phi^4_3$ model and the $\pP(\Phi)_2$ model. One disadvantage however is that we are not able to cover
the entire high temperature regime up to the phase transition.
\end{rmk}

\begin{remark} \label{rmk:invariance_gibbs}
In the discrete case, it was shown in \cite{doss1978processus,HS77,FritzGibbs,BRW04} that that Gibbs measures are equivalent
to invariant measures for the corresponding infinite-dimensional SDE. 
In the continuum case, it is not clear a priori how to even formulate the Gibbs property for $\Phi^4_3$,
but the formulation for $\Phi^4_2$ is clear since in finite volume it is absolutely continuous with respect to the free field. We believe that it is much easier to show that every
Gibbs measure is invariant for the infinite-volume dynamic than the converse.
In this sense, our result is a strong form of uniqueness for the $\Phi^4_3$ measure
at high temperature. An intrinsic continuum formulation of the Gibbs property for the $\Phi^4$ model 
and the uniqueness of the corresponding Gibbs measure at high temperature have been established 
in two dimensions \cite{AHZ89a, AHZ89b},
but a rigorous formulation of the relation between Gibbs measures and invariant measures is beyond the scope
of the present article. 
Regarding ergodicity for the $\Phi^4_2$ Langevin \slash Glauber dynamic, 
partial progress (showing that extremal Gibbs states are necessarily ergodic invariant measures for the dynamic) was made 
in \cite{albeverio1997ergodicity} and the problem was solved completely
in \cite{RolandHendrik} (all the way to the critical point). For $\Phi^4_3$, the recent work \cite{BG25} proves the domain Markov property on a cylinder, which may be relevant for formulating the Gibbs property.
\end{remark}

\begin{remark}
At fixed $\ell > 0$, the uniqueness of the invariant measure for the process $\Phi_\ell = \lim_{\eps \to 0} \Phi_{\eps,\ell}$ 
follows from the fact that it has full support \cite{HS22b} and satisfies the strong Feller property \cite{HM18b}. 
In infinite volume, there is no reason in general to expect the strong Feller property to hold
since it already fails for the massive stochastic heat equation. 
\end{remark}

\subsection{Short literature review}

It has been known since the seventies that bosonic QFTs satisfying the Wightman axioms \cite{Wig76} can be obtained from probability measures on the space of tempered distributions satisfying the Osterwalder--Schrader (OS) axioms, namely Euclidean invariance, reflection positivity, and decay of correlations, as well as some regularity properties (see for example \cite{OS75}, \cite[Section~6.1]{GJ87} for more details). Assuming a small coupling constant $\lambda$, the construction of the $\Phi^4_3$ measure and the verification of the OS axioms was completed in \cite{GJ73,FO76} using the phase-cell 
expansion method and in \cite{MS76} by the cluster expansion method. There have been subsequent efforts (e.g.\ \cite{brydges1983new, Wat89, BDH95}, etc.) to provide simpler proofs of the results in \cite{FO76, MS76}.
We also refer to \cite{GH21} for a recent review on the subject.

As already pointed out, the idea of stochastic quantisation proposed in \cite{PW81} is to view the $\Phi^4_3$ measure as the invariant measure of the $\Phi^4_3$ dynamic \eqref{e:Phi4}, for which we are now 
able to give an intrinsic rigorous meaning. Therefore, it is natural to revisit the construction 
of $\Phi^4_3$ from this dynamical perspective. There has been much recent progress in this direction. 
In \cite{GH21, DGR23} the authors proved the tightness of lattice approximation to the $\Phi^4_3$ 
measure and the OS axioms of every accumulation point except for the rotation invariance and the 
clustering properties. The quartic exponential tails of the $\Phi^4_3$ measure and a simple proof of 
its non-Guassianity were obtained in \cite{HS22}. A concise proof of the Euclidean invariance 
of the $\pP(\Phi)_2$ measure using stochastic quantisation techniques was given in~\cite{DDJ25}.

The present work may be seen as the culmination of the stochastic quantisation program for the $\Phi^4_3$ model. By employing techniques from stochastic partial differential equations, we verify all the OS axioms for $\Phi^4_3$ in the small-coupling regime, thereby recovering the results of~\cite{FO76, MS76} via an entirely different approach. Moreover, we construct the infinite volume dynamic \eqref{e:Phi4}, and prove that when $\lambda$ is small, it is exponentially mixing and admits a unique invariant measure.

Ergodicity and exponential decay of correlations for SPDEs in infinite volume have previously been studied in \cite{Fun91, GHR25}, under the assumption that the nonlinearity is convex. In the case of $\Phi^4$, convexity is destroyed by renormalisation, which is the main challenge of the current work. To address this, our main input is the new bound \eqref{e:linearised_1} for the linearised equation \eqref{e:linearised}. Since the linearised equation takes the form of a Parabolic Anderson Model, it is natural to try to apply the argument in \cite{HL15}. However, a direct application yields bounds with exponentially growing time-dependent weights and poor probabilistic integrability. To overcome this, we exploit the spatial stationarity of the enhanced noise, employ the stopping time argument from \cite{KT22} and use the coming down from infinity property from \cite{MW20}.

In the low temperature regime, one expects multiple invariant measures for the dynamic \eqref{e:Phi4}. The low temperature regime was studied in \cite{GJS75,GJS76a, GJS76b} for $\Phi^4_2$, and in \cite{FSS76,CGW22} for $\Phi^4_3$. It would also be interesting to study the dynamic \eqref{e:Phi4} in this regime, and to derive properties of the invariant measures from it.

\subsection{Relation between parameter regimes}
\label{sec:regimes}

Let us discuss in a bit more detail the relation between ``temperature'', ``mass'', and ``coupling''.
Recall that $\hat \mu_{\eps,\ell}$ can be written as
\begin{equ}
\hat \mu_{\eps,\ell}(\md\Phi) \propto \exp \bigl(- 2H_{\lambda}^{(\ell,\eps)}\bigr)\, \md\Phi\;,
\end{equ}
for some ``renormalised'' discretisation $H_{\lambda}^{(\ell,\eps)}$ of $H_{1,\lambda}$.

Introducing an inverse temperature $\beta$, it would be natural to also consider the measure 
``at inverse temperature $\beta$'' given by $\exp(-2\beta H_{\lambda}^{(\ell,\eps)}(\Phi))\,\md\Phi$, which 
can be written as
\begin{equ}
 \exp \Biggl(-2\beta\lambda \int_{\T_{\eps,\ell}^d} \Bigl( \frac{|\Phi(x)|^4}{4}- \bigl(3 C_\eps^{(1)} - 9\lambda C_\eps^{(2)}\bigr)\frac{|\Phi(x)|^2}{2}\Bigr)\,dx\biggr)\,\P_{\eps,\ell}^{(\beta)}(\md\Phi)\;,
\end{equ}
where $\P_{\eps,\ell}^{(\beta)}$ has covariance $(2\beta)^{-1}(\id-\Delta_\eps)^{-1}$.
Since $\P_{\eps,\ell}^{(\beta)}$ is the image of $\P_{\eps,\ell}$ under multiplication of $\Phi$ by $\beta^{-1/2}$,
this is essentially equivalent to considering the measure
\begin{equ}
\hat \mu_{\eps,\ell}^{(\beta)}(\md\Phi) \propto \exp \Biggl(-2\lambda \int_{\T_{\eps,\ell}^d} \Bigl( \frac{|\Phi(x)|^4}{4\beta}- \bigl(3 C_\eps^{(1)} - 9\lambda C_\eps^{(2)}\bigr)\frac{|\Phi(x)|^2}{2}\Bigr)\,dx\biggr)\,\P_{\eps,\ell}(\md\Phi)\;.
\end{equ}
Setting $\hat \lambda = \lambda/\beta$ and making the dependence on $\lambda$ explicit, one finds that
\begin{equ}
\hat \mu_{\eps,\ell,\lambda}^{(\beta)}(\md\Phi) \propto \exp \Bigl(-\delta m\int_{\T_{\eps,\ell}^d}|\Phi(x)|^2\,dx\Bigr) \hat \mu_{\eps,\ell,\hat \lambda}(\md\Phi)\;,
\end{equ}
with $\delta m = 3\hat \lambda (1-\beta) C_\eps^{(1)} + 9\hat \lambda^2 (\beta^2-1) C_\eps^{(2)}$.
This shows that the temperature is in fact essentially fixed: if we want $\hat \mu_{\eps,\ell,\lambda}^{(\beta)}$ to have
a non-trivial limit as $\eps \to 0$, then $|\beta-1|$ can be at most of order $1/C_\eps^{(1)}\approx \eps$.
This is consistent with \cite{MW17c, HI18, GMW25} where the authors derive the $\Phi^4_d$ measure as the scaling 
limit of a long-range Ising model near its critical temperature. Furthermore, since $C_\eps^{(1)}\gg C_\eps^{(2)}$, we see that $\beta < 1$ (``high temperature'')
yields a positive change $\delta m$ of the mass, while the coupling $\lambda$ remains essentially unchanged since 
$\beta$ is very close to $1$.

On the other hand, one finds that, setting $m = 1+\delta m$, the image of the measure 
\begin{equ}[e:massive]
\exp \Bigl(-\delta m\int_{\T_{\eps,\ell}^d}|\Phi(x)|^2\,dx\Bigr) \P_{\eps,\ell}(\md\Phi)\;,
\end{equ}
under the map $\Phi \mapsto m^{1/4} \Phi(\sqrt m \bigcdot)$
is given by $\P_{\sqrt m\eps,\sqrt m\ell}(\md\Phi)$, so that, setting
$\tilde \eps = \sqrt m\eps$ and $\tilde \ell = \sqrt m\ell$, the measure $\hat \mu_{\eps,\ell,\lambda}^{(\beta)}$
is essentially equivalent to the measure
\begin{equ}
\exp \Biggl(-2\hat\lambda \int_{\T_{\tilde\eps,\tilde\ell}^d} \Bigl( \frac{|\Phi(x)|^4}{4\sqrt m}- \bigl(3 C_\eps^{(1)} - 9\hat\lambda C_\eps^{(2)}\bigr)\frac{|\Phi(x)|^2}{2m}\Bigr)\,dx\biggr)\,\P_{\tilde\eps,\tilde\ell}(\md\Phi)\;.
\end{equ}
Setting $\tilde \lambda = \hat \lambda / \sqrt m$
and noting that $C_\eps^{(1)} \propto \eps^{-1}$ while $C_\eps^{(2)} \propto \abs{\log \eps}$,
there is a constant $c>0$ such that this in turn equals
\begin{equ}
\exp \Biggl(-2\tilde\lambda \int_{\T_{\tilde\eps,\tilde\ell}^d} \Bigl( \frac{|\Phi(x)|^4}{4}- \bigl(3 C_{\tilde\eps}^{(1)} - 9\tilde\lambda C_{\tilde\eps}^{(2)}+ c \tilde\lambda \log m\bigr)\frac{|\Phi(x)|^2}{2}\Bigr)\,dx\biggr)\,\P_{\tilde \eps,\tilde\ell}(\md\Phi)\;.
\end{equ}
In other words, the measure with mass $m>1$ and coupling $\lambda$ is equivalent to the measure with mass
$1 - c\frac{\log m}m\lambda^2$ and coupling $\lambda/\sqrt m$, so that ``high temperature'', ``large mass'' 
and ``small coupling'' are equivalent regimes.

\begin{remark}
The somewhat strange correction $c\frac{\log m}m\lambda^2$ appearing here is a consequence of the
fact that even when $m \neq 1$, our renormalisation constants $C_\eps^{(1)}$ and $C_\eps^{(2)}$ are defined in
a way that doesn't depend on $m$.
\end{remark}

\subsection*{Acknowledgements}
WZ is grateful to Nimit Rana for interesting discussions on \cite{GHR25}. 

\section{Main technical results}

In this section, we state the precise formulations of our main results. To set the stage, we begin by collecting known facts about the finite volume dynamic on a torus $\T_\ell^3$ of length $\ell\in\N_+$. Next, we present our key new results: construction of the infinite volume dynamic and a decay estimate for the solutions to the linearised equation. The proofs of these results are deferred to Sections~\ref{sec:ergodicity} and~\ref{sec:global_solution_theory}, respectively. Finally, we discuss a number of applications of these results.

By a function \slash distribution on $\T_\ell^3$ we mean a periodic
function \slash distribution on $\R^3$ with period $\ell\in\N_+$.
For a distribution $\phi$ and test function $f$ we denote by $\phi(f)\equiv\langle\phi,f\rangle$
the usual pairing that generalises the integral over $\R^3$. 
We denote by $\cC^\alpha(\T^3_\ell)$ the standard H{\"o}lder--Besov space of regularity $\alpha\in\R$,
by $C^\infty_{\mathrm c}(\R^3)$ the space of smooth compactly supported functions and by
$C^2_{\mathrm{b}}(\R^k)$ the space of bounded twice differentiable functions with bounded
derivatives up to order $2$.
We say that a function $H \colon \dD'(\R^3) \to \R$ is cylindrical if it is of the form
$H(\phi) = h(\phi(f_1),\ldots,\phi(f_k))$ for some $k \ge 1$, some test functions
$f_i \in C^\infty(\R^3)$, and some $h \in C^2_{\mathrm b}(\R^k)$.
We write $\mathrm{D}H$ for its $L^2(\T^3_\ell)$ gradient, namely
$\mathrm{D}H(\phi) = \sum_j (\d_j h)(\phi(f_1),\ldots,\phi(f_k)) f_{j,\ell}$, where $f_{j,\ell}$ denotes the periodisation of $f_j$ with period $\ell$.

Let $\lL \eqdef \d_t - \Laplace+1$ and $\lambda>0$. Given a cylindrical functional $H$ on $\dD'(\R^3)$, 
we denote by $\PHImod(\phi;\bigcdot)$ the solution to the stochastic PDE
\begin{equ}
\label{e:phi4_mod}
\lL \Phi_{\eps,\ell} =
\xi_{\eps,\ell}
- \lambda \Phi_{\eps,\ell}^3
+ C_{\eps, \ell}(\lambda) \Phi_{\eps,\ell}
+ \mathrm{D}H(\Phi_{\eps,\ell})\;,
\qquad
\Phi_{\eps,\ell}(0) = \phi\;,
\end{equ}
where $\xi_{\eps,\ell}$ is the periodisation of space-time white noise on $\R\times\R^3$, mollified in space at scale $\eps\in(0,1]$ and $C_{\eps, \ell}(\lambda)$ is a renormalisation constant. The precise definitions of $\xi_{\eps,\ell}$ and $C_{\eps, \ell}(\lambda)$ can be found in Definition~\ref{def:smoothed_noise} below. We included an extra drift term $\mathrm{D}H$ in the equation 
in anticipation of the proof of correlation decay.

Note that when $H = 0$, equation~\eqref{e:phi4_mod} reduces to the standard $\Phi^4_3$ equation, and in this case we denote the solution by $\Phi_{\eps, \ell}(\phi;\bigcdot)$. For fixed size of the torus, the existence of the $\eps \rightarrow 0$ limit of the solution, as well as its properties, were studied in \cite{Hai14,MW17b,HM18,CC18,HS22}, etc. We summarise the relevant properties in the finite volume setting in the following theorem.

\begin{defn}\label{def:kappa}
	We denote by $\bar\kappa=\frac{1}{10}$ a small parameter and let $\kappa=\bar\kappa^4$.
\end{defn}

\begin{thm}\label{thm:dynamic_on_torus}
Fix $\lambda>0$, $\ell\in\N_+$ and a cylindrical functional $H$ on $\dD'(\R^3)$. The dynamic governed by~\eqref{e:phi4_mod} converges globally in time in probability as $\eps\searrow0$. More precisely, there exists a~continuous random map
\begin{equs}[e:solMap]
	\cC^{-\frac12-\kappa}(\T^3_\ell) \ni \phi
	\mapsto \Phi_\ell^{(H)}(\phi;\bigcdot)\in
	C(\R_\geq,\cC^{-\frac12-\kappa}(\T^3_\ell))
\end{equs}
such that
\begin{equ}
\lim_{\eps\searrow0}\sup_{t\in[0,T]} \|\PHImod(\phi;t,\bigcdot)-\Phi_\ell^{(H)}(\phi;t,\bigcdot)\|_{\cC^{-\f12-\kappa}(\T^3_\ell)}=0
\end{equ}
for every $T>0$ and $\phi\in \cC^{-\frac12-\kappa}(\T^3_\ell)$, with convergence taking place in probability. The limiting dynamic $\Phi_\ell^{(H)}$ is exponentially ergodic with unique invariant measure $\mu_\ell^{(H)}$.
Moreover, writing $\mu_{\ell}$ as a shorthand for $\mu_{\ell}^{(0)}$, we have
\begin{equ}[eq:tildemu]
 \mu_{\ell}^{(H)}(\md\phi) \propto e^{2H(\phi)}\,\mu_{\ell}(\md\phi) \;.
\end{equ}
\end{thm}
\begin{proof}
The global in time convergence of $\Phi_{\eps,\ell}$ as $\eps\searrow0$ was proved in~\cite{MW17b}. 
The local in time convergence of $\PHImod$ can be proved by modifying~\cite{Hai14}, as 
in~\cite{HS22}, to account for the additional non-local term. Global convergence is a 
consequence of the ``coming down from infinity'' estimate stated in Lemma~\ref{le:psi} below. 
Ergodicity of the 
dynamic for $\Phi_\ell$ follows from the proof of~\cite[Corollary~1.13]{HS22b}, which relies on 
the strong Feller property shown in \cite{HM18b}. Ergodicity of the dynamic $\Phi_\ell^{(H)}$
can be shown by adapting the argument in the proof of~\cite[Theorem~2.2]{HS22}.
The identity~\eqref{eq:tildemu} follows from~\cite[Corollary~2.5]{HS22}.
\end{proof}

\begin{defn}\label{def:weights_w}
	For $x\in\R^3$, set $\bracket{x}\eqdef(1+|x|_2^2)^{1/2}$, 
	where $|x|_2$ is the Euclidean norm.
	Let $w=\langle \bigcdot\rangle^{-\kappa}\in C(\R^3)$ be a fixed weight decaying polynomially at infinity. We denote by $\cC^\alpha(w)$ the weighted H{\"o}lder--Besov space, see also Definition~\ref{defn:weighted_Besov}.
\end{defn}

\begin{thm}\label{thm:tightness}
The sequence of measures $(\mu_\ell)_{\ell\in\N_+}$ on $\cC^{-\f12-\kappa}(w)$ is tight.
\end{thm}
\begin{proof}
This result has been well-known since \cite{FO76,MS76} and follows 
in particular from the ``space-time localisation'' estimate~\cite{MW20} stated in Appendix~\ref{sec:spacetime_localization}.
\end{proof}

\subsection{Infinite volume dynamic}

We now state our main result concerning the construction of the $\Phi^4_3$ dynamic on $\R^3$. While a solution theory for the infinite-volume $\Phi^4_3$ equation was developed in~\cite{GH19}, it was established under highly restrictive assumptions on the initial data. (It needs to be a Hölder continuous
perturbation of the stationary solution to the massive stochastic heat equation.) In particular, 
it does not provide a solution map that defines a Feller Markov process on a natural state space, such as 
a weighted H{\"o}lder–Besov space. One of the key contribution of the present work is to establish that the 
$\Phi^4_3$ dynamic on $\R^3$ indeed defines a Markov process with the Feller property on the space 
$\cC^{-\frac12-\kappa}(w)$.

\begin{thm}\label{thm:global_solution}
Fix arbitrary $\lambda>0$. There exists a continuous random map
\begin{equs}[e:solMap_inf]
	\cC^{-\frac12-\kappa}(w) \ni \phi
\mapsto \Phi(\phi;\bigcdot)\in
	C(\R_\geq,\cC^{-\frac12-\kappa}(w^{4}))
	\cap
	C(\R_>,\cC^{-\frac12-\frac{\kappa}{2}}(w^{\frac12}))
\end{equs} 
and a random variable $R\geq 0$ with finite moments of all orders such that
\begin{equ}[e:bound]
	t^{\frac12}\, \|\Phi(\phi;t,\bigcdot)\|_{\cC^{-\frac12-\frac\kappa2}(w^{\frac12})} \leq R
\end{equ}
for all $t\in(0,1]$, $\phi\in \cC^{-\frac12-\kappa}(w)$ and such that
\begin{equ}
\label{e:Phi_convergence} 
\lim_{\ell\to\infty}\lim_{\eps\searrow0}\|\Phi_{\eps,\ell}(\phi_{\eps,\ell};t,\bigcdot)-\Phi(\phi;t,\bigcdot)\|_{\cC^{-\frac12-\frac\kappa2}(w^{\frac12})} =0
\end{equ}
for all $t>0$, $\phi\in\cC^{-\frac12-\kappa}(w)$ and $\phi_{\eps,\ell}\in C(\T_\ell^3)$ such that $\lim_{\ell\to\infty}\lim_{\eps\searrow0}\phi_{\eps,\ell}=\phi$ in $\cC^{-\frac12-\kappa}(w)$,
with convergence taking place in probability. 
Moreover, $\Phi(\phi;\bigcdot)=\rR U$, where $U$ is the unique singular modelled distribution solving
\begin{equ}\label{eq:dynamic_modelled}
	U = \kK (\one_>  U^3 +{\color{blue}\Xi}) + K(\phi-\ou(0))
\end{equ}
on $\R_>\times\R^3$. Here $\rR$ and  $\kK$ are the reconstruction and abstract integration operators, ${\color{blue}\Xi}$ is the symbol representing the noise, $K\phi$ denotes the unique solution of 
the massive heat equation that coincides with $\phi$ at time zero, and $\ou$ is the stationary solution of the massive stochastic heat equation.
We also have
\begin{equ}
\label{e:Euc_invariance}
	 \Phi(\varrho\cdot\phi;t,\bigcdot)\stackrel{\mathrm{law}}{=}\varrho\cdot\Phi(\phi;t,\bigcdot)
\end{equ}
for all $t>0$, $\phi\in \cC^{-\frac12-\kappa}(w)$ and all elements $\varrho$ of the Euclidean group $\R^3 \rtimes O(3)$, where $\varrho\cdot f$ denotes the standard action of $\varrho$ on $f\in\dD'(\R^3)$.
\end{thm}
\begin{rmk}
	Unfortunately, we are not able to establish continuity of the map $\R_\geq\ni t\mapsto \Phi(\phi;t,\bigcdot)\in \cC^{-\frac12-\kappa}(w)$ at $t=0$, a common requirement in the theory of random dynamical systems. The reason for this is that, for initial data in $\cC^{-\frac12-\kappa}(w)$, we are only
	able to obtain uniform bounds on the solution near $t=0$ in the larger space 
	$\cC^{-\frac12-\kappa}(w^3)$.
\end{rmk}

To prove Theorem~\ref{thm:global_solution}, our approach builds on the space-time localisation bounds for solutions of the $\Phi^4_3$ model established in~\cite{MW20}, which impose no constraints on the initial condition but yield a singularity of order $t^{-\frac{1}{2}}$ at the initial time $t=0$. While this result provides a bound on the cubic nonlinearity of the solution, it does so with a~non-integrable blow-up at the initial time hypersurface, making it unsuitable for directly constructing a~mild solution in the space of modelled distributions. 

Therefore, the main difficulty we have to overcome is to obtain improved control of the behaviour of the solution near the initial time. Our strategy is quite similar to the strategy used in~\cite{MW17a,RolandHendrik} to establish a solution theory for the dynamical $\Phi^4$ model on $\R_\geq \times \R^2$, though the extension to three dimensions presents many challenges due to the more singular nature of the equation.

We apply the space-time localisation estimate to the solution with the initial data contribution subtracted, and incorporate this contribution into the definitions of the trees that appear in the estimate. The estimates for such trees are presented in Appendix~\ref{sec:stochastic_estimates}, which might be of independent interest. Since the resulting equation has zero initial condition, it can be extended to negative times, yielding a bound without blow-up at time zero. After reintroducing the initial data, we obtain an a priori bound with an improved blow-up rate \dash from $t^{-\frac{1}{2}}$ to $t^{-\frac{1}{4} - \frac{\kappa}{2}}$ (see Lemma~\ref{le:apriori_bound_Linfty}). This ensures that the cubic nonlinearity remains integrable in time. 

As a result, every possible subsequential limit can be identified as a (singular) modelled distribution solving the abstract $\Phi^4_3$ equation in infinite volume. To establish uniqueness of the limit, we observe that the difference of two solutions satisfies an equation of the same form as the Parabolic Anderson Model. Uniqueness then follows by adapting the argument from~\cite{HL18}. Consequently, the finite-volume equations converge to a unique limit. The proof of Theorem~\ref{thm:global_solution} is given in Section~\ref{sec:proof_markov_process_construction}.

\begin{defn}
	For $\lambda>0$, we write $(\pP_t)_{t\in\R_\geq}$ for the Markov semigroup on $\cC^{-\frac12-\kappa}(w)$ associated to the dynamical $\Phi^4_3$ model on $\R^3$ constructed in the above theorem.
\end{defn}

\begin{rmk}
	The  continuity of the solution map \eqref{e:solMap_inf} ensures that $\pP_t$ satisfies the Feller property. 
	By combining \eqref{e:Phi_convergence} with Theorem~\ref{thm:tightness}, we deduce that all subsequential limits of $(\mu_\ell)_{\ell \in \N_+}$ are invariant under $\pP_t$.
	Furthermore, it follows from the intrinsic characterisation \eqref{eq:dynamic_modelled}
	that the Markov semigroup $\pP_t$ is covariant under the action of the Euclidean isometry group.
\end{rmk}

\subsection{Linearised equation}

To prove uniqueness of the invariant measure, it is natural to consider the difference between two solutions of \eqref{e:phi4_mod} started from (potentially different) invariant measures. As we will see, controlling this difference reduces to understanding the long‑time behaviour of the linearisation of \eqref{e:phi4_mod}.
\begin{defn}
	Suppose that $S\in C_{\mathrm{b}}(\R\times\T_\ell^3)$ is a bounded, adapted in time stochastic process and $\Phi_{\eps,\ell}$ solves
\begin{equ}\label{eq:phi4_S} 
	\lL \Phi_{\eps,\ell} =
	\xi_{\eps,\ell}
	+ S
	- \lambda \Phi_{\eps,\ell}^3
	+ C_{\eps, \ell}(\lambda) \Phi_{\eps,\ell}\;, 
	\qquad 
	\Phi_{\eps,\ell}(0) = \phi \in C(\T_\ell^3)
	\;.
\end{equ}
Given a solution $\Phi_{\eps,\ell}$ of the above equation, and for any $0 \leq s \leq t < \infty$, we define a~random operator
\begin{equ} 
	J_{\eps,\ell}(s, t)\equiv J_{\eps,\ell}[\Phi_{\eps,\ell}](s, t)\,:\,D(s)\mapsto D(t) \;,
\end{equ}
where $D$ solves the linearised equation
\begin{equ}
\label{e:linearised}
	\bigl(\lL + 3\lambda\Phi_{\eps,\ell}^2 -C_{\eps, \ell}(\lambda) \bigr)D = 0
\end{equ}
in the time interval $[s,t]$, with $D(s) \in C(\T_\ell^3)$.
\end{defn}
\begin{rmk}\label{rmk:phi4_S_H}
	Observe that if $\Phi_{\eps,\ell}$ is a solution to~\eqref{e:phi4_mod} and $S = \mathrm{D} H(\Phi_{\eps,\ell})$, then $\Phi_{\eps,\ell}$ satisfies~\eqref{eq:phi4_S} as well.
\end{rmk}

\begin{defn}
	We write $L^p$ and $L^p(\T_\ell^3)$ for the standard $L^p$ spaces over $\R^3$ and $\T_\ell^3$. Given a non-negative weight $w$, the weighted $L^p(w)$-norm of a function $f$ over $\R^3$ coincides with the standard $L^p$-norm of $w f$. We define $\rho\eqdef\langle \bigcdot\rangle^{-4}\in C(\R^3)$. Let $\chi\in C^\infty(\R^3)$ 
	be a positive function such that $\chi=1$ on $[-1/3,1/3]^3$, $\supp\chi\subset[-1,1]^3$ 
	and the periodisation of $\chi$ with period $1$ coincides with the constant function $1$. We denote by $\langle \bigcdot\rangle_\ell\in C^2(\T_\ell^3)$ the periodisation with period $\ell$ of the function $\langle\bigcdot\rangle\chi(\bigcdot/\ell)$. We note that $\langle x\rangle_\ell \geq |x|$ for all $x\in\R^3$ such that $|x|\leq\ell/3$, where $|x|$ denotes the supremum norm. Moreover, $|\nabla\langle\bigcdot\rangle_\ell|,|\Delta\langle\bigcdot\rangle_\ell|\lesssim 1$ uniformly in $\ell\in\N_+$.
\end{defn}
\begin{defn}\label{def:filtration}
	For $t\in\R$ let $\fF_t$ be the $\sigma$-algebra generated by
	\begin{equ}
		\{\xi(f)\,|\, f\in L^2(\R^{1+3}),\,\supp f\subset (-\infty,t]\times\R^3\}
	\end{equ}
	augmented with the events of probability zero.
\end{defn}

To establish both uniqueness and exponential decay of correlations for the invariant measure, our key ingredient is the following bound on the solution map $J_{\varepsilon, \ell}$ associated with the linearised equation.
\begin{thm}\label{thm:linearised} 
Fix $p\ge1$ and let $\rho_{\ell,\gamma,\fc}=\langle\fc\bigcdot\rangle^{-4} \exp(\gamma\langle \bigcdot\rangle_\ell)\in C_0(\R^3)$. Then, there exists $\lambda_\star> 0$ such that
\begin{equ}[e:linearised_1]
	\EE\|J_{\eps,\ell}(s,t)v\|^p_{L^p(\rho_{\ell,\gamma,\fc})} 
	\lesssim \exp(-p\,(t-s)/3)\,
	\EE\|v\|^p_{L^p(\rho_{\ell,\gamma,\fc})}
\end{equ}
uniformly over $\lambda\in[0,\lambda_\star]$, $\eps\in(0,1]$, $\ell\in\N_+$, $0\leq s\leq t <\infty$, $\fc\in(0,1]$, $\gamma\in[0,\lambda_\star]$, $\fF_s$-measurable $v\in C(\T_\ell^3)$ and $\Phi_{\eps,\ell}$ solving~\eqref{eq:phi4_S} with an adapted and continuous $S$ in a unit ball of $L^\infty(\R_\geq\times\T_\ell^3)$ and arbitrary initial data.
\end{thm}

\begin{rmk}\label{rmk:lambda_p}
	The constant $\lambda_\star$ in the above theorem cannot be fixed independently of $p\geq1$ (although we do of course believe this to be the case). The same applies to  all results stated below.
\end{rmk}
\begin{rmk}\label{rmk:torus_linearisation}
	For every $p\geq 1$ there exists $C>0$ such that
	\begin{equ}
		C^{-1}\,\bigl\|\exp(\gamma\langle \bigcdot\rangle_\ell)v\bigr\|_{L^p(\T_\ell^3)}
		\leq
		\|v\|_{L^p(\rho_{\ell,\gamma,1/\ell})}
		\leq 
		C\,\bigl\|\exp(\gamma\langle \bigcdot\rangle_\ell)v\bigr\|_{L^p(\T_\ell^3)}
	\end{equ}
	for all $v\in L^p(\T_\ell^3)$ and $\ell\in\N_+$. In particular, applying the above estimate with $\gamma=0$ we conclude that $\EE\|J_{\eps,\ell}(s,t)v\|^p_{L^p(\T_\ell^3)} \lesssim \exp(-p\,(t-s)/3)\,\EE\|v\|^p_{L^p(\T_\ell^3)}$ uniformly over $\ell \ge 1$.
\end{rmk}
\begin{rmk}
	\label{rmk:difference_linearisation}
	Let $0\leq s\leq t<\infty$. If $\Phi^{(j)}_{\eps,\ell}$, $j\in\{0,1\}$, solve~\eqref{eq:phi4_S} on the time interval $[s,t]$ with $S=0$ and respective initial data $\phi^{(j)}_{\eps,\ell}$ at time $s$, then
	\begin{equ}
		(\Phi^{(1)}_{\eps,\ell}-\Phi^{(0)}_{\eps,\ell})(t)
		=
		\int_0^1 J^{(u)}_{\eps,\ell}(s, t)\, (\Phi^{(1)}_{\eps,\ell}-\Phi^{(0)}_{\eps,\ell})(s)\, \md u\,,
	\end{equ}
	where $J^{(u)}_{\eps,\ell}=J_{\eps,\ell}[\Phi^{(u)}_{\eps,\ell}]$ and $\Phi^{(u)}_{\eps,\ell}$ denotes the solution to~\eqref{eq:phi4_S} on $[s,t]$ with $S=0$ and initial data $\Phi_{\eps,\ell}^{(u)}(s)=u\phi^{(1)}+(1-u)\phi^{(0)}$. Hence,
	\begin{equ}
		\EE\|\Phi^{(1)}_{\eps,\ell}(t)-\Phi^{(0)}_{\eps,\ell}(t)\|^p_{L^p(\rho)}
		\lesssim
		\exp(-p\,(t-s)/3)\,\EE\|\Phi^{(1)}_{\eps,\ell}(s)-\Phi^{(0)}_{\eps,\ell}(s)\|^p_{L^p(\rho)}\,.
	\end{equ}
	Using the space-time localisation bound~\cite{MW20} (see Lemma~\ref{le:MW20}) and the decay property of $\rho$ we obtain immediately that
	\begin{equ}
		\EE\|\Phi^{(1)}_{\eps,\ell}(1)-\Phi^{(0)}_{\eps,\ell}(1)\|^p_{L^p(\rho)} \lesssim 1
	\end{equ}
	uniformly over the initial conditions. Combining the above estimates we arrive at
	\begin{equ}
		\EE\|\Phi^{(1)}_{\eps,\ell}(t)-\Phi^{(0)}_{\eps,\ell}(t)\|^p_{L^p(\rho)}
		\lesssim
		\exp(-p\,t/3)
	\end{equ}
	uniformly over $t\geq 1$ and all initial data, which almost immediately implies uniqueness of the invariant measure.
\end{rmk}
\begin{rmk}\label{rmk:decay_linearisation}
	Let $\Phi^{(j)}_{\eps,\ell}$, $j\in\{0,1\}$, be solutions of~\eqref{eq:phi4_S} with vanishing initial data, where in the case $j = 1$ we take an arbitrary source term $S$ that is adapted, continuous and satisfies $\sup_{t \ge 0} \sup_{x \in \mathbb R^3} |S(t,x)| \le 1$ and $\bigcup_{t \geq 0} \supp S(t) \subset [-K,K]^3$, while for $j = 0$ we take $S \equiv 0$. Then
	\begin{equ}
		(\Phi^{(1)}_{\eps,\ell}-\Phi^{(0)}_{\eps,\ell})(t) = \int_0^1\int_0^t J^{(u)}_{\eps,\ell}(s, t)\, S(s) \, \md s \md u\,,
	\end{equ}
	where $J^{(u)}_{\eps,\ell}=J_{\eps,\ell}[\Phi^{(u)}_{\eps,\ell}]$ and $\Phi^{(u)}_{\eps,\ell}$ denotes the solution to~\eqref{eq:phi4_S} on $[0,t]$ with source $uS$ and zero initial data. Hence, by Remarks~\ref{rmk:phi4_S_H} and~\ref{rmk:torus_linearisation} and the Minkowski inequality we have
	\begin{equ}
		\EE\bigl\|\exp(\gamma\langle \bigcdot\rangle_\ell)\,(\Phi^{(H)}_{\eps,\ell}-\Phi_{\eps,\ell})(t)\bigr\|^p_{L^p(\T_\ell^3)}
		\lesssim
		1
	\end{equ}
	uniformly over $t\geq 0$, where $\PHI,\PHImod$ solve~\eqref{e:phi4_mod} with zero and some fixed nonzero cylindrical functional $H$ such that $\|\mathrm D H(\phi)\|_{L^\infty}\leq1$ for all $\phi$, respectively. This will be instrumental in proving exponential decay of correlation of the invariant measure. 
\end{rmk}

We end this section by outlining the main ideas behind the proof of Theorem~\ref{thm:linearised}.
The starting point is the now-standard Da Prato--Debussche decomposition \cite{DPD03,CC18}, namely
$\Phi = \ou - \lambda \outz + \Psi,$
which allows us to cancel the most irregular terms (see Definition~\ref{def:enhanced_noise} for the stochastic objects $\ou,\outz$).
The control of the remainder $\Psi$ is based on a stopping time argument inspired by \cite{KT22} combined with a ``coming down from infinity'' estimate (Lemma~\ref{le:psi}).
While Theorem~\ref{thm:linearised} can be viewed as an extension of the results of \cite{KT22} from $\T^2$ to $\R^3$, the low regularity in three dimensions prevents a direct adaptation of their energy estimates.
Moreover, working in infinite volume requires us to handle the growth of the noise at spatial infinity.

To address these issues, we apply an exponential transformation (Lemma~\ref{lem:exp_transform_mod}) and use a fixed-point argument with time-dependent (stretched) exponential weights, originally introduced in \cite{HL15}. 
After the transformation, 
a key obstacle is the term $V_s^{(2)}(t)$ (see Lemma~\ref{lem:exp_transform_mod}), whose norm admits only a bound of order $(t-s)^{-1}$, non-integrable as $t\searrow s$.
This issue is resolved through a comparison argument (Lemma~\ref{lem:comparison_mod}) that takes advantage of the positivity of $V_s^{(2)}(t)$.
With this difficulty removed, the approach of \cite{HL15} can be applied, leading to Proposition~\ref{pr:deterministic_main}, which yields an estimate that, however, involves different weights on the two sides of the inequality.

We then follow the idea in \cite{KT22} to iterate the estimate up to time one using the strong Markov property, with a control on the number of iterations, and then take expectations.
By spatial stationarity of the enhanced noise, the same argument works when centred at any point $z \in \R^3$, giving an analogous bound around $z$.
Averaging over $z$ allows us to obtain an estimate with identical weights on both sides.
Finally, one iterates the bound valid up to time one and uses the Markov property to obtain the desired long-time estimate.
The proof of Theorem~\ref{thm:linearised} is given in Section~\ref{sec:proof_ergodicity}.

\subsection{Applications}
With Remark~\ref{rmk:difference_linearisation} and Theorem~\ref{thm:linearised}, it is not hard to show that the Markov semigroup $(\pP_t)_{t \in \R_{\geq}}$ has a unique invariant measure $\mu$ when $\lambda>0$ is small enough. Various properties of $\mu$ can also be derived using the dynamic. In particular, we prove that $\mu$ satisfies all of the Osterwalder--Schrader axioms~\cite{OS75},~\cite[Section~6.1]{GJ87}.
\begin{thm}\label{thm:uniqueness}
There exists $\lambda_\star\in(0,1]$ such that for all $\lambda\in(0,\lambda_\star]$ the Markov semigroup $(\pP_t)_{t\in\R_\geq}$ admits a unique invariant measure $\mu$ and one has $\mu = \lim_{\ell\to \infty}\mu_\ell$, where $\mu_\ell$ is the invariant measure of the dynamic on $\T_\ell^3$. Furthermore, $\mu$ has the following properties:
\begin{enumerate}
	\item $\mu$ invariant under Euclidean isometries. 
	\item $\mu$ is reflection positive.
	\item For every $f\in C^\infty_{\mathrm c}(\R^3)$ there exists $\beta>0$ such that
	\begin{equ}
	\label{e:subgaussian_tail}
		\int \exp\bigl(\beta\langle\phi,f\rangle^4\bigr)\,\mu(\md\phi)<\infty \;.
	\end{equ}
	\item For all $F,G\in C^2_{\mathrm{b}}(\R)$, $f,g\in C^\infty_{\mathrm{c}}(\R^3)$ we have 
	\begin{equ}
	\label{e:correlation_decay}
		\Cov_\mu\big(F(\langle\bigcdot,f\rangle),G(\langle\bigcdot,g_L\rangle)\big)\lesssim \exp(-\gamma|L|)
	\end{equ}
	uniformly over $L \in \R^3$, where $g_L\eqdef g(\bigcdot-L)$. 
\end{enumerate}

\end{thm}
\begin{proof}
We start by showing that if $\mu$ is any accumulation point of the $\mu_\ell$, then
it is invariant under $(\pP_t)_{t\in\R_\geq}$. Let $(\ell_n)_{n\in\N_+}$ be such that $\lim_{n\to\infty}\mu_{\ell_n}=\mu$ in the sense of weak convergence of measures on $\cC^{-\f12-\kappa}(w)$. Skorokhod's representation theorem yields a probability space and random variables $\xi$, $\phi$, $(\phi_n)_{n\in\N_+}$ such that: 
\begin{enumerate}
\item the law of $\phi$ is $\mu$ and the law of $\phi_n$ is $\mu_{\ell_n}$ for all $n\in\N_+$,
\item $(\phi_n)_{n\in\N_+}$ converges almost surely to $\phi$ in $\cC^{-\f12-\kappa}(w)$,
\item $\xi$ is a space-time white noise independent of $\phi$ and $(\phi_n)_{n\in\N_+}$.
\end{enumerate}
Let $F$ be a bounded continuous functional on $\cC^{-\f12-\kappa}(w)\supset\cC^{-\f12-\kappa}(\T_\ell^3)$. By the definition of $\pP_t$ and Theorem~\ref{thm:global_solution} we have
\begin{equ}
\mu(\pP_tF)
=
\EE\big(F(\Phi(\phi;t,\bigcdot))\big)
=
\lim_{n\to\infty}
\EE\big(F(\Phi_{\ell_n}(\phi_n;t,\bigcdot))\big)
\end{equ}
for all $t\geq0$. Since $\mu_\ell$ is invariant under $\pP^{(\ell)}_t$, the process $t\mapsto \Phi_\ell(\phi_\ell;t,\bigcdot)$ is stationary. Hence, $\EE\big(F(\Phi_{\ell_n}(\phi_n;t,\bigcdot))\big)=\EE\big(F(\Phi_{\ell_n}(\phi_n;0,\bigcdot))\big)$ and $\mu(\pP_tF)=\mu(F)$ for all $t\geq0$, 
showing that $\mu$ is indeed invariant.

To prove that the invariant measure is unique, assume that $\mu,\tilde\mu$ are both invariant
and let $F$ be a bounded Lipschitz continuous functional on $\bB_{2,2}^{-\frac12-2\kappa}(\rho)\supset \cC^{-\f12-\kappa}(w)$. Then, by Theorems~\ref{thm:global_solution} and~\ref{thm:linearised} 
(see in particular Remark~\ref{rmk:difference_linearisation}), we obtain $|\mu(F)-\tilde\mu(F)|\lesssim\exp(-t/3)$ uniformly over $t\ge 1$, yielding $\mu=\tilde\mu$.

For the properties of $\mu$, since $(\pP_t)_{t\in\R_\geq}$ is covariant under Euclidean isometries, $\mu$ is invariant under these transformations. The bound~\eqref{e:subgaussian_tail} was proved in~\cite[Theorem~1.1]{HS22}. By Lemma~\ref{lem:measure_torus} and the convergence of $\mu_\ell$ to $\mu$, $\mu$ is also reflection positive. It remains to prove \eqref{e:correlation_decay}. It suffices to show that
\begin{equ}
	\lvert\Cov_{\mu_{\ell}}\big(e^{2F(\langle\bigcdot,f\rangle)},e^{2G(\langle\bigcdot,g_L\rangle)}\big)\rvert
	\lesssim\exp(-\gamma|L|)
\end{equ}
uniformly over $\ell\in\N_+$ and $L\in\R$ for arbitrary fixed $F,G\in C^2_{\mathrm{b}}(\R)$, $f,g\in C^\infty_{\mathrm{c}}(\R^3)$ such that $\|F'\|_{L^\infty} \|f_\ell\|_{L^\infty}\leq1$, 
where $f_\ell$ denotes the periodisation of $f$.
Here, with a~slight abuse of notation, we write $g_L := g_{(L,0,0)}$. Let $H(\phi)=F(\phi(f))=F(\langle\phi,f\rangle)$. Then
$\mathrm{D}H(\phi)=F'(\phi(f))f_\ell$
and $\|\mathrm D H(\phi)\|_{L^\infty}\leq1$ for all $\phi$. 
Using the identity~\eqref{eq:tildemu} we obtain
\begin{equs}
	\Cov_{\mu_\ell}&\big(e^{2F(\langle\bigcdot,f\rangle)},e^{2G(\langle\bigcdot,g_L\rangle)}\big)
	\\
	&=
	\int e^{2F(\phi(f))}e^{2G(\phi(g))}\,\mu_\ell(\md \phi)
	- \int e^{2F(\phi(f))}\,\mu_\ell(\md \phi)
	\int e^{2G(\phi(g_L))}\,\mu_\ell(\md \phi)
	\\
	&=
	\int e^{2F(\phi(f))}\,\mu_\ell(\md \phi)\,
	\int e^{2G(\phi(g_L))}\,\Big(\mu^{(H)}_\ell(\md\phi)-\mu_\ell(\md \phi)\Big)
	\\
	&=
	\lim_{t\to\infty}
	\EE(e^{2F(\Phi_\ell(t,f))})\,
	\EE(e^{2G(\Phi^{(H)}_\ell(t,g_L))}-e^{2G(\Phi_\ell(t,g_L))})
	\\
	&=
	\lim_{t\to\infty}\lim_{\eps\searrow0}
	\EE(e^{2F(\PHI(t,f))})\,
	\EE(e^{2G(\PHImod(t,g_L))}-e^{2G(\PHI(t,g_L))})\;.
\end{equs}
The penultimate equality follows from ergodicity in finite volume and the last one follows from the fact that the dynamic $\Phi_\ell,\Phi^{(H)}_\ell$ can be approximated by $\PHI,\PHImod$ solving~\eqref{e:phi4_mod} with zero and nonzero $H$, respectively. Now suppose that $\PHI$ and $\PHImod$ vanish at time zero and choose $N\in\N_+$ such that $\supp f,g\subset (-N/2,N/2)^3$. Then
\begin{equs}
	\lvert\Cov&_{\mu_{\ell}}\big(e^{2F(\langle\bigcdot,f\rangle)},e^{2G(\langle\bigcdot,g_{L+N}\rangle)}\big)\rvert
	\\
	&\lesssim
	\sup_{t\geq 0}\sup_{\eps\in(0,1]}
	\EE|(\PHImod-\PHI)(t,g_{L+N})|
	\\
	&\lesssim
	\sup_{t\geq 0}\sup_{\eps\in(0,1]}
	\lVert\exp(-\gamma\langle\bigcdot\rangle_\ell)\,g_{L+N}\rVert_{L^2(\R^3)}\,
	\EE\lVert\exp(\gamma\langle\bigcdot\rangle_\ell)\,(\PHImod-\PHI)(t)\rVert_{L^2(\T_\ell^3)}
	\\
	&\lesssim
	\exp(-\gamma|L|)
	\sup_{t\geq 0}\sup_{\eps\in(0,1]}
	\EE\lVert\exp(\gamma\langle\bigcdot\rangle_\ell)\,(\PHImod-\PHI)(t)\rVert_{L^2(\T_\ell^3)}
	\\
	&\lesssim\exp(-\gamma|L|)
\end{equs}
uniformly in $L\in[1,\ell/3]$ and $\ell \geq N$, where the last step is a consequence of Remark~\ref{rmk:decay_linearisation}.
\end{proof}

\begin{lem}\label{lem:measure_torus}
Let $\lambda>0$, $\ell\in\N_+$ and $\eps = 2^{-n}$ for some $n \in \N_+$. We introduce a map $\iota_{\epsilon,\ell}\,:\, \R^{\T^3_{\epsilon,\ell}} \to \dD'(\T_\ell^3)$ by setting
\begin{equ}
 \langle \iota_\epsilon f, \varphi \rangle \eqdef \sum_{y \in \T^3_{\epsilon,\ell}} f(y)\int_{\square_\epsilon(y)} \varphi(z)\, \md z\;,
 \qquad
 \varphi\in C^\infty(\T_\ell^3)\;,
\end{equ}
where $\square_\epsilon(y) \eqdef \{z \in \R^3\,|\, \|z-y\|_\infty \leq \epsilon/2\}$.
Recall the definition~\eqref{e:muepsell} of the $\Phi^4_3$ measure $\hat\mu_{\epsilon,\ell}$ on $\R^{\T^3_{\epsilon,\ell}}$.
The sequence of measures $(\iota_\epsilon\sharp \hat\mu_{\epsilon,\ell})$ converges weakly on  $\cC^{-\f12-\kappa}(\T_\ell^3)$ as $\varepsilon\searrow0$ to the invariant measure $\mu_\ell$ of the dynamic on $\T_\ell^3$. Moreover, $\mu_\ell$ is reflection positive in the sense of~\cite[Section~6.1]{GJ87}.
\end{lem}
\begin{proof}
The claim about the convergence follows from~\cite[Proposition~7.8]{HM18} or~\cite[Theorem~2.2]{HS22},
but has been known since \cite{MR418721}. To show that $\mu_\ell$ is reflection positive one uses the argument from the proof of Proposition~5.3 in~\cite{GH21}.
\end{proof}

By standard arguments, see~\cite{KT22} for example, Theorem~\ref{thm:linearised} implies the spectral gap inequality for the $\Phi^4_3$ measure stated in the corollary below. Note that this result is not new, as it follows from the log-Sobolev inequality for the $\Phi^4_3$ measure proved in \cite{RolandLogSob}.
\begin{cor}\label{thm:spectral_gap}
There exist $\lambda_\star\in(0,1]$ and $C>0$ such that for all $\lambda\in(0,\lambda_\star]$ and all cylindrical functions $F$ on $\dD'(\R^3)$, the $\Phi^4_3$ measure $\mu$ on $\R^3$ satisfies the following spectral gap inequality 
\begin{equ}
(\mu(F^2)) - \big(\mu(F)\big)^2 \leq C\, \mu\big(\|\mathrm{D}F \|_{L^2(\R^3)}^2\big)\;,
\end{equ}
where $\mathrm{D}F$ denotes the $L^2(\R^3)$ gradient of $F$.
\end{cor}
\begin{proof}
It suffices to show that
\begin{equ}
	(\mu_\ell(F^2)) - \big(\mu_\ell(F)\big)^2 \leq C\, \mu_\ell\big(\|\mathrm{D}F \|_{L^2(\T_\ell^3)}^2\big)\;,
\end{equ}
where $\mu_\ell$ is the invariant measure of the dynamic on $\T_\ell^3$, and use the weak convergence $\mu=\lim_{\ell\to\infty}\mu_\ell$. By Remark~\ref{rmk:torus_linearisation} the solution map of the linearised equation~\eqref{e:linearised} satisfies
\begin{equ}
\EE \|J_{\eps,\ell}(0,t)v\|^2_{L^2(\T^3_\ell)} 
\leq \frac C3 \exp(-2t/3)\,
\|v\|^2_{L^2(\T^3_\ell)}
\end{equ}
for deterministic initial conditions $v$. Consequently, by the argument presented in~\cite[Section~3]{KT22} we have
\begin{equ}
 \|{\rm D}(\pP^{(\ell)}_t F(\phi))\|^2_{L^2(\T^3_\ell)}
 \leq
 \frac C3 \exp(-2t/3)\,
 \pP^{(\ell)}_t\|({\rm D} F)(\phi)\|^2_{L^2(\T^3_\ell)}
\end{equ}
for all cylindrical functions $F$ on $\T^3_\ell$ and $\phi\in \cC^{-\f12-\kappa}(\T^3_\ell)$. Here $\pP^{(\ell)}_t $ is defined to be the Markov semigroup corresponding to $\Phi_{\ell}$. By Proposition 4.2 in~\cite{KT22} we have
\begin{equ}
 \pP^{(\ell)}_t(F^2)-(\pP^{(\ell)}_tF)^2
 =
 2\int_0^t \pP^{(\ell)}_{t-s}
 \|{\rm D}(\pP^{(\ell)}_s F)\|^2_{L^2(\T^3_\ell)}\,.
\end{equ}
Hence,
\begin{equ}
 \pP^{(\ell)}_t(F^2)-(\pP^{(\ell)}_tF)^2
 \leq
 C\,\pP^{(\ell)}_{t}\|{\rm D}F\|^2_{L^2(\T^3_\ell)}\,.
\end{equ}
The statement follows now from ergodicity of the $\Phi^4_3$ dynamic on $\T^3_\ell$.
\end{proof}

Theorem~\ref{thm:linearised} also implies synchronisation for the infinite-volume $\Phi^4_3$ dynamic, extending the finite-volume result established in~\cite{GT20}. The proof relies critically on the order-preserving property of the scalar $\Phi^4_3$ dynamic. Among our results, this is the only one that does not generalise to the vector-valued $\Phi^4_3$ model.

\begin{defn}
	We say that $\phi\in\sS'(\R^3)$ is non-negative if $\langle \phi,f\rangle\geq 0$ for all non-negative $f\in\sS(\R^3)$. For $\phi_1,\phi_2\in\sS'(\R^3)$ we write $\phi_1\preceq \phi_2$ if $\phi_2-\phi_1$ is non-negative.
\end{defn}
\begin{rmk}\label{rmk:preceq_dist}
	For $\alpha<0$ we have $\|\phi_1\|_{\cC^\alpha(w)}\lesssim \|\phi_2\|_{\cC^\alpha(w)}$ uniformly over $0\preceq \phi_1\preceq \phi_2$. The last fact is a consequence of the positivity of the heat kernel and the equivalence of the norm $\|\bigcdot\|_{\cC^\alpha(w)}$ to the norm $\phi\mapsto\sup_{r\in(0,1]}r^{-\frac\alpha2}\|e^{r\Delta}\phi\|_{L^\infty(w)}$, which can be proved along the lines of~\cite[Theorem~2.34]{BCD11} with the use of~\cite[Lemmas~2 and 3 and Section~4.1]{MW17a}.
\end{rmk}

\begin{thm}\label{thm:synchronisation}
	Fix $p\in[1,\infty)$. There exist $\lambda_\star\in(0,1]$ and $C,c>0$ such that
	\begin{equ}
		\EE\Bigl(\sup_{\phi_1,\phi_2\in \cC^{-\frac12-\kappa}(w)} \|\Phi(\phi_1;t,\bigcdot)-\Phi(\phi_2;t,\bigcdot)\|^p_{\cC^{-\frac12-\kappa}(w)}\Bigr)
		\leq 
		C\exp(-c\,t)
	\end{equ}
	for all $\lambda\in(0,\lambda_\star]$ and $t\geq 1$.
\end{thm}
\begin{proof}
	Without loss of generality we assume that $p\geq6$. By Remark~\ref{rmk:difference_linearisation} and Theorem~\ref{thm:global_solution}, we obtain
	\begin{equ}
		\sup_{\phi_1,\phi_2\in \cC^{-\frac12-\kappa}(w)}\EE \|\Phi(\phi_1;t,\bigcdot)-\Phi(\phi_2;t,\bigcdot)\|^p_{\cC^{-\frac12-\kappa}(w)}
		\leq 
		C\exp(-c\,t)\,,
	\end{equ} 
	where we used the continuous embeddings $L^p(\rho)\subset \cC^{-\frac12-\kappa}(\rho)$ and $\cC^{-\frac12-\frac\kappa2}(w^{\frac12})\subset \cC^{-\frac12-\kappa}(w^{\frac12})$ (see \cite[Section~6.4.1]{Tri06}), along with interpolation between $\cC^{-\frac12-\kappa}(\rho)$ and $\cC^{-\frac12-\kappa}(w^{\frac12})$. Recall also that the embedding $\cC^{-\frac12-\frac\kappa2}(w^{\frac12})\subset\cC^{-\frac12-\kappa}(w)$ is compact by \cite[Theorem~6.31]{Tri06}. Using Theorem~\ref{thm:global_solution}, Lemma~\ref{lem:map_T} and repeating the argument from Step~1 of the proof of~\cite[Theorem~2.5]{GT20} we construct random functions $\Phi^\pm\in C([1,\infty),\cC^{-\frac12-\kappa}(w))$ such that $\Phi^-(t,\bigcdot)\preceq\Phi(\phi;t,\bigcdot)\preceq\Phi^+(t,\bigcdot)$ for all $t\geq 1$ and $\phi\in \cC^{-\frac12-\kappa}(w)$ and
	\begin{equ}
		\EE \|\Phi^+(t,\bigcdot)-\Phi^-(t,\bigcdot)\|^p_{\cC^{-\frac12-\kappa}(w)}
		\leq 
		C\exp(-c\,t)\,.
	\end{equ}
	The statement  now follows from the bound
	\begin{equs}
		\|\Phi&(\phi_1;t,\bigcdot)-\Phi(\phi_2;t,\bigcdot)\|_{\cC^{-\frac12-\kappa}(w)}
		\\
		&\leq
		\|\Phi(\phi_1;t,\bigcdot)-\Phi^-(t,\bigcdot)\|_{\cC^{-\frac12-\kappa}(w)}
		+
		\|\Phi(\phi_2;t,\bigcdot)-\Phi^-(t,\bigcdot)\|_{\cC^{-\frac12-\kappa}(w)}
		\\
		&\lesssim
		\|\Phi^+(t,\bigcdot)-\Phi^-(t,\bigcdot)\|_{\cC^{-\frac12-\kappa}(w)}\,,
	\end{equs}
	where in the last bound we used $0\preceq\Phi(\phi_j;t,\bigcdot)-\Phi^-(t,\bigcdot)\preceq\Phi^+(t,\bigcdot)-\Phi^-(t,\bigcdot)$ and Remark~\ref{rmk:preceq_dist}.
\end{proof}

\section{Ergodicity in infinite volume}
\label{sec:ergodicity}

The main contribution of this section is to prove Theorem~\ref{thm:linearised} which gives an estimate for the linearised equation \eqref{e:linearised} and serves as a main ingredient for the proof of ergodicity of $\Phi^4_3$ measure in infinite volume. In Section~\ref{sec:weighted} we present definitions of weighted Besov spaces and their properties. The definitions of stochastic objects (see Definitions~\ref{def:smoothed_noise} and \ref{def:enhanced_noise}) and the Da Prato--Debussche trick to rewrite the $\Phi^4_3$ equation (see \eqref{eq:phi_ansatz}--\eqref{e:linearisation1_mod}) are introduced in Section~\ref{sec:stochastic_objects}. In Section~\ref{sec:proof_ergodicity}, we present the proof of Theorem~\ref{thm:linearised}. We then prove Proposition~\ref{pr:deterministic_main} in Section~\ref{sec:proof_ergodicty_aux}, which is the key deterministic bound for the linearised equation used for the proof of Theorem~\ref{thm:linearised}.

\subsection{Weighted space}\label{sec:weighted}

Given any weight $\rmw: \R^3 \rightarrow \R_>$ we use the notation 
\begin{equ}[e:Lebesgue]
\|f\|_{L^p(\rmw)} \eqdef \|f \rmw\|_{L^p}
\end{equ}
for the corresponding weighted Lebesgue spaces. Note that this corresponds to the usual $L^p$ space with respect to the 
measure $\rmw^p(x)\,dx$. The reason for the convention \eqref{e:Lebesgue} is that it is still useful when $p=\infty$. Throughout this article, we will work with the following sub-exponential and polynomial weights.

\begin{defn}
	\label{defn:weights}
	For $\delta\in\R$, $\zzz\in\R^3$ and $\fa,\fb\in(0,1]$ we define
	\begin{equ}
		e_\delta\eqdef e^{-\delta\bracket{\bigcdot}^{1/2}}\,,
		\quad
		e_{\delta,\zzz} \eqdef e_{\fa\delta+\fb}(\bigcdot-\zzz),
		\quad
		\rho\eqdef\bracket{\bigcdot}^{-4}\,,
		\quad
		w\eqdef\bracket{\bigcdot}^{-\kappa}\,,
		\quad
		w_{\zzz} \eqdef w(\bigcdot -\zzz)\,.
	\end{equ}
\end{defn}

\begin{lem}\label{lem:weights}
	Given $p\in[1,\infty)$ there exists a choice of parameters $\fa,\fb\in(0,1]$ such that
	\begin{equ}\label{eq:heat_almost_contraction}
		\|e^{t\Laplace}f\|^p_{L^p(e_{\delta,\zzz})}
		\leq
		e^{1/6}\,\|f\|^p_{L^p(e_{\delta,\zzz})}
	\end{equ}
	for all $t\in(0,1]$, $\delta\in[0,4]$, $\zzz\in\R^3$ and $f\in L^\infty$ and
	\begin{equ}\label{eq:bound_weight_e}
		e^{-1/12}\,\frac{\int_{\R^3}\rho_\fc(\zzz)\,\|f\|^p_{L^p(e_{0,\zzz})}\,\md\zzz}{\int_{\R^3} e_{0,\zzz}(x)^p\,\md\zzz}
		\leq
		\|f\|^p_{L^p(\rho_\fc)}
		\leq
		e^{1/12}\,\frac{\int_{\R^3} \rho_\fc(\zzz)\,\|f\|^p_{L^p(e_{\delta,\zzz})}\,\md\zzz}{\int_{\R^3} e_{0,\zzz}(x)^p\,\md\zzz}
	\end{equ}
	for all $\delta\in[0,2]$, $\fc\in(0,\fa^2]$ and $f\in L^\infty$, where $\rho_\fc\eqdef\rho(\fc\, \bigcdot)$.
\end{lem}
\begin{rmk}
	The remaining results of this section are true for generic $p\in[1,\infty]$ and $\fa,\fb\in(0,1]$. In Section~\ref{sec:proof_ergodicity} and~\ref{sec:proof_ergodicty_aux} the exponent $p\in[1,\infty)$ is fixed as in Theorem~\ref{thm:linearised} (see also Remark~\ref{rmk:lambda_p}) and the parameters $\fa,\fb\in(0,1]$ are fixed so that the bounds stated this lemma are true.
\end{rmk}
\begin{proof} Let $P(t,\bigcdot)$ be the kernel of the heat semigroup $\exp(t\Laplace)$. We choose the parameter $\fb \in(0,1]$ 
small enough so that for all $t\in (0,1]$ we have
	\begin{equ}
		\int_{\R^3} P(t,x)\,e^{3\fb\bracket{x}^{1/2}}\,\md x\leq e^{1/(6p)}\,.
	\end{equ}
	Then
	\begin{equ}
		\int_{\R^3}  \frac{P(t,x)}{e_{\delta,0}(x)}\,\md x
		\in[1,e^{1/(6p)}]
	\end{equ}
	for all $t\in[0,1]$, $\delta\in[0,2]$ and $\fa\in(0,\fb]$. The bound~\eqref{eq:heat_almost_contraction} follows now from the estimate $e_{\delta,\zzz}(x)\leq e_{\delta,\zzz}(y)/e_{\delta,0}(x-y)$ and the Young inequality for convolutions. 
	
	We now prove \eqref{eq:bound_weight_e}. We choose the parameter $\fa\in(0,\fb]$ so that 
	\begin{equ}
		\int_{\R^3} (1+\fa^2|\zzz|_2)^{-4}\,e^{-p(2\fa+\fb)\bracket{\zzz}^{1/2}}\,\md\zzz
		\geq
		e^{-1/12}\,\int_{\R^3} e^{-p\fb\bracket{\zzz}^{1/2}}\,\md\zzz
	\end{equ}
	and
	\begin{equ}
		\int_{\R^3} (1+\fa^2|\zzz|_2)^4\,e^{-p\fb\bracket{\zzz}^{1/2}}\,\md\zzz
		\leq
		e^{1/12}\,\int_{\R^3} e^{-p\fb\bracket{\zzz}^{1/2}}\,\md\zzz\,.
	\end{equ}
	Observe that
	\begin{equ}
		(1+\fc|x-\zzz|_2)^{-4} \leq \rho_\fc(\zzz)/\rho_\fc(x)\leq (1+\fc|x-\zzz|_2)^4
	\end{equ}
	for all $\fc\in(0,1]$ and $x,\zzz\in\R^3$. Thus,
	\begin{equ}
		\frac{\int_{\R^3} \rho_\fc(\zzz)\,e_{\delta,\zzz}(x)^p\,\md\zzz}{\rho_\fc(x)\int_{\R^3} e_{0,\zzz}(x)^p\,\md\zzz}
		\in[e^{-1/12},e^{1/12}]
	\end{equ}
	for all $x\in\R^3$, $\delta\in[0,2]$ and $\fc\in(0,\fa^2]$. The bound~\eqref{eq:bound_weight_e} follows now from the Fubini theorem.
\end{proof}

We have the following basic property for the weights.
\begin{lem}
	\label{le:weights_time}
	For all $N\geq0$ there exists a constant $C>0$ such that
	\begin{equ}
		e_{t,\zzz}(x) \leq C\, (t-s)^{-2N\kappa}\,w_z(x)^N\, e_{s,\zzz}(x)
	\end{equ}
	for all $-\infty<s < t<\infty$ and $x,\zzz \in \R^3$.
\end{lem}

Now we define the weighted Besov spaces.
\begin{defn}\label{defn:weighted_Besov}
	Let $(\chi_j)_{j\geq-1}$ be the smooth dyadic decomposition of unity belonging to the Gevrey class of index $\frac32$ defined in~\cite[Section~3.1]{MW17a}. The Littlewood--Paley blocks $(\delta_j)_{j\geq-1}$ are defined by the formula $\delta_jf\eqdef\mathscr{F}^{-1}(\chi_j \mathscr{F}f)$ for $f\in\sS'(\R^3)$, where $\mathscr{F}$ denotes the Fourier transform. Given a weight $\rmw: \R^3 \rightarrow \R_>$ and parameters $\alpha \in \R$, $p, q \in [1, \infty]$, we define the weighted Besov norm of a distribution $f\in\sS'(\R^3)$ to be
	\begin{equ}
		\|f \|_{\Besov{\alpha}{p, q}(\rmw)} \eqdef \Big(\sum_{j \geq -1} \|\delta_j f\|_{L^p(\rmw)}^q 2^{j \alpha q} \Big)^{\frac{1}{q}} \;,
	\end{equ} 
	where the case $q=\infty$ is interpreted as a supremum. The weighted Besov space $\Besov{\alpha}{p, q}(\rmw)$ is defined as the completion of $C_c^\infty(\R^3)$ with respect to the above norm. We use the notation $\cC^{\alpha}(\rmw)=\Besov{\alpha}{\infty, \infty}(\rmw)$.
\end{defn}

\begin{remark}
	The Besov spaces $\Besov{\alpha}{p, q}(\rmw)$ defined in this way are separable, even when $p$ and \slash or $q$ are infinite. We only consider Besov norms $\|\bigcdot\|_{\Besov{\alpha}{p, q}(\rmw)}$ with weights $\rmw$ of the form $e_{t, \zzz} w_{\zzz}^{N}$. For such weights $\rmw$, elements of $\Besov{\alpha}{p, q}(\rmw)$ can be interpreted as distributions. Moreover, using Lemma~\ref{lem:PL_block_est} one shows, along the lines of the proof of~\cite[Lemma~2.23]{BCD11}, that the weighted Besov norms corresponding to different choices of the dyadic decomposition of unity $(\chi_j)_{j\geq-1}$ belonging to the Gevrey class of index $\frac32$ are equivalent, see \cite[Remark~14]{MW17a} for more details.
\end{remark}

\begin{lem}\label{lem:PL_block_est}
	Let $\gamma\in[0,2/3)$. In the setting of the above definition we have $\|\delta_j f\|_{L^p(\rmw)}\lesssim C\,\|f\|_{L^p(\rmw)}$ uniformly over $j\geq -1$, $f\in L^p(\rmw)$ and $\rmw: \R^3 \rightarrow \R_>$ such that $\rmw(x)/\rmw(y)\leq C \exp(|x-y|^\gamma)$ with $C>0$ for all $x,y\in\R^3$.
\end{lem}
\begin{proof}
	The result follows from the identity $\chi_i=\chi_0(\bigcdot/2^i)$ for $i\geq0$, the decay property of the Fourier transform of a function in the Gevrey class~\cite[Proposition~1]{MW17a} and the weighted Young inequality~\cite[Theorem~2.1]{MW17a}.
\end{proof}

By definition of the Besov norm we immediately get the following properties.
\begin{lem}[Monotonicity]\label{lem:monotone}
	If $\rmw_1 \leq \rmw_2$, then $\|f\|_{\Besov{\alpha}{p, q}(\rmw_1)} \leq \|f\|_{\Besov{\alpha}{p, q}(\rmw_2)}$.
\end{lem}
\begin{lem}[Translation invariance]\label{lem:Besov_trans}
	Let $\tau_z f$ denote the translation of a function or a~distribution $f$ in space by $z$. We have $\|f\|_{\Besov{\alpha}{p, q}(\tau_z\rmw)}=\|\tau_{-z} f\|_{\Besov{\alpha}{p, q}(\rmw)}$.
\end{lem}

Now we discuss some estimates for the weighted Besov norms that we will frequently use.

\begin{lem}\label{le:besov_embedding} 
	Let $p\in [1,\infty]$ and $\alpha>0$. Then we have 
  	\begin{equ}
       \| f \|_{\Besov{-\alpha}{p,p}(e_{\delta,\zzz})} \lesssim \| f \|_{L^p(e_{\delta,\zzz})} \lesssim \| f \|_{\Besov{\alpha}{p,p}(e_{\delta,\zzz})}
  	\end{equ}	
  	uniformly over $\delta$ in a compact subset of $\R_\geq$, $\zzz \in \R^3$ and $f\in C_c^\infty(\R^3)$.
\end{lem}
\begin{proof} 
	The result follows from the definition of the Besov norm and Lemma~\ref{lem:PL_block_est}.
\end{proof}

\begin{lem}\label{le:besov_derivative}
	For any $\alpha \in \R$, $p, q \in [1, \infty]$, we have
	\begin{equ}
		\label{e:derivative_Besov}
		\|\grad f \|_{\bB^{\alpha}_{p, q}(e_{\delta,\zzz})} \lesssim \|f\|_{\bB^{\alpha + 1}_{p,q}(e_{\delta,\zzz})}
	\end{equ}
	uniformly over $\delta$ in a compact subset of $\R_\geq$, $\zzz \in \R^d$ and $f\in C_c^\infty(\R^3)$.
\end{lem}
\begin{proof}
	The result follows from Lemma~\ref{lem:Besov_trans} and \cite[Proposition~3]{MW17a}.
\end{proof}

\begin{lem}[Smoothing effect of heat flow]\label{le:smoothing_heat}
	Let $\alpha \geq \beta$ and $p, q \in [1, \infty]$. We have
	\begin{equ}
		\label{e:heat_flow1}
		\| e^{t \Delta} f \|_{\Besov{\alpha}{p, q}(e_{\delta,\zzz})}
		\lesssim
		t^{\frac{\beta - \alpha}{2}} \|f\|_{\Besov{\beta}{p, q}(e_{\delta,\zzz})}
	\end{equ}
	uniformly over $\delta$ in a compact subset of $\R_\geq$, $\zzz \in \R^d$, $t\in(0,1]$ and $f\in C_c^\infty(\R^3)$.
\end{lem}
\begin{proof}
	The result follows from Lemma~\ref{lem:Besov_trans} and \cite[Proposition~5]{MW17a}.
\end{proof}

\begin{lem}\label{lem:heat_smoothing_L_inf}
    For all $T>0$, $\alpha<0$ and $k\in\N_0^3$ we have
\begin{equ}\label{eq:heat_smoothing}
    \|\partial^k e^{t\Delta}f\|_{L^\infty(w)}
    \lesssim
    t^{\frac{\alpha-|k|}{2}} \|f\|_{\cC^\alpha(w)}
\end{equ}
uniformly over $t\in(0,T]$ and $f\in \cC^\alpha(w)$.
\end{lem}
\begin{proof} First note that the Bernstein inequality in \cite[Lemma~2]{MW17a} and the smoothing of the heat flow in \cite[Lemma~3]{MW17a} for $L^\infty(w)$ norm also hold due to the fact that
\begin{equ}
    w(x+y) \lesssim w(x)/w(y)
\end{equ}
holds uniformly for all $x,y\in \R^3$. By \cite[Lemma~2]{MW17a}  for the $L^\infty(w)$ norm, we have 
    \begin{equ}
        \|\partial^k g \|_{L^\infty(w)} \leq \sum_{j\geq -1} \| \partial^k \delta_j g \|_{L^\infty(w)} \lesssim \sum_{j\geq -1} 2^{|k| j} \| \delta_j g \|_{L^\infty(w)} = \| g \|_{\Besov{|k|}{\infty,1}(w)} \;,
    \end{equ} 
    so that it suffices to bound $\| e^{t\Delta }f\|_{\Besov{|k|}{\infty,1}(w)} $. By applying \cite[Lemma~3]{MW17a}, there exists $c>0$ such that for all $t\in(0,1]$
    \begin{equs}
       t^{-\frac{\alpha-|k|}{2}}\| e^{t\Delta }f\|_{\Besov{|k|}{\infty,1}(w)} &\leq \sum_{j\geq -1 }  t^{-\frac{\alpha-|k|}{2}} 2^{|k|j}\big\| \delta_j e^{t\Delta} f \big\|_{L^\infty(w)} \\
       &\lesssim \sum_{j\geq -1}  t^{-\frac{\alpha-|k|}{2}} 2^{-(\alpha-|k|)j} 2^{\alpha j} e^{-ct 2^{2j}} \| \delta_j f\|_{L^\infty(w)}\;.
    \end{equs} Then by \cite[Lemma~2.35]{BCD11} we can bound 
    \begin{equ}
        \sum_{j\geq -1}  t^{-\frac{\alpha-|k|}{2}} 2^{-(\alpha-|k|)j} 2^{\alpha j} e^{-ct 2^{2j}} \| \delta_j f\|_{L^\infty(w)} \lesssim \sup_{j\geq -1} 2^{\alpha j}\|\delta_j f \|_{L^\infty(w)} 
        = \| f \|_{\cC^\alpha(w)}\;,
    \end{equ}
concluding the proof.
\end{proof}

\begin{lem}[Paraproduct estimates]\label{le:paraproduct} 
	Let $\alpha,\beta\in \R$, $p\in[1,\infty]$, $N,M\geq 0$, $\rmw_1=e_{\delta,\zzz}w_z^N$, $\rmw_2=w_z^M$ and $\weight=\weight_1\weight_2$. The bounds
	\begin{equs}
		\label{e:para1}
		\| f \para g \|_{\Besov{\alpha}{p,p}(\rmw)}
		&\lesssim
		\|f\|_{\Besov{-\beta/2}{p,p}(\rmw_1)}\, \|g\|_{\cC^{\alpha+\beta}(\rmw_2)}\;,
		\ \ & &  \text{if} \ \ \beta>0 \;,
		\\
		\label{e:para2}
		\| f \rpara g \|_{\Besov{\alpha+ \beta}{p,p}(\rmw)}
		&\lesssim
		\|f\|_{\Besov{\alpha}{p,p}(\rmw_1)}\,\|g\|_{\cC^{\beta}(\rmw_2 )} \;,
		\ \ & &  \text{if} \ \ \beta  <0 \;,
		\\
		\label{e:reso}
		\| f \reso g \|_{\Besov{\alpha+ \beta}{p,p}(\rmw)}
		&\lesssim
		\|f\|_{\Besov{\alpha}{p,p}(\rmw_1)} \|g\|_{\cC^{\beta}(\rmw_2)}\,\;,\ \ \ & & \text{if} \ \ \alpha+\beta > 0 \;,
	\end{equs}
	and
	\begin{equs}
		\| f \para g \|_{\Besov{\alpha}{p,p}(\rmw)}
		&\lesssim
		\| f\|_{L^p(\rmw_1)}\, \|g\|_{\cC^{\alpha+\beta}(\rmw_2)} \;,
		\ \ & &  \text{if} \ \ \beta>0 \;,
		\\
		\| f \rpara g \|_{\Besov{\alpha}{p,p}(\rmw)}
		&\lesssim
		\|f\|_{\Besov{\alpha+\beta}{p,p}(\rmw_1)}\,\|g\|_{L^\infty(\rmw_2)}  \;,
		\ \ & &  \text{if} \ \ \beta>0 \;,
		\\
		\| f \reso g \|_{\Besov{\alpha}{p,p}(\rmw)}
		&\lesssim
		\| f\|_{L^p(\rmw_1)}\, \|g\|_{\cC^{\alpha+\beta}(\rmw_2)} \,\;,\ \ \ & & \text{if} \ \ \alpha,\beta> 0 \;,
	\end{equs} 
	hold uniformly over $\delta$ in a compact subset of $\R_\geq$, $z\in\R^3$ and $f, g\in C_c^\infty(\R^3)$.
\end{lem}

\begin{proof}
	The proof is almost the same as the proof of \cite[Theorem~3.1]{MW17a} and is presented for the sake of completeness. Writing $S_{k}f = \sum_{j<k} \delta_j f $, we first note that, as a consequence of Lemma~\ref{lem:PL_block_est}, one has
	\begin{equ}
		\| f \para g \|_{\Besov{\alpha}{p,p}(\rmw)} \lesssim \Bigl( \sum_{k\geq 0} 2^{\alpha k p } \| S_{k-1}f \delta_{k} g \|_{L^p(\rmw)}^p\Bigr)^{1/p}\;.
	\end{equ}
	We have that for all $k\in \N_0$ and $\beta>0$
	\begin{equ}
		\| S_{k-1}f \delta_{k} g \|_{L^p(\rmw)} \leq  \| S_{k-1}f \|_{L^p(\rmw_1)} \| g \|_{\cC^{\alpha+\beta}(\rmw_2)} 2^{-(\alpha+\beta) k }  \;,
	\end{equ} which also implies that 
	\begin{equ}
		\Bigl(\sum_{k\geq 0} 2^{\alpha k p } \| S_{k-1}f \delta_{k} g \|_{L^p(\rmw)}^p \Bigr)^{1/p}\lesssim \Bigl( \sum_{k\geq 0 } 2^{-\beta  kp} \| S_{k-1}f \|_{L^p(\rmw_1)}^p \Bigr)^{1/p}   \| g \|_{\cC^{\alpha+\beta }(\rmw_2)}\;.
	\end{equ} On the other hand, we have from H\"older's inequality and Lemma~\ref{lem:PL_block_est} that 
	\begin{equs}
		\| S_{k-1}f \|_{L^p(\rmw_1)} \leq \sum_{j=-1}^{k-2} \| \delta_j f \|_{L^p(\rmw_1)} &\leq \Bigl( \sum_{j=-1}^{k-2} 2^{-\beta j p/2} \| \delta_ j f \|_{L^{p}(\rmw_1)}^p\Bigr)^{1/p} \Bigl( \sum_{j=-1}^{k-2} 2^{\beta j q/2} \Bigr)^{1/q}\\&\lesssim 2^{\beta k/2} \| f \|_{\Besov{-\beta/2}{p,p}(\rmw_1)} \;,
	\end{equs} 
	where $q=p/(p-1)$. Combining the estimates above, we get 
	\begin{equ}
		\Bigl(\sum_{k\geq 0} 2^{\alpha k p } \| S_{k-1}f \delta_{k} g \|_{L^p(\rmw)}^p \Bigr)^{1/p}
		\lesssim
		\| f \|_{\Besov{-\beta/2}{p,p}(\rmw_1)} \| g \|_{\cC^{\alpha+\beta}(\rmw_2)}\,.
	\end{equ} For the second estimate, we bound 
  \begin{equ}
    \| f \rpara g \|_{\Besov{\alpha+\beta }{p,p}(\rmw)} \lesssim \Bigl( \sum_{k\geq -1} 2^{(\alpha +\beta)  k p } \| S_{k-1} g  \delta_k f  \|_{L^p(\rmw)}^p\Bigr)^{1/p}\;.
  \end{equ} We have that for all $k\in\N_0$ and  $\beta<0$
  \begin{equ}
    \| S_{k-1} g \|_{L^\infty(\rmw_2)} \leq \sum_{j=-1}^{k-2} \| \delta_j g \|_{L^{\infty}(\rmw_2)}\lesssim \| g \|_{\cC^{\beta}(\rmw_2)} \sum_{j<k-1} 2^{-\beta j }  \lesssim \| g \|_{\cC^{\beta}(\rmw_2)}  2^{-\beta k}\;.
  \end{equ} Therefore, we have 
  \begin{equs}
    \Bigl (\sum_{k\geq -1} 2^{(\alpha +\beta)  k p } \| S_{k-1} g  \delta_k f  \|_{L^p(\rmw)}^p \Bigr)^{1/p}&\lesssim  \Bigl(\sum_{k\geq -1} 2^{\alpha k p } \| \delta_k f \|_{L^p(\rmw_1)}^p \Bigr)^{1/p} \| g \|_{\cC^{\beta}(\rmw_2)} \\
    & = \| f \|_{\Besov{\alpha}{p,p}(\rmw_1)} \| g \|_{\cC^{\beta}(\rmw_2)}\;.
  \end{equs} 
	The third inequality is proven in the same way, so we skip it. The second part of the statement follows from \eqref{e:para1}--\eqref{e:reso} and Lemma~\ref{le:besov_embedding}. 
\end{proof}

We also record an embedding result used in the proof of Theorem~\ref{thm:uniqueness}. 
\begin{lem}
	For $\alpha\in\R$ we have 
	\begin{equ}
		\| f \|_{B_{2,2}^{\alpha-\kappa}(\rho)} \lesssim \| f \|_{\cC^\alpha(w)}\;.
	\end{equ}
\end{lem}
\begin{proof}
	Since $\rho^2 w^{-1}$ is integrable, it follows that
	\begin{equs}
		\| f \|^2_{B_{2,2}^{\alpha-\kappa}(\rho)}
		&= 
		\sum_{j\geq -1} \| \delta_{j} f\|^2_{L^2(\rho)} 2^{2(\alpha-\kappa)j} 
		\\
		&\leq \left(\sup_{j\geq -1} \| \delta_j f\|^2_{L^\infty(w)} \,2^{2\alpha j}\right) \left( \sum_{j\geq -1} 2^{-2\kappa j}\right) \|\rho^2 w^{-1}\|_{L^1}
		\lesssim
		\|f \|^2_{\cC^\alpha(w)}\;,
	\end{equs}
	which completes the proof.
\end{proof}

\subsection{Stochastic objects}
\label{sec:stochastic_objects}

\begin{defn}\label{def:smoothed_noise}
Let $\xi$ denote space-time white noise on $\R^{1+3}$, let $Q_\ell = [-\frac{\ell}{2},\frac{\ell}{2})^3$, and let $\xi_{\ell}$ be the spatial periodisation of $\one_{\R \times Q_\ell}\,\xi$ with period $\ell\in\N_+$. Furthermore, for $\eps\in(0,1]$ and $\ell\in\N_+$ we define $\xi_{\eps, \ell} \eqdef M_\eps\star\xi_\ell$, where $\star$ denotes the convolution over $\R^3$ and the mollifier $M_\eps\in C^\infty(\R^3)$ is defined by $M_\eps(x)=\eps^{-3}M(\frac{x}{\eps})$ for $M\in C^\infty(\R^3)$ supported in the unit ball such that $\int M(x)\,\md x=1$. Setting
\begin{equ}
\big(\lL^{-1} \phi \big) (t, \bigcdot) \eqdef \int_{-\infty}^{t} e^{(t-s)(\Laplace - 1)} \phi(s,\bigcdot)\, \md s\;,
\end{equ}
we define $\ou_{\eps,\ell} \eqdef \lL^{-1} \xi_{\eps,\ell}$ and $\Cone \eqdef \EE |\ou_{\eps,\ell}(t, x)|^2$. We
further define
\begin{equ}
\oud_{\eps,\ell} \eqdef \ou_{\eps,\ell}^2 - \Cone\;, \qquad \oudz_{\eps,\ell} \eqdef \lL^{-1} \oud_{\eps,\ell} \;, \qquad \Ctwo \eqdef \EE \oudz_{\eps,\ell}(t, x) \oud_{\eps,\ell}(t, x) \;.
\end{equ}
Note that, by stationarity, $\Cone$ and $\Ctwo$ are constants over space-time. Finally, we set $C_{\eps, \ell}(\lambda)$ in \eqref{e:phi4_mod} to be $C_{\eps, \ell}(\lambda) \eqdef 3 \lambda C_{\eps, \ell}^{(1)} - 9\lambda^2 C_{\eps, \ell}^{(2)}$.
\end{defn}	

We will also make use of the following stochastic objects, all starting from zero initial data.
\begin{defn}\label{def:enhanced_noise}
We define the renormalisation functions
\begin{equ}
\Cones(t)\eqdef \EE(\ou_{\eps,\ell,s}(t))^2\;,\qquad \Ctwos(t)\eqdef \EE(\nabla\oudz(t))^2\;.
\end{equ}
	We define $\ou_{\eps,\ell,s},\oud_{\eps,\ell,s},\out_{\eps,\ell,s},\oudz_{\eps,\ell,s},\outz_{\eps,\ell,s}\in C(\R\times\T_\ell^3)$ by the following equations
	\begin{equs}
		\label{e:ou_eps_s}
		\lL \ou_{\eps,\ell,s}(t) 	&\eqdef\xi_{\eps,\ell}(t)\;,&\quad \ou_{\eps,\ell,s}(s) &\eqdef 0\;,
		\\
		\label{eq:oud_def}
		\oud_{\eps,\ell,s}(t) &\eqdef 	(\ou_{\eps,\ell,s}(t))^2 - \Cones(t) \;, 
		\\
		\label{eq:out_def}
		\out_{\eps,\ell,s}(t) &\eqdef 	(\ou_{\eps,\ell,s}(t))^3 - 3\Cones(t) \ou_{\eps,\ell,s}(t) \;,
		\\
		\label{eq:outz_def}
		\lL \oudz_{\eps,\ell,s}(t) &\eqdef 	\oud_{\eps,\ell,s}(t) \;, &\qquad \oudz_{\eps,\ell,s} (s) &\eqdef 0\;,
		\\
		\label{eq:oudz_def}
		\lL \outz_{\eps,\ell,s}(t) 	&\eqdef\out_{\eps,\ell,s}(t)\;, &\qquad \outz_{\eps,\ell,s}(s) &\eqdef 0\;,
	\end{equs}
	for $t>s$. It is understood that the above functions are identically zero on $(-\infty,s)\times\R^3$ and that~\eqref{e:ou_eps_s}, \eqref{eq:outz_def} and~\eqref{eq:oudz_def} are interpreted in the mild form. We write $\tilde\oud_{\eps,\ell,s}$ and $\tilde\out_{\eps,\ell,s}$ for analogues of $\oud_{\eps,\ell,s}$ and $\out_{\eps,\ell,s}$ defined by~\eqref{eq:oud_def} and~\eqref{eq:out_def} with $\Cones(t)$ replaced by $\one_{(s,\infty)}\Cone$. We set $\tilde\oudz_{\eps,\ell,s}\eqdef K^+\ast\tilde\oud_{\eps,\ell,s}$ and $\tilde\outz_{\eps,\ell,s}\eqdef K^+\ast\tilde\out_{\eps,\ell,s}$, where $K^+$ is the truncation of the heat kernel from Lemma~\ref{lem:kernel_decomposition}. 
\end{defn}

Now we are ready to decompose the solution $\PHI$ using the following Da Prato--Debussche trick.  For $0\leq s \leq t $, we write the solutions to \eqref{eq:phi4_S} as
\begin{equ}\label{eq:phi_ansatz}
  \Phi_{\eps,\ell}(t)
  =
  \ou_{\eps,\ell,s}(t) + \PSImod(t)
  =
  \ou_{\eps,\ell,s}(t) -\lambda\outz_{\eps,\ell,s}(t) + \PSI(t)\;.
\end{equ}
Then we can rewrite~\eqref{eq:phi4_S} and~\eqref{e:linearised} as
\begin{equs}[e:psi_mod]
  \lL \PSImod
  =
  S&- \lambda \PSImod^3
  - 3\lambda \PSImod^2 \ou_{\eps,\ell,s}
  - 3\lambda \PSImod \tilde\oud_{\eps,\ell,s}
  \\
  &- 3\lambda \tilde\out_{\eps,\ell,s}
  -9\lambda^2\Ctwo\,(\ou_{\eps,\ell,s} + \PSImod)
\end{equs} and
\begin{equs}[e:linearisation1_mod]
  \lL D_{\eps,\ell}
  =
  -\Big(&3\lambda\PSImod^2-6\lambda\PSI\ou_{\eps,\ell,s}
  + 3\lambda\oud_{\eps,\ell,s}
  \\
  &+ 6\lambda^2 \ou_{\eps,\ell,s} \outz_{\eps,\ell,s}
  - 3\lambda\big(\Cone - \Cones\big) + 9\lambda^2 C^{(2)}_{\eps,\ell} \Big) D_{\eps,\ell}\;.
\end{equs}
Note that the above equations make sense for $t\geq s$. In Lemma~\ref{le:psi}, we obtain an estimate for the process $\PSI = \PSImod +\lambda\outz_{\eps,\ell,s}$, which (in the limit $\eps\searrow0$) has much better regularity than $\PHI$. To state this estimate, we need the following definition.

\begin{defn}\label{def:size_enhanced_noise}
We define
\begin{equs}
\hat\fX(\ou,\oudz,\outz,C,w,t)
\eqdef& \|\ou(t)\|_{\cC^{-\f12-\kappa}(w)}
\vee
\|\oudz(t)\|_{\cC^{1-2\kappa}(w)}
\vee
\|\outz(t)\|_{\cC^{1/2-3\kappa}(w)}
\\
&
\vee
\|\ou(t)\reso\outz(t)\|_{\cC^{-4\kappa}(w)}
\vee
\|(\nabla\oudz(t))^2-C(t)\|_{\cC^{-4\kappa}(w)}\,.
\end{equs}
For $s<t$, we then set
\begin{equs}
\fX_{\eps,\ell,s,t,\zzz}
\eqdef
1&\vee\sup_{u\in[s,t]}\hat\fX(\ou_{\eps,\ell,s},\oudz_{\eps,\ell,s},\outz_{\eps,\ell,s},\Ctwos,w_\zzz,u)
\\
&\vee\tilde\fX(\ou_{\eps,\ell,s},\tilde\oud_{\eps,\ell,s},\tilde\out_{\eps,\ell,s},\tilde\oudz_{\eps,\ell,s},\tilde\outz_{\eps,\ell,s},\one_{(s,\infty)}\Ctwo,[s,t]\times\R^3,w_\zzz)\,,
\end{equs}
where $\tilde\fX$ is introduced in Definition~\ref{def:enhanced_noise_spacetime}.
\end{defn}

\begin{lem}\label{le:stochastic}
There exists $C>0$ such that
\begin{equ}
	\EE \fX_{\eps,\ell,s, s+1,\zzz} \leq C
\end{equ}
for all $\eps\in(0,1]$, $\ell\in\N_+$, $s\in\R$, $\zzz\in\R^3$.
\end{lem}
\begin{proof}
	By translational invariance without loss of generality we may assume that $\zzz=0$. The result then follows immediate from Lemmas~\ref{lem:stochastic_paracontrolled} and~\ref{lem:stochastic_spacetime}.
\end{proof}

The ``coming down from infinity'' estimate stated below is proved in Appendix~\ref{sec:spacetime_localization}.
\begin{lem}\label{le:psi}
There exists $C>0$ such that
\begin{equs}
	\lambda^{1/2}\|\PSI(t)\|_{L^\infty(w_{\zzz}^{3})}
	&\leq C\,(t-s)^{-1/2}\vee C\, \fX_{\eps,\ell,s,t,\zzz}^{2/(1-2\kappa)},
	\\
	\lambda^{1/2}\|\PSI(t)\|_{\cC^{1/2 + 4\kappa}(w_\zzz^{4})}
	&\leq C\,(t-s)^{-3/4-2\kappa}\vee C\, \fX_{\eps,\ell,s,t,\zzz}^{(3+8\kappa)/(1-2\kappa)}\;,
\end{equs}
for all $\lambda\in(0,1]$, $\eps\in(0,1]$, $\ell\in\N_+$, $s\geq 0$, $t\in(s,s+1]$ and all $\PSImod=\PSI-\lambda\outz_{\eps,\ell,s}$ solving~\eqref{e:psi_mod} in the domain $[s,t]\times\R^3$ with an adapted and continuous $S$ in a unit ball of $L^\infty(\R_\geq\times\T_\ell^3)$ and arbitrary initial data.
\end{lem}

The following fact follows directly from Definitions~\ref{def:filtration} and~\ref{def:size_enhanced_noise}.
\begin{lem}
The random variable $\fX_{\eps,\ell,s,t,\zzz}$ is measurable with respect to $\fF_t$ and independent of $\fF_s$.
\end{lem}

\begin{defn}\label{def:stopping_time_T}
For $\eta\geq 1$, $\eps\in(0,1]$, $\ell\in\N_+$, $s\geq 0$, $\zzz\in\R^3$ define the stopping time
\begin{equ}
\label{e:stopping_time}
T_{\eps,\ell,s,\zzz} \eqdef \inf \{ t \geq s: \fX_{\eps,\ell,s, t,\zzz} \geq \eta\}\wedge (s+1) \;.
\end{equ}
\end{defn}

\begin{rmk}
It is easy to see that $t\mapsto \fX_{\eps,\ell,s, t,\zzz}$ is a.s.\ continuous for arbitrary fixed $\eps,\ell$. This implies that $T_{\eps,\ell,s,\zzz}>s$ a.s.
\end{rmk}

\begin{lem}
There exists $\eta\geq1$ such that for all $\eps\in(0,1]$, $\ell\in\N_+$ and $\zzz\in\R^3$ we have
\begin{equ}\label{e:stopping_time_lower_bound}
	\PP\big(T_{\eps,\ell,0,\zzz} < 1 \big)
	=
	1-\PP\big(T_{\eps,\ell,0,\zzz} =1 \big)
	\leq
	\frac{1}{100} \;.
\end{equ}
\end{lem}
\begin{proof}
Note that on the event $\{T_{\eps,\ell,0,\zzz} < 1\}$ we have $\fX_{\eps,\ell,0,1,\zzz}\geq \eta$. Hence, by Lemma~\ref{le:stochastic} we obtain
\begin{equ}
\PP\big(T_{\eps,\ell,0,\zzz} < 1 \big)
\leq
\EE\fX_{\eps,\ell,0,1,\zzz} /\eta
\leq
C/\eta\;,
\end{equ}
so it remains to choose $\eta$ large enough.
\end{proof}

\subsection{Proof of Theorem~\ref{thm:linearised}}\label{sec:proof_ergodicity}

In what follows, $\eta\geq 1$ is fixed so that~\eqref{e:stopping_time_lower_bound} holds and the parameters $\fa,\fb\in(0,1]$ of the weights introduced in Definition~\ref{defn:weights} are fixed as in Lemma~\ref{lem:weights}.
Our main result is the following deterministic bound for \eqref{e:linearisation1_mod} (equivalently, \eqref{e:linearised}), which employs a time-dependent weight inspired by~\cite{HL15}.

\begin{prop}
\label{pr:deterministic_main}
Fix $p\ge1$. Suppose that $D_{\eps,\ell}$ solves~\eqref{e:linearised} in the time interval $[s,T_{\eps,\ell,s,\zzz}]$. Then, there exists $\lambda_\star> 0$ such that
\begin{equ}[e:deterministic_main]
\lVert\exp(\gamma\langle\bigcdot\rangle_\ell)\DDD(t)\rVert^p_{L^p(e_{2t,\zzz})}
\leq
\exp\big(1/3-p\,(t-s)\big)\,
\lVert\exp(\gamma\langle\bigcdot\rangle_\ell)\DDD(s)\rVert^p_{L^p(e_{2s,\zzz})}
\end{equ}
for all $\lambda\in(0,\lambda_\star]$, $\eps\in(0,1]$, $\ell\in\N_+$, $0\leq s\leq t\leq T_{\eps,\ell,s,\zzz}$, $\gamma\in[0,\lambda_\star]$ and $\Phi_{\eps,\ell}$  solving~\eqref{eq:phi4_S} with an adapted and continuous $S$ in a unit ball of $L^\infty(\R_\geq\times\T_\ell^3)$ and arbitrary initial data.
\end{prop}

Note that in this proposition, we allow our weight to be centred at any point $\zzz\in\R^3$. We will exploit this fact, together with the stationarity of $\zzz \mapsto T_{\eps,\ell,s,\zzz}$, by averaging \eqref{e:deterministic_main} over $\zzz$ to get an estimate with \textit{the same} weight on both sides. This leads to the final proof of Theorem~\ref{thm:linearised}.

\begin{proof}[Proof of Theorem~\ref{thm:linearised}]
Throughout the proof we omit the dependence on $\eps,\ell$ and set $T_{s,z} = T_{\eps,\ell,s,z}$ for the stopping time introduced in Definition~\ref{def:stopping_time_T}. 

Fix an arbitrary point $\zzz \in \R^3$. With Proposition~\ref{pr:deterministic_main} established, we proceed by closely following the approach in \cite[Section~3.2]{KT22}. We define a family of stopping times $(\tau(i,\zzz))_{i\in\N_0}$ by $\tau(0,\zzz)=0$ and for $i\in\N_+$
\begin{equ}
\tau(i,\zzz) \eqdef T_{\tau(i-1,\zzz),\zzz}\;.
\end{equ}
By definition, for any $s\geq 0$, $T_{s,\zzz}-s$ is 
independent of $\fF_s$ and its law coincides with that of $T_{0,\zzz}$. The same remains true if $s$ is a stopping time. Thus, 
$\tau(i,\zzz)- \tau(i-1,\zzz)$ is independent of 
$\fF_{\tau(i-1,\zzz)}$ and its law coincides with $T_{0,\zzz}$. 
Consequently, by the bound~\eqref{e:stopping_time_lower_bound}
\begin{equs}
\PP\big(\tau(i,z) < 1|\fF_{\tau(i-1,\zzz)} \big)
&\leq
\PP\big(\tau(i,\zzz)-\tau(i-1,\zzz)< 1|\fF_{\tau(i-1,\zzz)} \big)
\\
&=
\PP\big(\tau(i,\zzz)-\tau(i-1,\zzz)< 1\big)
\\
&=
\PP\big(T_{0,\zzz} < 1 \big)
\leq
\frac{1}{100} \;.
\end{equs}
As a result, we get
\begin{equs}
\PP(\tau(i,\zzz) < 1)
&=
\PP\big(\tau(i,\zzz) < 1|\tau(i-1,\zzz) < 1 \big)
\PP\big(\tau(i-1,\zzz) < 1\big) \\
&\leq \frac{1}{100}\,\PP(\tau(i-1,\zzz) < 1) \leq \frac{1}{100^i} \;.
\end{equs}
Let
\begin{equ}
N_{z}\eqdef\inf\{i\in\N_+\,|\,\tau(i,\zzz)\geq 1\}\;.
\end{equ}
Then we have
\begin{equ}
\PP(N_{z} \geq i) \leq \PP(\tau(i-1,\zzz) < 1) \leq \frac{1}{100^{i-1}} \;.
\end{equ}
Iterating the bound stated in Proposition~\ref{pr:deterministic_main} $N_{z}$ times, we get
\begin{equ}
\lVert\exp(\gamma\langle\bigcdot\rangle_\ell)D(t)\rVert^p_{L^p(e_{2t,\zzz})}
\leq
\exp\big(N_{z}/3-p\,t\big)\,
\lVert\exp(\gamma\langle\bigcdot\rangle_\ell)D(0)\rVert_{L^p(e_{0,\zzz})}^p
\end{equ}
for all $t\in[0,1]$. Next, we note that for $c\in(0,100)$ we have
\begin{equ}
	\EE c^{N_{z}}
	=
	\sum_{i = 1}^{\infty} c^i\, \PP(N_{z} = i)
	\leq
	\sum_{i = 1}^{\infty} \frac{c^i}{100^{i-1}} = \frac{c}{1-c/100} \;.
\end{equ}
Applying the above estimate with $c=\exp(1/3)$ and noting that $\frac{c}{1-c/100}\leq \exp(1/2)$ gives
\begin{equ}
\label{e:exp_decay2_weight}
\EE\lVert\exp(\gamma\langle\bigcdot\rangle_\ell)D(t)\rVert^p_{L^p(e_{2t,\zzz})}
\leq
\exp(1/2-p\,t)\,
\lVert\exp(\gamma\langle\bigcdot\rangle_\ell)D(0)\rVert_{L^p(e_{0,\zzz})}^p
\end{equ}
for all $t\in[0,1]$.

Now, we exploit the fact that~\eqref{e:exp_decay2_weight} holds true for all $\zzz\in\R^3$ to derive a similar estimate with a polynomial weight. To this end, we use Fubini's theorem together with~\eqref{eq:bound_weight_e} and obtain
\begin{equ}
\EE\lVert\exp(\gamma\langle\bigcdot\rangle_\ell)D(t)\rVert^p_{L^p(\rho_\fc)}
\leq
\exp(2/3-p\,t)\,
\lVert\exp(\gamma\langle\bigcdot\rangle_\ell)D(0)\rVert_{L^p(\rho_\fc)}^p
\end{equ}
for all $t\in [0, 1]$ and $\fc\in(0,\fa^2]$. Exploiting the Markov property we iterate the above bound and arrive at
\begin{equ}
\EE\lVert\exp(\gamma\langle\bigcdot\rangle_\ell)D(t)\rVert^p_{L^p(\rho_\fc)}
\leq
\exp(2/3-p\,t/3)\,
\lVert\exp(\gamma\langle\bigcdot\rangle_\ell)D(0)\rVert_{L^p(\rho_\fc)}^p
\label{e:exp_decay_general}
\end{equ}
for all $t\in \R_\geq$ and $\fc\in(0,\fa^2]$.
\end{proof}

\subsection{Proof of Proposition~\ref{pr:deterministic_main}}\label{sec:proof_ergodicty_aux}

To complete the proof of Theorem~\ref{thm:linearised}, it remains to establish Proposition~\ref{pr:deterministic_main}.  We begin by addressing the second renormalisation constant $C^{(2)} =\Ctwo$ using an exponential transform trick, in the spirit of \cite{HL15, JP23}, as formulated in the following lemma.
Throughout this section, we omit the dependence on $\eps\in(0,1]$ and $\ell\in \N_+$ in subscripts to lighten the notation. However, we continue specifying the uniformity with respect to $\eps,\ell$ in the statements of the results.

\begin{lem}
\label{lem:exp_transform_mod}
Suppose that $D$ solves~\eqref{e:linearisation1_mod}. Then for $0\leq s \leq t $,
\begin{equ}
 \hat{D}_s(t) = \exp\bigl((t-s)+3\lambda\oudz_{s}(t)+\gamma\langle\bigcdot\rangle_\ell\bigr) D(t)
\end{equ}
solves  
\begin{equ}\label{e:D1_mod}
\bigl(\d_t  - \Laplace \bigr)\hat{D}_s =
- \bigl(V_s^{(1)}+V_s^{(2)} \bigr) \hat{D}_s
- U_{s}  \cdot \grad \hat{D}_s
\;,
\end{equ}
where $V_s^{(1)}\eqdef \VVmodoneA + \VVmodoneB + \VVmodoneC$ with
\begin{equs}
\VVmodoneA &\eqdef
 6\lambda \ou_{ s} \Psi_s\;,
\\
\VVmodoneB &\eqdef
3\lambda\oudz_{s}
+ 6\lambda^2 \ou_{ s} \outz_{ s}
- 9\lambda^2(|\grad \oudz_{s} |^{2}  - C^{(2)}_s)\;,
\\
\VVmodoneC &\eqdef
-3\lambda\big(C^{(1)}-C^{(1)}_s \big)
+ 9\lambda^2 (C^{(2)} - C^{(2)}_s)
+\gamma\Delta\langle\bigcdot\rangle_\ell-\gamma^2(\nabla\langle\bigcdot\rangle_\ell)^2
\;,
\\
\VVmodtwo &\eqdef 3\lambda\PSImod^2\;,
\\
U_{s}&\eqdef
6\lambda \grad \oudz_{s}+2\gamma\nabla\langle\bigcdot\rangle_\ell\;.
\end{equs}
\end{lem}

\begin{proof}
This is a straightforward calculation.
\end{proof}

We turn to the analysis of~\eqref{e:D1_mod}. A naïve use of the coming down from infinity 
property for $\Psi_s$ would lead to an estimate of $\Psi_s^2$ of order $(t-s)^{-1}$, which
 is non-integrable at time $s$. This term however has ``the right sign'', 
 so it can easily be cured by the following simple comparison argument using the 
 Feynman--Kac formula.

\begin{lem}\label{lem:comparison_mod} Let $s\geq 0 $. Suppose that $\hat{D}_s$ solves~\eqref{e:D1_mod} and $\DDDmodone$ solves
\begin{equ}\label{e:simplified_equ_mod}
\big(\d_t  - \Laplace \big) \DDDmodone =
-\VVmodone\DDDmodone
- U_{s} \cdot \grad \DDDmodone
\;,
\end{equ} 
with the initial condition $\DDDmodone(s,x)=|\hat{D}_s(s,x)|$. Then, we have 
$ |\hat{D}_s(t,x)| \leq \DDDmodone(t,x)$
for all $t\geq s$, $\eps\in (0,1]$, $\ell\in \N_+$, $\lambda,\gamma\geq 0 $ and $x \in \T^3_\ell$.
\end{lem}
\begin{proof}
Without loss of generality, we assume $s=0$ and drop the dependence on $s$. By the Feynman--Kac formula, for any $t>0$ and $x\in \T^3_\ell$ we have 
\begin{equ}
	\hat{D}(t,x)
	=
	\EE_x\Biggl(\exp \Big(- \int_0^t (V^{(1)}+V^{(2)})(t-u, X_u)\,\md u \Big) \hat D(0, X_t)\Biggr)\;,
	\end{equ} 
where the expectation $\EE_x$ is taken with respect to the law of a stochastic process $(X_r)_{r\geq 0}$ starting at $X_0 =x $ and satisfying
\begin{equ}
	\md X_r  = - U(r,X_r)\, \md r + \sqrt{2}\,\md W_r \;,
\end{equ} 
for a Brownian motion $(W_r)_{r\geq 0 }$ with $W_0=0$. Since $V^{(2)}$ is non-negative, we have
\begin{equ}
	|\hat{D}(t,x)|
	\leq
	\EE_x\Biggl(\exp\Big(- \int_0^t  V^{(1)}(t-u, X_u) \,\md u\Big)|\hat D(0, X_t)|\Biggr)
	=
	\hat{D}^{(1)}(t,x)\;, 
\end{equ} 
where we used the Feynman--Kac formula again in the equality.
\end{proof}

\begin{defn}
For $\delta\geq0$, $p\geq1$, $\zzz\in\R^3$ and $0\leq s< u<\infty$ define the norm
\begin{equ}
\|D\|_{\xX(s,u)} \eqdef
\sup_{t\in (s,u]} \|D(t)\|_{L^p(e_{\delta + t,\zzz})}
\vee
\lambda_\star^{1/4} \sup_{\substack{t\in (s,u]\\\alpha\in[0,\alpha_\star]}} (t-s)^{\alpha/2+\kappa}
\|D(t)\|_{\Besovt{\alpha}} \;,
\end{equ}
where $\alpha_\star=3/2-19\kappa$.
\end{defn}

The main step of our analysis can be formulated as the following proposition, which is a 
modification of the argument in \cite{HL15}. Recall that the stopping time 
$T_{s,\zzz} = T_{\eps,\ell,s,\zzz}$ was defined in \eqref{e:stopping_time}.

\begin{prop}
\label{pr:fixed_point_mod}
Fix $p\geq1$. Suppose that $\DDDmodone$ solves~\eqref{e:simplified_equ_mod} in the time interval $[s,T_{s,\zzz}]$. Then there exists $\lambda_\star\in(0,1]$ such that
\begin{equ}
\|\DDDmodone\|_{\xX(s,T_{s,\zzz})}
\leq
\exp(1/(4p))\, \|\DDDmodone(s)\|_{L^p(e_{\delta+s,\zzz})} \;
\label{e:fixed_point1_mod}
\end{equ}
for all $\lambda\in(0,\lambda_\star]$, $\eps\in(0,1]$, $\ell\in\N_+$, $s\geq 0$, $\gamma\in[0,\lambda_\star]$ and $\Phi_{\eps,\ell}$  solving~\eqref{eq:phi4_S} with an adapted and continuous $S$ in a unit ball of $L^\infty(\R_\geq\times\T_\ell^3)$ and arbitrary initial data.
\end{prop}
\begin{proof}
Duhamel's formula yields
\begin{equs}\label{e:Duhamel_mod}
\DDDmodone(t) &= e^{(t -s)\Laplace}\DDDmodone(s)
\\
&\quad -
\int_s^t e^{(t - r)\Laplace} \big((\VVmodoneA + \VVmodoneB + \VVmodoneC)  \DDDmodone
+
U_{s}\cdot \grad \DDDmodone \big)(r)\,\md r \;.
\end{equs}
We shall bound separately each of the five terms  appearing in the right-hand side of this expression 
and prove that there exist a universal constant $c>0$, depending only on $p$ and the constant 
$\eta$ appearing in \eqref{e:stopping_time}, such that
\begin{equs}\label{eq:contraction_proof_main}
\|\DDDmodone\|_{\xX(s,T_{s,\zzz})}
&\leq
c\,\lambda_\star^{1/4}\,\|\DDDmodone\|_{\xX(s,T_{s,\zzz})}
\\
&\quad +
\bigl(\exp(1/(6p))\vee c\,\lambda_\star^{1/4}\bigr)\,\|\DDDmodone(s)\|_{L^p(e_{\delta +s,\zzz})}\;.
\end{equs}
The claim then follows  by choosing $\lambda_\star\in(0,1]$ small enough.

\step{Initial data contribution:}
By Lemmas~\ref{le:smoothing_heat} and~\ref{le:besov_embedding} we have 
\begin{equs}
\|e^{(t -s)\Laplace} \DDDmodone(s)\|_{\Besovt{\alpha}}
&\lesssim
(t -s)^{-(\alpha+\kappa)/2} \|\DDDmodone(s)\|_{\Besovt[e_{\delta +s,\zzz}]{-\kappa}} \\
&\lesssim
(t -s)^{-(\alpha+\kappa)/2} \|\DDDmodone(s)\|_{L^p(e_{\delta +s,\zzz})} \;.
\end{equs}
By the estimate~\eqref{eq:heat_almost_contraction} on the heat kernel, we furthermore have
\begin{equs}
\|e^{(t -s)\Laplace} \DDDmodone(s)\|_{L^p(e_{\delta + t,\zzz})}
&\leq
\exp(1/(6p))\,\|\DDDmodone(s)\|_{L^p(e_{\delta +t,\zzz})}\\
&\leq
\exp(1/(6p))\,\|\DDDmodone(s)\|_{L^p(e_{\delta +s,\zzz})}\;.
\end{equs}
Combining these bounds, we conclude that
\begin{equ}
\|e^{(\bigcdot -s)\Laplace} \DDDmodone(s)\|_{\xX(s,T_{s,\zzz})}
\leq
\bigl(\exp(1/(6p))\vee c\,\lambda_\star^{1/4}\bigr)\,\|\DDDmodone(s)\|_{L^p(e_{\delta +s,\zzz})}\;.
\end{equ}

\step{Term involving \TitleEquation{\VVmodoneA}{V1a}:}\label{step:V1a}
First, observe that this term, namely
\begin{equ}
\int_s^t e^{(t-r)\Laplace} \VVmodoneA(r) \DDDmodone(r)\,\md r\;,
\end{equ}
can be written as
\begin{equ}\label{eq:proof_A}
6\lambda\int_s^t e^{(t-r)\Laplace}\big(\ou_{s} \rpara (\Psi_s \DDDmodone) \big)(r)\,\md r
+
6\lambda\int_s^t e^{(t-r)\Laplace}\big(\ou_{s}\pareq (\Psi_s \DDDmodone) \big)(r)\,\md r\,.
\end{equ}
For the first term in~\eqref{eq:proof_A} we get
\begin{equs}
\biggl\|\int_s^t &e^{(t-r)\Laplace}\big(\ou_{s} \rpara (\Psi_s \DDDmodone) \big)(r)\,\md r\biggr\|_{\Besovt{\alpha}} \label{e:boundWeight}
\\
&\lesssim
\int_s^t (t-r)^{- 1/4 -\alpha/2 - \kappa} \|\big(\ou_{s} \rpara (\Psi_s \DDDmodone) \big)(r) \|_{\Besov{-1/2-2\kappa}{p, p}(e_{\delta + t,\zzz})}\,\md r
\\
&\lesssim
\int_s^t (t-r)^{- 1/4 -\alpha/2 - 9\kappa} \|\big(\ou_{s} \rpara (\Psi_s \DDDmodone) \big)(r) \|_{\Besov{ - 1/2 - 2\kappa}{p, p}(w^4_{\zzz} e_{\delta + r,\zzz})}\,\md r\;,
\end{equs} 
where we have used Lemma~\ref{le:smoothing_heat} and then Lemmas~\ref{le:weights_time} and~\ref{lem:monotone}.
Note that for $r\in[s,T_{s,z}]$ we have $\|\ou_{s}(r)\|_{\cC^{-\f12-\kappa}(w_{\zzz})} \lesssim1$ and from Lemma~\ref{le:psi}
\begin{equs}
	\|\Psi_s(r)\|_{L^\infty(w_{\zzz}^3)} &\lesssim \lambda^{-1/2}\,(r - s)^{-1/2}\;,
	\\
	\|\Psi_s(r)\|_{\cC^{1/2 + 4\kappa}(w^4_z)} &\lesssim \lambda^{-1/2}\,(r - s)^{-3/4 - 2\kappa}\;.
\end{equs}
Hence, by Lemma~\ref{le:paraproduct} we have 
\begin{equs}
\|\big(\ou_{s} \rpara (\Psi_s\DDDmodone) \big)(r) \|_{\Besov{ - 1/2 - 2\kappa}{p, p}(w^4_{\zzz} e_{\delta + r,\zzz})}
&\lesssim
\|\ou_{s}\|_{\cC^{-1/2 - \kappa}(w_{\zzz})} \|\Psi_s\|_{L^\infty(w^3_{\zzz})} \|\DDDmodone(r) \|_{L^p(e_{\delta + r,\zzz})}
\\
&\lesssim \lambda^{-1/2}\,(r - s)^{-1/2} \|\DDDmodone\|_{\xX(s,T_{s,\zzz})}\;.
\end{equs}
Using the fact that $-1/4 -\alpha/2 - 9\kappa>-1$ and
\begin{equ}
	\int_s^t (t-r)^{-1/4 -\alpha/2 - 9\kappa}\,(r - s)^{-1/2}\lesssim (t-s)^{1/4-\alpha/2-9\kappa}\,,
\end{equ}
we conclude that
\begin{equ}
	\biggl\|\int_s^t e^{(t-r)\Laplace}\big(\ou_{s} \rpara (\Psi_s \DDDmodone) \big)(r)\,\md r\biggr\|_{\Besovt{\alpha}}
	\lesssim\lambda^{-\f12}\,(t-s)^{\f14-\f\alpha2-9\kappa}\,\|\DDDmodone\|_{\xX(s,T_{s,\zzz})}
\end{equ}
for all $\alpha\in[0,\alpha_\star]$. Since by Lemma~\ref{le:besov_embedding} the embedding $\Besov{\kappa}{p, p}(e_{\delta + t,\zzz}) \hookrightarrow L^p(e_{\delta + t,\zzz})$ is continuous, it follows that
\begin{equ}
\lambda\,\left\|\int_s^{\bigcdot} e^{(\bigcdot-r)\Laplace}\big(\ou_{s} \rpara (\Psi_s \DDDmodone) \big)(r)\,\md r\right\|_{\xX(s,T_{s,\zzz})}
\lesssim
\lambda^{1/4}\,\|\DDDmodone\|_{\xX(s,T_{s,\zzz})}\;.
\end{equ}
For the second term, following a similar procedure we get
\begin{equs}
\biggl\|\int_s^t &e^{(t-r)\Laplace}\big(\ou_{s} \pareq (\Psi_s \DDDmodone) \big)(r)\,\md r\biggr\|_{\Besovt{\alpha}}\label{e:boundWeight2}
\\
&\lesssim
\int_s^t (t-r)^{-\alpha/2-\kappa} \|\big(\ou_{s} \pareq (\Psi_s \DDDmodone) \big)(r) \|_{\Besov{2\kappa}{p, p}(e_{\delta + t,\zzz})}\,\md r
\\
&\lesssim
\int_s^t (t-r)^{ - \alpha/2- 11\kappa} \|\big(\ou_{s} \pareq (\Psi_s \DDDmodone) \big)(r) \|_{\Besov{2\kappa}{p, p}(w^5_{\zzz} e_{\delta + r,\zzz})}\,\md r\;.
\end{equs}
By Lemma~\ref{le:paraproduct} we have 
\begin{equs}
\|\big(\ou_{s} &\pareq (\Psi_s \DDDmodone) \big)(r) \|_{\Besov{2\kappa}{p, p}(w^5_{\zzz} e_{\delta + r,\zzz})}
\lesssim
\|\ou_{s}(r) \|_{\cC^{-\f12 - \kappa}(w_{\zzz})} \|(\Psi_s \DDDmodone)(r) \|_{\Besovt[w^4_{\zzz} e_{\delta + r,\zzz}]{1/2 + 3\kappa}}
\\
&\lesssim
\|(\Psi_s \DDDmodone)(r) \|_{\Besovt[w^4_{\zzz} e_{\delta + r,\zzz}]{1/2 + 3\kappa}}\;.
\end{equs}
We use Lemma~\ref{le:paraproduct} again to get
\begin{equs}
\,&\|(\Psi_s \DDDmodone)(r) \|_{\Besovt[w^4_{\zzz} e_{\delta + r,\zzz}]{1/2 + 3\kappa}}
\\
&\leq
\|(\Psi_s \para \DDDmodone)(r) \|_{\Besovt[w^4_{\zzz} e_{\delta + r,\zzz}]{1/2 + 3\kappa}} + \|(\Psi_s \rpareq \DDDmodone)(r) \|_{\Besovt[w^4_{\zzz} e_{\delta + r,\zzz}]{1/2 + 3\kappa}}
\\
&\leq
\|\Psi_s(r)\|_{L^\infty(w^4_z)} \|\DDDmodone(r) \|_{\Besovt[e_{\delta + r,\zzz}]{1/2 + 4\kappa}}
+
\|\Psi_s(r)\|_{\cC^{1/2 + 4\kappa}(w^4_z)} \|\DDDmodone(r) \|_{L^p(e_{\delta + r,\zzz})}
\end{equs}
Therefore,
\begin{equs}
	\,&\|(\Psi_s \DDDmodone)(r) \|_{\Besovt[w^4_{\zzz} e_{\delta + r,\zzz}]{1/2 + 3\kappa}}
	\\
	&\lesssim
	\lambda^{-1/2}\,
	(r -s)^{-1/2}\, \|\DDDmodone(r)\|_{\Besovt[e_{\delta + r,\zzz}]{1/2 + 4\kappa}}
	+
	\lambda^{-1/2}\,
	(r -s)^{-3/4 - 2\kappa}\, \|\DDDmodone(r)\|_{L^p(e_{\delta + r,\zzz})}
	\\
	&\lesssim
	\lambda^{-1/2}\,\lambda_\star^{-1/4}\,(r -s)^{-3/4-3\kappa}
	\|\DDDmodone\|_{\xX(s,T_{s,\zzz})}\,.
\end{equs} 
Using the fact that $-\alpha/2-11\kappa>-1$ and $-3/4-3\kappa>-1$ we conclude that
\begin{equs}
\biggl\|\int_s^t &e^{(t-r)\Laplace}\big(\ou_{s} \pareq (\Psi_s \DDDmodone) \big)(r)\,\md r\biggr\|_{\Besovt{\alpha}}
\\
&\lesssim  
\lambda^{-1/2}\,\lambda_\star^{-1/4}\,(t -s)^{1/4-\alpha/2-14\kappa}
\|\DDDmodone\|_{\xX(s,T_{s,\zzz})}\,.
\end{equs}
Combining it with the embedding $\Besov{\kappa}{p, p}(e_{\delta + t,\zzz}) \hookrightarrow L^p(e_{\delta + t,\zzz})$, it follows that
\begin{equ}
\lambda\left\|\int_s^{\bigcdot}  e^{(\bigcdot-r)\Laplace}\big(\ou_{s} \pareq (\Psi_s \DDDmodone) \big)(r)\,\md r\right\|_{\xX(s,T_{s,\zzz})}
\lesssim
\lambda^{1/4}\,\|\DDDmodone\|_{\xX(s,T_{s,\zzz})} \;,
\end{equ} whence we conclude that
\begin{equ}
\left \| \int_s^{\bigcdot} e^{(\bigcdot-r)\Laplace} \VVmodoneA(r) \DDDmodone(r)\,\md r \right\|_{\xX(s,T_{s,\zzz})} \lesssim \lambda_\star^{1/4} \|\DDDmodone\|_{\xX(s,T_{s,\zzz})} \;.
\end{equ}

\step{Term involving \TitleEquation{\VVoneB}{V1b}:} The bound of the term involving $\VVoneB$ is obtained similarly to the argument in \ref{step:V1a}. Recall that 
\begin{equ}
\VVmodoneB \eqdef 3\lambda\oudz_{s}
+ 6\lambda^2 \ou_{s} \outz_{s}
- 9\lambda^2(|\grad \oudz_{s} |^{2}  - C^{(2)}_s)\;.
\end{equ} Therefore, for any $r\in[s,T_{s,z}]$ we have 
\begin{equ}[e:boundVlambda]
    \| \VVmodoneB (r) \|_{\cC^{-\f12-\kappa}(w_z^2)} \lesssim \lambda\;. 
\end{equ} 
Then, proceeding as in \eqref{e:boundWeight}, we obtain
\begin{equs}
    \biggl\|\int_s^t &e^{(t-r)\Laplace} \big(\VVmodoneB(r)  \rpara \DDDmodone(r)\big) \,\md r\biggr\|_{\Besovt{\alpha}}\\
    &\lesssim \int_s^t (t-r)^{-1/4 - \alpha/2 - 5\kappa }   \| \VVmodoneB(r)  \rpara \DDDmodone(r) \|_{\bB^{-1/2 - 2\kappa}_{p,p}(w_ze_{\delta + r,z })} \, \md r\\
    &\lesssim \int_s^t \lambda (t-r)^{-1/4 - \alpha/2 - 5\kappa } \| \DDDmodone(r)\|_{L^p(e_{\delta +r,z})} \,\md r \lesssim \lambda \|\DDDmodone\|_{\xX(s,T_{s,\zzz})}\;. 
\end{equs} 
On the other hand, we have similarly to \eqref{e:boundWeight2}
and using again \eqref{e:boundVlambda}
\begin{equs}
\biggl\|\int_s^t &e^{(t-r)\Laplace}\big(\VVmodoneB(r) \pareq \DDDmodone (r)\big)\,\md r\biggr\|_{\Besovt{\alpha}}
\\
&\lesssim
\lambda\int_s^t (t-r)^{ - \alpha/2- 9\kappa/2} \|\DDDmodone (r) \|_{\Besov{1/2+2\kappa}{p, p}(e_{\delta + r,\zzz})}\,\md r
\\
&\lesssim 
\lambda^{3/4}\|\DDDmodone (r) \|_{\xX(s,T_{s,\zzz})}\;.
\end{equs} Combining these estimates, we obtain the bound 
\begin{equ}
    \left\|\int_s^{\bigcdot} e^{(\bigcdot-r)\Laplace}\VVmodoneB(r)  \DDDmodone(r)\,\md r\right\|_{\xX(s,T_{s,\zzz})} \lesssim \lambda_\star^{1/4} \| \DDDmodone\|_{\xX(s,T_{s,\zzz})}\;.
\end{equ}

\step{Term involving \TitleEquation{\VVoneC}{V1c}:}  By Lemmas~\ref{le:smoothing_heat}, \ref{le:besov_embedding} and~\ref{le:renormalisation_diff}, we have 
\begin{equs}
    \biggl\|\int_s^t &e^{(t-r)\Laplace}\VVmodoneC(r)  \DDDmodone(r)\,\md r\biggr\|_{\Besovt{\alpha}}
    \\
    &\lesssim \int_s^t  (t-r)^{-(\alpha+\kappa)/2}\| \VVmodoneC(r) \|_{L^\infty} \| \DDDmodone(r) \|_{L^p(e_{\delta+t,z})} \, \md r 
    \\
    &\lesssim \int_s^t (\lambda(r-s)^{-1/2} + \lambda^2(r-s)^{-\kappa} + \gamma )(t-r)^{-(\alpha+\kappa)/2}  \| \DDDmodone(r) \|_{L^p(e_{\delta+r,z})} \, \md r 
    \\
    & \lesssim  \lambda_\star \| \DDDmodone\|_{\xX(s,T_{s,\zzz})} \;,
\end{equs} where we have used the fact that $|\nabla\langle\bigcdot\rangle_\ell|,|\Delta\langle\bigcdot\rangle_\ell|\lesssim 1$ uniformly for $\ell \in \N_+$ and $\lambda,\gamma\leq \lambda_\star \leq 1 $. By the embedding  $\Besov{\kappa}{p, p}(e_{\delta + t,\zzz}) \hookrightarrow L^p(e_{\delta + t,\zzz})$, we have 
\begin{equ}
    \left\|\int_s^{\bigcdot} e^{(\bigcdot-r)\Laplace}\VVmodoneC(r)  \DDDmodone(r)\,\md r\right\|_{\xX(s,T_{s,\zzz})} \lesssim \lambda_\star^{1/4} \| \DDDmodone\|_{\xX(s,T_{s,\zzz})}\;.
\end{equ}

\step{Term involving \TitleEquation{U_{s}}{U}:} Similarly to before, we have
\begin{equs}
	\biggl\| \int_s^t &e^{(t - r)\Laplace} \big(U_{s}\cdot \grad\DDDmodone \big)(r)\,\md r \biggr\|_{\bB^{\alpha}_{p, p}(e_{\delta + t,\zzz})}
	\\
	&\lesssim  
	\int_s^t (t - r)^{-3\kappa/2-\alpha/2}\,
	\|(U_{s} \cdot \grad\DDDmodone)(r)\|_{\bB^{-3\kappa}_{p, p}(e_{\delta + t,\zzz})}\,\md r
	\\
	&\lesssim
	\int_s^t (t - r)^{-5\kappa/2-\alpha/2}\,
	\|(U_{s} \cdot \grad\DDDmodone)(r) \|_{\bB^{-3\kappa}_{p, p}(w_{\zzz}e_{\delta + r,\zzz})}\,\md r\;.
\end{equs} 
By Lemmas~\ref{le:besov_embedding}, \ref{le:besov_derivative} and~\ref{le:paraproduct},
\begin{equs}
    \|(U_{s}\cdot \grad\DDDmodone)(r)\|_{\bB^{-3\kappa}_{p, p}(w_{\zzz}e_{\delta + r,\zzz})}  
    &\lesssim 
    \| U_{s}(r)\|_{\cC^{-2\kappa}(w_z)} 
    \|(\grad\DDDmodone)(r)\|_{\bB^{3\kappa}_{p,p}(e_{\delta + r,\zzz})}
    \\
    &\lesssim
    (\lambda \|  \oudz_{s}(r) \|_{\cC^{1-2\kappa}(w_z) }+ \gamma)\, 
    \|\DDDmodone(r) \|_{\bB^{1+3\kappa}_{p,p}(e_{\delta + r,\zzz})}
    \\
    &\lesssim (\lambda + \gamma )\, \|  \DDDmodone(r) \|_{\bB^{1+3\kappa}_{p,p}(e_{\delta + r,\zzz})}
    \\
    &\lesssim \lambda_\star^{3/4}\, (r-s)^{-1/2-5\kappa/2} \| \DDDmodone\|_{\xX(s,T_{s,\zzz})}\;,
\end{equs} 
where we have used the fact that $|\nabla\langle\bigcdot\rangle_\ell|\lesssim 1$ uniformly over $\ell\in\N_+$ and $\lambda,\gamma \leq \lambda_\star\leq 1 $. Thus,
\begin{equs}
\biggl\|\int_s^t &e^{(t - r)\Laplace} \big(U_{s}\cdot \grad\DDDmodone \big)(r)\,\md r \biggr\|_{\bB^{\alpha}_{p, p}(e_{\delta + t,\zzz})}
\\
&\lesssim \lambda_\star^{3/4}\,(t-s)^{1/2-\alpha/2-5\kappa}\,\|\DDDmodone\|_{\xX(s,T_{s,\zzz})}\,.
\end{equs} 
This bound yields,
\begin{equ}
\left\|\int_s^{\bigcdot} e^{(\bigcdot - r)\Laplace} \big(U_{s}\cdot \grad\DDDmodone \big)(r)\,\md r\right\|_{\xX(s,T_{s,\zzz})}
\lesssim
\lambda_\star^{1/4}\, \|\DDDmodone\|_{\xX(s,T_{s,\zzz})} \;.
\end{equ}
Collecting all these bounds gives~\eqref{eq:contraction_proof_main} 
and completes the proof of the proposition.
\end{proof}

Now we are ready to prove Proposition~\ref{pr:deterministic_main}.
\begin{proof}[Proof of Proposition~\ref{pr:deterministic_main}]
By Lemma~\ref{lem:exp_transform_mod} for $0\leq s\leq t $,
\begin{equ}
\hat{D}_s(t) = \exp((t-s)+3\lambda\oudz_{s}(t)+\gamma\langle\bigcdot\rangle_\ell) D(t)
\end{equ}
solves~\eqref{e:D1_mod}. By Lemma~\ref{lem:comparison_mod} we have
\begin{equ}
\lvert\exp((t-s)+3\lambda\oudz_{s}(t,x)+\gamma\langle x\rangle_\ell) D(t,x)\rvert
\leq
|\DDDmodone(t, x)| \;,
\end{equ}
where $\DDDmodone$ solves \eqref{e:simplified_equ_mod} with $\DDDmodone(s) = D(s)$. Therefore, for all $s\in[0,1]$ and $t\in[s,T_{s,\zzz}]$ we have
\begin{equ}
\label{e:comparison2}
\lVert\exp(3\lambda\oudz_{s}(t)+\gamma\langle \bigcdot\rangle_\ell) D(t)\rVert_{L^p(e_{s+t,\zzz})}
\leq
e^{-(t-s)}\,\|\DDDmodone(t)\|_{L^p(e_{s+t,\zzz})}
\end{equ}
as well as
\begin{equ}
\|\DDDmodone(t)\|_{L^p(e_{s+t,\zzz})}
\leq
\exp(1/(4p))\, \lVert\exp(\gamma\langle\bigcdot\rangle_\ell)D(s)\rVert_{L^p(e_{2s,\zzz})} \;,
\end{equ}
by Proposition~\ref{pr:fixed_point_mod} applied with $\delta=s$. Consequently,
\begin{equ}
\|\exp(3\lambda\oudz_{s}(t)+\gamma\langle\bigcdot\rangle_\ell) D(t) \|_{L^p(e_{s+t,\zzz})}
\leq \exp\bigl(1/(4p)-(t-s)\bigr)\,\lVert\exp(\gamma\langle\bigcdot\rangle_\ell)D(s)\rVert_{L^p(e_{2s,\zzz})} \,.
\end{equ}
Then, by H\"older's inequality we get that
\begin{equs}
\,&\lVert\exp(\gamma\langle\bigcdot\rangle_\ell)D(t)\rVert_{L^p(e_{2t,\zzz})}
\\
&\leq
\lVert\exp(-3\lambda\oudz_{s}(t))\, e^{-\fa(t-s)\bracket{\bigcdot - \zzz}^{1/2}}\rVert_{L^{\infty}}\,
\lVert\exp(3\lambda\oudz_{s}(t)+\gamma\langle\bigcdot\rangle_\ell)D(t)\rVert_{L^p(e_{s+t,\zzz})}\;.
\end{equs}
Note that for $s \leq t \leq T_{s,\zzz}$ by the heat flow estimate we have
\begin{equs}
\|\oudz_{s}(t)\|_{L^\infty(w_{\zzz})}
&\lesssim
\int_s^t \|e^{(\Laplace - 1)(t - r)}\, \oud_{s}(r)\|_{\cC^{-2\kappa}(w_{\zzz})}\,\md r
\\
&\lesssim
\int_s^t (t - r)^{-1/2-2\kappa} \|\oud_{s}(r)\|_{\cC^{-1-2\kappa}(w_{\zzz})}\,\md r \\
&\lesssim \label{e:oudz_bound}
(t-s)^{1/2-2\kappa}
\lesssim
(t-s)^{1/3}\;.
\end{equs}
By Young's inequality and the fact that $1 \leq \bracket{x}$ there exists a constant $C>0$ such that 
\begin{equs}
3|\oudz_{s}(t,x)|
&\leq
3\bracket{x-\zzz}^\kappa
\|\oudz_{s}(t)\|_{L^\infty(w_{\zzz})}
\\
&\leq
C^{2/3}\, (t-s)^{1/3} \bracket{x-\zzz}^\kappa
\leq C + \fa\,(t-s)\bracket{x-\zzz}^{1/2} \;.
\end{equs}
As a result, provided that $\lambda_\star $ is sufficiently small, for all $\lambda \in (0, \lambda_\star ]$ 
\begin{equ}\label{e:oudz_bound2}
\lVert\exp(3\lambda\oudz_{s}(t))\, e^{-\fa(t-s)\bracket{\bigcdot - \zzz}^{1/2}}\rVert_{L^{\infty}}
\leq
\exp(\lambda\,C)
\leq
\exp(1/(12p))\,.
\end{equ}
Therefore, we get
\begin{equ}
\lVert\exp(\gamma\langle\bigcdot\rangle_\ell)D(t)\rVert_{L^p(e_{2t,\zzz})}
\leq
\exp\bigl(1/(3p)-(t-s)\bigr)\,
\lVert\exp(\gamma\langle\bigcdot\rangle_\ell)D(s)\rVert_{L^p(e_{2s,\zzz})}\;.
\end{equ}
This finishes the proof.
\end{proof}

\section{Solution theory in infinite volume}\label{sec:global_solution_theory}

The aim of this section is to construct a solution to the dynamical $\Phi^4_3$ model on $\R_\geq \times \R^3$ for arbitrary initial data in $\cC^{-\frac{1}{2} - \kappa}(w)$, prove its uniqueness and to demonstrate that it satisfies all the properties stated in Theorem~\ref{thm:global_solution}. 

The core of our argument is presented in Sections~\ref{sec:a_priori_solution} and~\ref{sec:proof_markov_process_construction}. In Sections~\ref{sec:reg_structure}–\ref{sec:init_data}, we collect the necessary preliminary results: we construct a suitable regularity structure and extend the results of~\cite{HL18}. The main technicality here is that our equation for the difference between two solutions satisfies the Parabolic Anderson Model with a non-trivial ``noise" which is a modelled distribution instead of a symbol in the regularity structures. Consequently, estimating this modelled distribution requires the use of weighted norms, whereas no such weights are needed if the noise is merely a symbol.

We also use some auxiliary results from Appendices~\ref{sec:spacetime_localization} and~\ref{sec:stochastic_estimates}.  To obtain the improved a priori bounds stated in Lemma~\ref{le:apriori_bound_Linfty}, we use a generalisation of the space-time localisation estimate originally derived in~\cite{MW20}, which is formulated as Theorem~\ref{thm:spacetime_localization}. As mentioned below the statement of Theorem~\ref{thm:global_solution}, we apply this result using trees that incorporate contributions from the initial data. These trees are shown to be bounded in Lemma~\ref{lem:convergence_model_phi_deterministic}, where they are interpreted as (singular) modelled distributions with respect to the stationary model.

\subsection{Regularity structure}
\label{sec:reg_structure}

Let $(\bar\aA,\bar\tT, G)$ be the truncated regularity structure for the $\Phi^4_3$ equation constructed following the procedure in~\cite[Section~8.1]{Hai14}. We denote by
\begin{equ}
\bar\tT^\circ\eqdef\{{\color{blue}\Xi},\bout,\boud,\boutd,\bou,\boudd,\boutu,\boud\bXXX,\oneb,\boutz,\boudz,\bXXX,\ldots\}
\end{equ}
the set of linearly independent elements of $\bar\tT$ such that $\bar\tT = \Span\bar\tT^\circ$, where $\bXXX=(\bXXX^0,\bXXX^1,\bXXX^2,\bXXX^3)$. We use blue trees to denote the trees as abstract symbols appearing in the regularity structure, while the black trees denote the corresponding concrete functions \slash distributions. The elements of $\bar\tT^\circ$ are generated from $\{{\color{blue}\Xi},\oneb,\bXXX\}$ with the use of abstract integration $\tau\mapsto\iI(\tau)$ and multiplication $(\tau,\bar\tau)\mapsto\tau\bar\tau$ and we adopt the usual graphical notation of representing the integration by drawing an edge downward from the root and represent the multiplication by concatenation of trees at the root. The grading $|\bigcdot|\,:\, \bar\tT\to \bar\aA$ is a surjective map defined by the conditions 
\begin{equ}
	|{\color{blue}\Xi}|=-\frac52-\kappa\,,
	\quad
	|\oneb|=0\,,
	\quad
	|\bXXX|=1\,,
	\quad
	\iI(\tau)=|\tau|+2\,,
	\quad
	|\tau\bar\tau|=|\tau|+|\bar\tau|\,.
\end{equ}
In what follows, we work with a modified regularity structure with the noise symbol ${\color{blue}\Xi}$ removed and define  $\tT^\circ\eqdef\bar\tT^\circ\setminus\{{\color{blue}\Xi}\}$, $\aA=\bar\aA\setminus\{|{\color{blue}\Xi}|\}$ and $\tT=\Span\tT^\circ$. Given $\beta\in\aA$ we denote by $\tT_\beta$ the subset of $\tT$ consisting of $\tau$ such that $|\tau|=\beta$. For $I\subset\R$ we set $\tT_{I}\eqdef \oplus_{\beta\in\aA\cap I}\tT_\beta$. We denote by $\qQ_\beta: \tT \rightarrow \tT_\beta$ the projection onto $\tT_\beta$. Let $\|\bigcdot\|$ be a norm on $\tT$. Since $\tT$ is finite dimensional vector space, the choice of the norm does not affect the topology. We denote by $\|\tau\|_\beta$ the norm of $\qQ_\beta\tau$. We truncate the regularity structure so that $\aA\subset[|\bout|,3)$. Note that the choice of a truncation affects the conditions for the weights formulated in Assumption~\ref{ass:weights} below.

\subsection{Weights}

In this section, we present the set of assumptions that will constrain the weights used in our construction. We then demonstrate that it is possible to choose weights satisfying all these assumptions.

\begin{assumption}\label{ass:weights}
	The maps
	\begin{equ}
		w,\weight_\Pi, \weight_{\mathrm{S}}\in C(\R^3,(0,1])\,,
		\qquad
		\weight_{\mathrm{L}},\weight_{\mathrm{R}}\,:\,\{1,2\}\times \aA\to C(\R^{1+3},(0,1])\,,
	\end{equ}
	called weights, satisfy the conditions:
	\begin{itemize}		
		\item The weights $\weight_{\mathrm{L}},\weight_{\mathrm{R}}$ are decreasing functions of time and at time zero are bounded by $\weight_{\mathrm S}$. Moreover, there exists $C>0$ such that we have
		\begin{equ}[eq:W-0]\tag{W-0}
			\frac{1}{C}<\sup_{\substack{x,y \in\R^3\\|x-y| \leq 1}} \frac{\weight(x)}{\weight(y)}<C
		\end{equ}
		for $\weight\in\{w,\weight_\Pi, \weight_{\mathrm{S}}\}\cup\{\weight^{(i)}_{\mathrm{L}/\mathrm{R}}(\beta;t,\bigcdot) \,|\,i\in\{1,2\},\beta\in\aA,t\in[0,1]\}$\,. Here $\weight_{\mathrm{L}/\mathrm{R}}$ refers to either $\weight_{\mathrm{L}}$ or $\weight_{\mathrm{R}}$.
		\item There exists $\fc>0$ such that we have
		\begin{equ}[eq:W-1]\tag{W-1}
			\sup_{0\leq s<t\leq 1} \sup_{x\in \R^3} 
			\frac{(t-s)^{\fc/2}\,\overline\weight_{\mathrm{L}}(t,x)}{\weight_\Pi(x)^2\underline\weight_{\mathrm{R}}(s,x)} <\infty \,,
		\end{equ}
		where
		\begin{equ}
			\overline\weight_{\mathrm{L}}(t,x)\eqdef\sup_{i\in\{1,2\}}\sup_{\beta\in\aA}\weight^{(i)}_{\mathrm{L}}(\beta;t,x)\,,
			\qquad
			\underline\weight_{\mathrm{R}}(t,x)\eqdef\inf_{i\in\{1,2\}}\inf_{\beta\in\aA}\weight^{(i)}_{\mathrm{R}}(\beta;t,x)\,.
		\end{equ}
		\item For all $i\in\{1,2\}$, $\tau,\bar\tau\in\tT^\circ$, $t\in[0,1]$, $x\in\R^3$ and $k\in\N_0^{1+3}$ such that $|k|\leq 2$ we have
		\begin{equs}
			\label{eq:W-2}\tag{W-2}
			\weight^{(i)}_{\mathrm{L}}(|\mathcal{I}(\tau)|;t,x) &\leq \weight^{(i)}_{\mathrm{R}}(|\tau|;t,x)&\quad\text{if}&\quad\mathcal{I}(\tau)\neq0\,,
			\\
			\label{eq:W-3}\tag{W-3}
			\weight^{(i)}_{\mathrm{L}}(|\bXXX^k|;t,x) &\leq \weight_\Pi(x) 
			\weight^{(i)}_{\mathrm{R}}(|\tau|;t,x)&\quad\text{if}&\quad|\tau| \leq |k|-2\,,
			\\
			\label{eq:W-4}\tag{W-4}
			\weight^{(2)}_{\mathrm{L}}(|\bXXX^k|;t,x) &\leq \weight_\Pi(x) \weight^{(1)}_{\mathrm{R}}(|\tau|;t,x)\,,
			\\
			\label{eq:W-5}\tag{W-5}
			\weight^{(i)}_{\mathrm{R}}(|\tau\bar\tau|;t,x) &\leq \weight_{\mathrm{S}}(x)^2\weight^4_\Pi(x) 
			\weight^{(i)}_{\mathrm{L}}(|\tau|;t,x)\,.
		\end{equs}
		\item We have
		\begin{equ}\label{eq:W-6}\tag{W-6}
			\weight(x)/\weight(y) \lesssim \exp(|x-y|^2/8)
		\end{equ}
		uniformly over $\weight\in\{w,\weight_{\mathrm{S}}\}\cup\{\weight^{(i)}(\beta;t,\bigcdot) \,|\,i\in\{1,2\},\beta\in\aA,t\in[0,1]\}$ and $x,y\in\R^3$.
		\item With the same $\fc>0$ as above we have
		\begin{equ}[eq:W-7]\tag{W-7}
			\sup_{t\in[0,1]} \sup_{x\in \R^3}\frac{\overline\weight^{(2)}_{\mathrm{L}}(t,x)}{\weight_\Pi(x)\underline\weight^{(1)}_{\mathrm{L}}(t,x)}
			\vee
			\sup_{0\leq s<t\leq 1} \sup_{x\in \R^3} 
			\frac{(t-s)^{\fc/2}\,\overline\weight^{(1)}_{\mathrm{L}}(t,x)}{\weight_\Pi(x)\underline\weight^{(1)}_{\mathrm{L}}(s,x)} <\infty ,
		\end{equ}
		where
		\begin{equ}
			\overline\weight_\mathrm{L}^{(i)}(t,x)\eqdef\sup_{\beta\in\aA}\weight_\mathrm{L}^{(i)}(\beta;t,x)\,,
			\qquad
			\underline\weight_\mathrm{L}^{(i)}(t,x)\eqdef\inf_{\beta\in\aA}\weight_\mathrm{L}^{(i)}(\beta;t,x)\,.
		\end{equ}
	\end{itemize}	
\end{assumption}
\begin{rmk}
	The results stated in Section~\ref{sec:modelled_distributions} are true for all weights satisfying the above assumption, with the necessary conditions detailed in each theorem and lemma. In remaining part of Section~\ref{sec:global_solution_theory} and in Appendix~\ref{sec:stochastic_estimates} we work with weights fixed as in Lemma~\ref{lem:weights_solution} below.
\end{rmk}
\begin{rmk}
	We consider the initial data $\Phi(0)$ in the space $\cC^\eta(w)$ with $\eta=-\f12-\kappa$. We shall show that for every $t > 0$, the Da Prato--Debussche remainder $v(t)=\Phi(t)-\ou(t)$ has a finite $L^\infty(w^{1/2})$ norm, which, however, diverges at $t = 0$ at the rate $t^{-1/2}$. We will also prove that $L^\infty(\weight_{\mathrm{S}})$ norm of the remainder remains finite and blows up at the slower rate $\eta/2>-1/2$ at $t=0$. Thus, the temporal behaviour can be improved at the cost of employing a more rapidly decaying weight. We use the weights $\weight_{\mathrm{L}},\weight_{\mathrm{R}}$ in proving the uniqueness of solutions and their continuous dependence on the initial data. They appear in the norms that control the left- and right-hand sides of the equation governing the difference between two solutions. The weight $\weight_\Pi$ will be used to introduce a topology in the space of models.
\end{rmk}
\begin{rmk}
	Our weights are inverses of the weights that appear in~\cite{HL18}. The conditions~\eqref{eq:W-0}-\eqref{eq:W-5} are analogs of the conditions (W-0)-(W-5) therein.
\end{rmk}
\begin{rmk}
	When defining a seminorm in a~function space over $\R^{1+3}$ involving a~weight $\weight$ it is usual to demand that $\weight(t,x)/\weight(s,y)$ is bounded from below and above uniformly over $(t,x),(s,y)\in\R^{1+3}$. In the case of time-independent weights $w,\weight_\Pi, \weight_{\mathrm{S}}$, this condition is implied by~\eqref{eq:W-0}. Since it is not possible to satisfy this condition together with~\eqref{eq:W-1}, in the case of the time-dependent weights $\weight_{\mathrm{L}},\weight_{\mathrm{R}}$ we impose only the weaker condition~\eqref{eq:W-0} in addition demanding that these weights decrease in time. Note that~\eqref{eq:W-0} is essential for all the results stated in this section.
\end{rmk}
\begin{rmk}	
	The condition~\eqref{eq:W-1} plays a similar role to the estimate stated in Lemma~\ref{le:weights_time}. The conditions~\eqref{eq:W-2}-\eqref{eq:W-4} are needed to prove bounds for the integration operator $\kK^\pm$ stated in Theorem~\ref{thm:integration_K_exp}. We use~\eqref{eq:W-5} in the proof of the estimate for the product stated in Lemma~\ref{lem:multiplication_exp}. The condition~\eqref{eq:W-6} ensures that for times in the interval $[0,1]$ the weights are compatible with the decay property of the heat kernel and is used in Lemma~\ref{lem:integration_K} and Theorem~\ref{thm:integration_K_exp} about the integration operator $\kK^\pm$. We need~\eqref{eq:W-7} in the estimate for the projection $\qQ_{<\gamma}$ in Lemma~\ref{lem:embedding_modelled_distributions}.
\end{rmk}
\begin{lem}\label{lem:weights_solution}
	Recall that $\bar\kappa=\frac{1}{10}$, $\kappa=\bar\kappa^4$ and $w\eqdef\bracket{\bigcdot}^{-\bar\kappa^4}$. Let
	\begin{equ}
		b^{(1)}_{\mathrm{L}}=2+\bar\kappa^3,
		\quad
		b^{(1)}_{\mathrm{R}}=4,
		\quad
		b^{(2)}_{\mathrm{L}}=9+\bar\kappa^3,
		\quad
		b^{(2)}_{\mathrm{R}}=11\,.
	\end{equ}
	The weights
	\begin{equs}
		\weight_\Pi(\bigcdot) &\eqdef  \bracket{\bigcdot}^{-\bar\kappa^5},
		\quad
		\weight_{\mathrm{S}}(\bigcdot) \eqdef  \bracket{\bigcdot}^{-\bar\kappa^3},
		\quad
		\weight^{(i)}_{\mathrm{L/R}}(\beta;t,\bigcdot) &\eqdef  \bracket{\bigcdot}^{-\bar\kappa^2(\beta+b^{(i)}_{\mathrm{L/R}})}\, {\rm e}^{-t\bracket{\bigcdot}},
	\end{equs}
	satisfy Assumption~\ref{ass:weights} with $\fc=2\bar\kappa$.
\end{lem}
\begin{proof}
	It is evident that~\eqref{eq:W-0} and~\eqref{eq:W-6} hold true. Set $b_{\mathrm{L}}\eqdef b^{(1)}_{\mathrm{L}}\wedge b^{(2)}_{\mathrm{L}}$ and $b_{\mathrm{R}}\eqdef b^{(1)}_{\mathrm{R}}\vee b^{(2)}_{\mathrm{R}}$. Using the fact that $\sup\aA-\inf\aA\leq5$ we get
	\begin{equ}
		\frac{\overline\weight_{\mathrm{L}}(t,x)}{\weight_\Pi(x)^2\underline\weight_{\mathrm{R}}(s,x)}\lesssim 
		(t-s)^{-\bar\kappa^2(5-b_{\mathrm{L}}+b_{\mathrm{R}}+2\bar\kappa^3)}\,,
		\qquad
		\frac{\overline\weight_{\mathrm{L}}^{(1)}(t,x)}{\weight_\Pi(x)^2\underline\weight_{\mathrm{L}}^{(1)}(s,x)}\lesssim 
		(t-s)^{-\bar\kappa^2(5+2\bar\kappa^3)}\,.
	\end{equ}
	This implies \eqref{eq:W-1} and \eqref{eq:W-7} since
	\begin{equ}
		\bar\kappa^2(5-b_{\mathrm{L}}+b_{\mathrm{R}}+2\bar\kappa^3)
		\leq\fc/2\,,
		\qquad
		\bar\kappa^2(5+2\bar\kappa^3)\leq\fc/2\,.
	\end{equ}
	We observe that the conditions~\eqref{eq:W-2}-\eqref{eq:W-5} are satisfied if for all $\beta,\bar\beta\in\aA$ we have
	\begin{equs}
		\weight^{(i)}_{\mathrm{L}}(\beta;t,x) &\leq \weight_\Pi(x) 
		\weight^{(i)}_{\mathrm{R}}(\bar\beta;t,x)\quad\text{if}\quad \bar\beta\leq\beta-2\,,
		\\
		\weight^{(2)}_{\mathrm{L}}(\beta;t,x) &\leq \weight_\Pi(x) \weight^{(1)}_{\mathrm{R}}(\bar\beta;t,x)\,,
		\\
		\weight^{(i)}_{\mathrm{R}}(\bar\beta;t,x) &\leq \weight_{\mathrm{S}}(x)^2\weight^4_\Pi(x) 
		\weight^{(i)}_{\mathrm{L}}(\beta;t,x)\quad\text{if}\quad \bar\beta\geq\beta+\inf\aA\,.
	\end{equs}
	The above bounds are implied by
	\begin{equs}
		(\beta+b^{(i)}_{\mathrm{L}}) &\geq (\bar\beta+b^{(i)}_{\mathrm{R}})+\bar\kappa^3
		\quad\text{if}\quad \bar\beta\leq\beta-2\,,
		\\
		(\beta+b^{(2)}_{\mathrm{L}}) &\geq (\bar\beta+b^{(1)}_{\mathrm{R}})+\bar\kappa^3\,,	
		\\
		(\bar\beta+b^{(i)}_{\mathrm{R}}) &\geq (\beta+b^{(i)}_{\mathrm{L}})+2\bar\kappa+4\bar\kappa^3
		\quad\text{if}\quad \bar\beta\geq\beta+\inf\aA\,,
	\end{equs}
	which are certainly true if
	\begin{equs}
		(\beta+b^{(i)}_{\mathrm{L}}) &\geq (\beta-2+b^{(i)}_{\mathrm{R}})+\bar\kappa^3\,,
		\\
		(\beta+b^{(2)}_{\mathrm{L}}) &\geq (\beta+5+b^{(1)}_{\mathrm{R}})+\bar\kappa^3\,,	
		\\
		(\beta+\inf\aA+b^{(i)}_{\mathrm{R}}) &\geq (\beta+b^{(i)}_{\mathrm{L}})+2\bar\kappa+4\bar\kappa^3\,.
	\end{equs}
	The above bounds are satisfied since $2\bar\kappa+5\bar\kappa^3\leq 1/4 \leq 2+\inf\aA$.
\end{proof}

\subsection{Singular modelled distributions}\label{sec:modelled_distributions}

Given a point $x\in\R^3$, we write $|x|$ for the supremum norm, 
and denote by $B(x,r)$ the open ball centred at $z$ of radius $r>0$.
Given a space-time point $z=(t,x)\in\R^{1+3}$, we write $|z|\eqdef \max\{t^{1/2},|x|\}$ for the parabolic distance, 
and denote by $B(z,r)$ the open parabolic ball centred at $z$ of radius $r>0$.
For $k\in\N_0^{1+3}$ we write $|k|\eqdef 2k_0+k_1+k_2+k_3$.
We denote by $\bB$ the set of functions over space-time $\R^{1+3}$ supported in the unit parabolic ball centred 
at the origin with the $\alpha$-H{\"o}lder norm bounded by $1$ for some fixed $\alpha>3/2-3\kappa$. We denote by 
$\mathcal{B}_-$ the subset of $\mathcal{B}$ consisting of functions supported in the half-space $\{(t,x)\,|\,t \leq 0\}$.
For $\psi\in C(\R^{1+3})$, $(t,x)\in\R^{1+3}$ and $r>0$ we define $\psi^r_{t,x}\in C(\R^{1+3})$ by $\psi^r_{t,x}(s,y)\eqdef\frac{1}{r^5}\psi(\frac{s-t}{r^2},\frac{y-x}{r})$. 
We note the following result about the kernel $K$ of the heat semigroup with unit mass $t\mapsto \exp(t(\Laplace -1))$.

\begin{lem}\label{lem:kernel_decomposition}
	The heat kernel $K$ with unit mass is regularizing of order $2$, that is
	\begin{equ}
		K=K^++K^-=\sum_{n\geq0} K_n + K^-\,,
	\end{equ}
	where the kernels $K^\pm$, $(K_n)_{n\in\N_0}$ satisfy Assumptions 5.1 and 5.4 from \cite{Hai14} and for all $t\in\R$ the function $x\mapsto K^\pm(t,x)$ depends only on $|x|$.
\end{lem}

Recall that a model is a pair of maps 
\begin{equs}
	\Pi\,:\,\R^{1+3}\ni z&\mapsto \Pi^{z}\in L(\tT,\sS'(\R^{1+3}))\,,
	\\
	\Gamma\,:\,(\R^{1+3})^2\ni (z,\bar z)&\mapsto\Gamma^{z;\bar z}\in\gG\,,
\end{equs}
satisfying the conditions specified in~\cite[Definition~2.17]{Hai14}. The space of models is equipped with the topology generated by the family of seminorms $\|(\Pi,\Gamma)\|_\fK$ indexed by compact sets $\fK\subset\R^{1+3}$ (see Definition~\ref{defn:norm_model} below).

\begin{defn}\label{def:model}
	A model $(\Pi,\Gamma)$ is continuous if $\Pi^{z}\tau\in C(\R^{1+3})$ for all $z\in\R^{1+3}$ and $\tau\in\tT$. We say that a model $(\Pi,\Gamma)$ is admissible if
	\begin{equs}
		(\Pi^{z} X^k)(\bar z) &= (\bar z - z)^k\,,
		\\
		(\Pi^{z} \mathcal{I} \tau)(\bar z) &= (\Pi^{z} \tau, K^+(\bar z - \bigcdot)) 
		- \sum_{|k| < |\mathcal{I}\tau|} \frac{(\bar z - z)^k}{k!} (\Pi^{z} \tau, \partial^k K^+(z -\bigcdot))
	\end{equs}
	for all $z,\bar z\in\R^{1+3}$ and denote by $\mM$ the set of admissible models.
\end{defn}

\begin{defn}\label{defn:norm_model}	
	Given $\weight_\Pi\in C(\R^3,(0,1])$, a (typically non-compact) closed set $\fK\subset\R^{1+3}$ and a model $(\Pi,\Gamma)$, we define
	its ``weighted norm'' by
	\begin{equs}
		\|\Pi\|_{\fK,\weight_\Pi} &\eqdef
		\sup_{\tau\in\tT^\circ}\sup_{\psi\in\bB}\sup_{r\in(0,1]}\sup_{z\in\fK}r^{-|\tau|}\,\weight_\Pi(x)\,
		|(\Pi^{z} \tau)(\psi^r_{z})|\;,\\
		\|\Gamma\|_{\fK,\weight_\Pi}&\eqdef
		\sup_{\tau\in \tT^\circ}\sup_{\beta< |\tau|}\sup_{\substack{z,\bar z \in \fK\\0<|z-\bar z|\leq 1}}\weight_\Pi(x)\,\frac{\|\Gamma^{z;\bar z} \tau\|_{\beta}}{|z-\bar z|^{|\tau|-\beta}} \;,
	\end{equs}
and we set $\|(\Pi,\Gamma)\|_{\fK,\weight_\Pi}\eqdef\|\Pi\|_{\fK,\weight_\Pi}+\|\Gamma\|_{\fK,\weight_\Pi}$. We omit $\weight_\Pi$ in the notation if $\weight_\Pi=1$.  We denote by $\mM(\weight_\Pi)$ the set of $(\Pi,\Gamma)\in\mM$ such that  $\|(\Pi,\Gamma)\|_{T, \weight_\Pi}\eqdef\|(\Pi,\Gamma)\|_{\bar\oO_T,\weight_\Pi}<\infty$ for all $T>0$, where $\bar\oO_T=[-1,T]\times\R^3$.
\end{defn}

Given $\gamma\in \R$ and a model $(\Pi,\Gamma)$, the space of modelled distributions $\dD^\gamma=\dD^\gamma(\Gamma)$ was defined in~\cite[Definition~3.1]{Hai14}. Recall that $\dD^\gamma$ consists of functions $f\,:\,\R^{1+3}\to\tT_{<\gamma}$ such that $\VERT{f}_{\gamma;\fK}<\infty$ for every compact set $\fK\subset\R^{1+3}$. When comparing $f\in \dD^\gamma(\Gamma)$ and $\bar f\in \dD^\gamma(\bar\Gamma)$ for two different models $(\Pi,\Gamma)$ and $(\bar\Pi,\bar\Gamma)$ we use the quantity $\VERT{f;\bar f}_{\gamma;\fK}$ introduced in~\cite[Remark~3.6]{Hai14}. We denote by $\dD_+^\gamma=\dD_+^\gamma(\Gamma)$ the vector space of functions $f\,:\,\R_>\times\R^3\to\tT_{<\gamma}$ such that $\VERT{f}_{\gamma;\fK}<\infty$ for every compact set $\fK\subset\R_>\times\R^3$. We identify elements of $\dD_+^\gamma$ with functions $f\,:\,\R^{1+3}\to\tT_{<\gamma}$ vanishing on $\R_\leq\times\R^3$. The space of singular modelled distributions $\dD^{\gamma,\eta}=\dD^{\gamma,\eta}(\Gamma)$ consists of $f\in \dD_+^\gamma$ such that $\VERT{f}_{\gamma,\eta;\fK}<\infty$ for every compact set $\fK\subset\R^{1+3}$ with the seminorm introduced in~\cite[Definition~6.2]{Hai14}. Note that elements of $\dD^{\gamma,\eta}$ are allowed to be singular at the time zero hypersurface with the blow-up rate controlled by the parameter $\eta\in\R$. In the following definition we introduce seminorms that allow to control the growth in space of elements of $\dD^{\gamma,\eta}$.

\begin{defn}\label{def:modelled_distributions}
	Let $\gamma,\eta \in \R$, $\fK\subset \R_\geq\times\R^3$ and $\weight\,:\,\{1,2\}\times\aA\to C(\R^{1+3},(0,1])$. Given a model $(\Pi,\Gamma)$ and a map $f: \R_>\times\R^3 \to \tT_{<\gamma}$ we define
	\begin{equ}
		\VERT{f}_{\gamma,\eta;\fK,\weight} \eqdef \|f\|_{\gamma,\eta;\fK,\weight}  
		+ [f]^{\mathrm{time}}_{\gamma,\eta;\fK,\weight}
		+ [f]^{\mathrm{space}}_{\gamma,\eta;\fK,\weight}\,,
	\end{equ}
	where
	\begin{equs}{}
		\|f\|_{\gamma,\eta;\fK,\weight} 
		&\eqdef \sup_{\beta < \gamma}\sup_{(t,x)\in \fK} \weight^{(1)}(\beta;t,x)\frac{\|f(t,x)\|_\beta}{t^{((\eta-\beta)\wedge0)/2}}\,,
		\\{}
		[f]^{\mathrm{time}}_{\gamma,\eta;\fK,\weight}
		&\eqdef \sup_{\beta < \gamma}\sup_{\substack{(t,x),(s,x)\in\fK\\s<t\leq 2s}}  \weight^{(1)}(\beta;t,x) \frac{\|f(t,x) - \Gamma^{t,x;s,x} f(s,x)\|_\beta}{(t-s)^{(\gamma - \beta)/2} s^{(\eta-\gamma)/2}}\,,
		\\{}	
		[f]^{\mathrm{space}}_{\gamma,\eta;\fK,\weight}
		&\eqdef \sup_{\beta < \gamma}\sup_{\substack{(t,x),(t,y)\in\fK\\0<|x-y|^2\leq t}}  \weight^{(2)}(\beta;t,x) \frac{\|f(t,x) - \Gamma^{t,x;t,y} f(t,y)\|_\beta}{|x-y|^{\gamma - \beta} t^{(\eta-\gamma)/2}}\,.
	\end{equs}
	Given models $(\Pi,\Gamma),(\bar\Pi,\bar\Gamma)$ and maps $f,\bar f: \R^{1+3} \rightarrow \tT_{<\gamma}$ we define
	\begin{equ}
		\VERT{f;\bar f}_{\gamma,\eta;\fK,\weight} \eqdef \|f-\bar f\|_{\gamma,\eta;\fK,\weight}  
		+ [f;f]_{\gamma,\eta;\fK,\weight}\,,
	\end{equ}
	where
	\begin{equs}{}
		[f; \bar f]_{\gamma,\eta;\fK,\weight}
		&\eqdef
		\sup_{\beta < \gamma}\sup_{\substack{(t,x),(s,y)\in\fK\\0<|(t,x)-(s,y)|^2\leq s\leq t}} 
		\weight^{(1)}(\beta;t,x)
		\\
		&\times\frac{\|f(t,x) - \bar f(t,x) - \Gamma^{t,x;s,y} f(s,y) + \bar\Gamma^{t,x;s,y} \bar f(s,y)  \|_\beta}
		{ |(t,x)-(s,y)|^{\gamma - \beta}  s^{(\eta-\gamma)/2}}\,.
	\end{equs}
	For $T>0$ we write
	\begin{equ}
		\VERT{f}_{\gamma,\eta;T, \weight}=\VERT{f}_{\gamma,\eta;\oO_T,\weight}\,,
		\qquad
		\VERT{f;\bar f}_{\gamma,\eta;T, \weight}=\VERT{f;\bar f}_{\gamma,\eta;\oO_T,\weight}\,,	
	\end{equ}
	where $\oO_T=[0,T]\times\R^3$. We also use the above notation with $\weight\in C(\R^3,(0,1])$ by identifying it with a constant function $\weight\,:\,\{1,2\}\times\aA\to C(\R^{1+3},(0,1])$. We omit $\weight$ in the notation if $\weight=1$.
	
	Given $\weight\in C(\R^3,(0,1])$ and $T>0$ we define the space of weighted singular modelled distributions $\dD^{\gamma,\eta}_{T, \weight}(\fF,\Gamma)$ as the set of maps $f:(0,T]\times\R^3\to \fF\subset \tT_{<\gamma}$ such that $\VERT{f}_{\gamma,\eta;T, \weight}<\infty$. We omit $\fF$ and $\Gamma$ if they are clear from the context.
\end{defn}

\begin{rmk}
	For a fixed compact region $\fK$, the norm $\VERT{\bigcdot}_{\gamma,\eta;\fK}$ is equivalent to the norm in the space of singular modelled distributions introduced in~\cite[Definition~6.2]{Hai14}. 
\end{rmk}

\begin{rmk}
	All modelled distributions that will appear below belong to $\dD^{\gamma,\eta}_{T, \weight}(\tT_{[\eta,\gamma)})$ with $\weight\in C(\R^3,(0,1])$ of polynomial type. We will use the norms $\VERT{\bigcdot}_{\gamma,\eta;T,\weight}$ with a~general weight $\weight\,:\,\{1,2\}\times\aA\to C(\R^{1+3},(0,1])$ but we will always assume that $T\in(0,1]$. Hence, our assumptions about the weights involve only $t\in[0,1]$. We do not treat separately the increments in time and space in the definition of $[f;\bar f]_{\gamma,\eta;\fK,\weight}$ because when comparing two singular modelled distributions we will always use time-independent weights $\weight\in C(\R^3,(0,1])$.
\end{rmk}

\begin{rmk}\label{rmk:norm_cubes}
	If $\weight\in C(\R^3,(0,1])$ satisfies~\eqref{eq:W-0}, then the norm $\VERT{\bigcdot}_{\gamma,\eta;T, \weight}$ in $\dD^{\gamma,\eta}_{T, \weight}$ is equivalent to the norm
	\begin{equ}
		\sup_{n \in \Z^3} \weight(n) \,\VERT{\bigcdot}_{\gamma,\eta;[0,T]\times B(n,1)}\,.
	\end{equ}
\end{rmk}
	
\begin{rmk}\label{rmk:embedding_modelled_distributions}
	Let $\gamma\in\R$ and $\eta_1\geq\eta_2$. For $\weight\,:\,\{1,2\}\times \aA\to C(\R^{1+3},(0,1])$ and $f \in \dD^{\gamma,\eta_1}_{T, \weight}(\tT_{[\eta_1,\gamma)})$ we have $f\in \dD^{\gamma,\eta_2}_{T, \weight\weight_\Pi}$ and $\VERT{f}_{\gamma,\eta_2;T, \weight\weight_\Pi}\leq T^{\frac{\eta_1-\eta_2}{2}}\, \VERT{f}_{\gamma,\eta_1;T, \weight}$.
\end{rmk}	

We now discuss properties of weighted singular modelled distributions, including embeddings, compactness, product estimates, and Schauder estimates. The first three properties require minimal assumptions on the weights. For the Schauder estimates, we establish two separate results: one for polynomial weights and one for exponential weights. The estimate for polynomial weights will be used to prove existence of solutions, while the estimate for exponential weights will serve to establish uniqueness.

\begin{lem}\label{lem:embedding_modelled_distributions}
	Let $\gamma_2\geq\gamma_1\geq \eta$ and $\weight_\Pi,\weight_{\mathrm{S}}\in C(\R^3,(0,1])$, $\weight_{\mathrm{L}}\,:\,\{1,2\}\times \aA\to C(\R^{1+3},(0,1])$ satisfy~\eqref{eq:W-0} and~\eqref{eq:W-7}. We have
	\begin{equs}
		\VERT{\qQ_{<\gamma_1} f}_{\gamma_1,\eta;T, \weight_{\mathrm{S}}\weight_\Pi}
		\lesssim&
		(1+\|(\Pi,\Gamma)\|_{T, \weight_\Pi})\,\VERT{f}_{\gamma_2,\eta;T, \weight_{\mathrm{S}} }\,,
		\\
		\VERT{\qQ_{<\gamma_1}f}_{\gamma_1,\eta;T,\weight_\mathrm{L}}
		\lesssim&
		(1+\|(\Pi,\Gamma)\|_{T, \weight_\Pi})\,\VERT{f}_{\gamma_2,\eta;T,\weight_\mathrm{L}}
	\end{equs}
	uniformly over $f\in\dD^{\gamma_2,\eta}_{T,\weight_{\mathrm{S}}}$.
\end{lem}
\begin{proof}
	To prove the second bound we note that 
	\begin{equs}
		\frac{\|\Gamma^{t,x;s,y}\qQ_{\gamma} f(s,y)\|_\beta}{{|(t,x)-(s,y)|^{\gamma - \beta} s^{(\eta-\gamma)/2}}} 
		\leq
		\frac{\|\Gamma^{t,x;s,y}\qQ_{\gamma} f(s,y)\|_\beta}{{|(t,x)-(s,y)|^{\gamma_1 - \beta} s^{(\eta-\gamma_1)/2}}} 
		\\
		\lesssim 
		\weight_{\mathrm{L}}^{(1)}(\gamma;s,x)^{-1}\,\weight_\Pi(x)^{-1}\,
		\|f\|_{\gamma_2,\eta;T, \weight_{\mathrm{L}}} \|(\Pi,\Gamma)\|_{\mM(\weight_\Pi)}
	\end{equs}
	uniformly over $\gamma\in[\gamma_1,\gamma_2)$, $\beta<\gamma_1$ and $(t,x),(s,y)\in\R^{1+3}$ such that $0<s<t$ and $|(t,x)-(s,y)|^2\leq s$ and subsequently use~\eqref{eq:W-7}. The proof of the first bound is based on an analogous estimate with $\weight_{\mathrm{L}}$ replaced by $\weight_{\mathrm S}$.
\end{proof}

\begin{lem}
\label{le:modelled_distribution_compactness_local}
	Let $0<\bar\gamma<\gamma$ be such that $\tT_{[\bar\gamma,\gamma)}=\emptyset$ and $(\Pi_n, \Gamma_n)_{n\in\N_+}$ be a sequence of models converging to $(\Pi, \Gamma)$. Suppose that
	$f_n \in \dD_+^{\gamma}(\Gamma_n)$ are such that $\VERT{f_n}_{\gamma;\fK}$ is uniformly bounded in $n\in\N_+$ for every compact set $\fK\subset\R_>\times\R^3$. Then there exists a sequence $(n_k)_{k\in\N_+}$ and $f \in \dD_+^\gamma(\Gamma)$ such that $\lim_{k\to\infty}\VERT{f_{n_k};f}_{\bar\gamma;\fK}=0$ for every compact set $\fK\subset\R_>\times\R^3$.
\end{lem}
\begin{proof}
	Fix a compact set $\fK\subset\R_>\times\R^3$. Uniform boundedness of $\VERT{f_n}_{\gamma;\fK}$ implies that $f_n$, viewed as a function $\fK\to\tT$, is uniformly bounded in $n\in\N_+$ in some H{\"o}lder space. Hence, by the Arzela--Ascoli theorem, there exists a sequence $(n_k)_{k\in\N_+}$ and $f\,:\,\fK\to\tT$ such that
	\begin{equ}
	\label{e:compactness1}
		\lim_{k\to\infty}\sup_{\beta\in\aA}\sup_{z\in\fK}\|f_{n_k}(z)-f(z)\|_\beta=0\;.
	\end{equ}
	From this, the convergence of the model and uniform boundedness of $\VERT{f_n}_{\gamma;\fK}$, it immediately follows that $\VERT{f}_{\gamma;\fK}<\infty$. Let us prove that $\lim_{k\to\infty}\VERT{f_{n_k};f}_{\bar\gamma;\fK}=0$. To this end, we have to show that if $\beta\in\aA$ and $\beta<\gamma$, then\footnote{Note that, by the assumption $\tT_{[\bar\gamma,\gamma)}=\emptyset$, we have $\beta<\bar\gamma$.}
	\begin{equ}
		\lim_{k\to\infty}\sup_{\substack{z,\bar z\in\fK\\0<|z-\bar z|\leq1}}
		\frac{\|f_{n_k}(z)-f(z)-\Gamma_{n_k}^{z,\bar z}f_{n_k}(\bar z)+\Gamma^{z,\bar z}f(\bar z)\|_{\beta}}{|z-\bar z|^{\bar\gamma-\beta}}=0\;.
	\end{equ}
	We distinguish between two cases based on whether $|z-\bar z|\leq\delta$ or not. In the first case, we have 
	\begin{equ}
		\sup_{\substack{z,\bar z\in\fK\\0<|z-\bar z|\leq\delta}}\frac{\|f_{n_k}(z)-\Gamma_{n_k}^{z,\bar z}f_{n_k}(\bar z)\|_{\beta}}{|z-\bar z|^{\bar\gamma-\beta}}
		\leq
		\delta^{\gamma-\bar\gamma}\,\VERT{f_{n_k}}_{\gamma;\fK}\;.
	\end{equ}
	The same bound holds with $f_{n_k},\Gamma_{n_k}$replaced by $f,\Gamma$. For the second case, we estimate 
	\begin{equ}
		\sup_{\substack{z,\bar z\in\fK\\\delta<|z-\bar z|\leq1}}
		\frac{\|\Gamma_{n_k}^{z,\bar z}f_{n_k}(\bar z)-\Gamma^{z,\bar z}f(\bar z)\|_{\beta}}{|z-\bar z|^{\bar\gamma-\beta}}
		\leq
		\delta^{-\bar\gamma+\beta}\sup_{\substack{z,\bar z\in\fK\\0<|z-\bar z|\leq1}}\|\Gamma_{n_k}^{z,\bar z}f_{n_k}(\bar z)-\Gamma^{z,\bar z}f(\bar z)\|_{\beta}
	\end{equ}
	and note that the right-hand side converges to $0$ as $k\to\infty$ by \eqref{e:compactness1} and the convergence of the model. We also have a similar estimate for $\|f_{n_k}(z)-f(z)\|_\beta$ and the triangle inequality. This finishes the proof of $\lim_{k\to\infty}\VERT{f_{n_k};f}_{\bar\gamma;\fK}=0$. In order to find a sequence $(n_k)_{k\in\N_+}$ and $f\,:\,\R_>\times\R^3\to\tT$ such that $\lim_{k\to\infty}\VERT{f_{n_k};f}_{\bar\gamma;\fK}=0$ for every compact set $\fK\subset\R_>\times\R^3$ we use a diagonal argument.
\end{proof}

\begin{lem}
\label{le:compactness_modelled_distribution}
	Let $0<\bar\gamma<\gamma$ be such that $\tT_{[\bar\gamma;\gamma)}=\emptyset$, $\bar\eta<\eta$, $T>0$, $\weight_\Pi,\weight,\bar\weight\in C(\R^3,(0,1])$ satisfy~\eqref{eq:W-0} and be such that $\lim_{|x|\to\infty}\frac{\bar\weight(x)}{\weight(x)}=0$,
	$(\Pi_n, \Gamma_n)_{n\in\N_+}$ be a sequence of models in $\mM(\weight_\Pi)$ converging to $(\Pi, \Gamma)$ and $f_n \in \dD^{\gamma,\eta}_{T, \weight}(\Gamma_n)$ be such that $\VERT{f_n}_{\gamma,\eta;T, \weight}$ is uniformly bounded in $n\in\N_+$. Then there exists a sequence $(n_k)_{k\in\N_+}$ and $f \in \dD^{\gamma,\eta}_{T, \weight}(\Gamma)$ such that
	\begin{equ}
		\lim_{k\to\infty}\VERT{f_{n_k}; f}_{\bar\gamma,\bar\eta;T, \bar\weight}=0\,.
	\end{equ}
\end{lem}
\begin{proof}
	By Lemma~\ref{le:modelled_distribution_compactness_local} there exists $f\in \dD_+^\gamma(\Gamma)$ and a sequence $(n_k)_{k\in\N_+}$ such that $\lim_{k\to\infty}\VERT{f_{n_k};f}_{\bar\gamma;\fK}=0$ for every compact set $\fK\subset\R_>\times\R^{3}$. In particular, $\lim_{k\to\infty}\VERT{f_{n_k};f}_{\bar\gamma,\bar\eta;\fK}=0$. From the convergence of the model and uniform boundedness of $\VERT{f_n}_{\gamma,\eta;T, \weight}$, it follows that $f\in \dD^{\gamma,\eta}_{T, \weight}(\Gamma)$. Let $\delta>0$. By uniform boundedness of $\VERT{f_n}_{\gamma,\eta;T, \weight}$ and $\VERT{f}_{\gamma,\eta;T, \weight}$, $\lim_{|x|\to\infty}\frac{\bar\weight(x)}{\weight(x)}=0$ and Remark~\ref{rmk:embedding_modelled_distributions}, there exists a compact set $\fK\subset \R_>\times\R^3$ such that
	\begin{equ}
		\VERT{f_{n}; f}_{\gamma,\bar\eta;\oO_T\setminus\fK,\bar\weight}
		\leq
		\VERT{f_n}_{\gamma,\bar\eta;\oO_T\setminus\fK,\bar\weight}
		+
		\VERT{f}_{\gamma,\bar\eta;\oO_T\setminus\fK,\bar\weight}\leq \delta\,,
	\end{equ}
	where $\oO_T=[0,T]\times\R^3$. Since $\VERT{f_{n_k}; f}_{\bar\gamma,\bar\eta;T, \bar\weight}\leq \VERT{f_{n_k};f}_{\bar\gamma,\bar\eta;\fK}+\VERT{f_{n_k}; f}_{\gamma,\bar\eta;\oO_T\setminus\fK,\bar\weight}$, the proof is complete.
\end{proof}

\begin{lem}\label{lem:multiplication}
	Let $\gamma_1,\gamma_2,\eta_1,\eta_2\in\R$, $\gamma=(\gamma_1+\eta_2)\wedge(\gamma_2 + \eta_1)$, $\eta=\eta_1+\eta_2$ and $\weight_1,\weight_2,\weight_\Pi\in C(\R^3,(0,1])$ satisfy~\eqref{eq:W-0}. Set $\weight=\weight_1\weight_2\weight_\Pi^2$. For $f\in \dD^{\gamma_1,\eta_1}_{T,\weight_1}(\tT_{[\eta_1,\gamma_1)})$ and $g\in \dD^{\gamma_2,\eta_2}_{T,\weight_2}(\tT_{[\eta_2,\gamma_2)})$ we have  $fg\in \dD^{\gamma,\eta}_{T, \weight}(\tT_{[\eta,\gamma)})$ and
		\begin{equ}
			\VERT{fg}_{\gamma,\eta;T, \weight} 
			\lesssim
			(1+\|(\Pi,\Gamma)\|_{T, \weight_\Pi})^2\,
			\VERT{f}_{\gamma_1,\eta_1;T,\weight_1}\,
			\VERT{g}_{\gamma_2,\eta_2;T,\weight_2}
		\end{equ}
		uniformly over $T\in(0,1]$, $(\Pi,\Gamma)\in\mM(\weight_\Pi)$, $f\in \dD^{\gamma_1,\eta_1}_{T,\weight_1}(\tT_{[\eta_1,\gamma_1)})$, $g\in \dD^{\gamma_2,\eta_2}_{T,\weight_2}(\tT_{[\eta_2,\gamma_2)})$. Moreover,
		\begin{equ}
			\VERT{fg;\bar f\bar g}_{\gamma,\eta;T, \weight} 
			\lesssim
			\VERT{f;\bar f}_{\gamma,\eta;T,\weight_1} 
			+
			\VERT{g;\bar g}_{\gamma,\eta;T,\weight_2}
			+
			\|(\Pi,\Gamma)-(\bar\Pi,\bar\Gamma)\|_{T, \weight_\Pi}
		\end{equ}
		uniformly over $T\in(0,1]$ and locally uniformly over $(\Pi,\Gamma),(\bar\Pi,\bar\Gamma)\in\mM(\weight_\Pi)$, $f\in \dD^{\gamma_1,\eta_1}_{T,\weight_1}(\tT_{[\eta_1,\gamma_1)},\Gamma)$, $g\in \dD^{\gamma_2,\eta_2}_{T,\weight_2}(\tT_{[\eta_2,\gamma_2)},\Gamma)$, $\bar f\in \dD^{\gamma_1,\eta_1}_{T,\weight_1}(\tT_{[\eta_1,\gamma_1)},\bar\Gamma)$, $\bar g\in \dD^{\gamma_2,\eta_2}_{T,\weight_2}(\tT_{[\eta_2,\gamma_2)},\bar\Gamma)$.
\end{lem}
\begin{proof}
	The result is a consequence of Remark~\ref{rmk:norm_cubes} and~\cite[Proposition~6.12]{Hai14}.
\end{proof}

\begin{defn}\label{def:integration_K}
	Let $\gamma\in(0,1)$, $\eta>-2$ be such that $\gamma+2,\eta+2\notin\N_0$, $T\in(0,1]$, $\weight_{\mathrm{S}},\weight_\Pi\in C(\R^3,(0,1])$ satisfy~\eqref{eq:W-0} and~\eqref{eq:W-6} and $(\Pi,\Gamma)\in\mM(\weight_\Pi)$. The maps
	\begin{equ}
		\kK^+,\kK^-\,:\,\dD^{\gamma,\eta}_{T,\weight_{\mathrm{S}}}(\tT_{[\eta,\gamma)}) \to \dD^{\gamma+2,\eta+2}_{T,\weight_{\mathrm{S}}\weight_\Pi^2}
	\end{equ}
	are defined by
	\begin{equs}
		(\kK^+ f)(t,x) &\eqdef \qQ_{<\gamma+2}\iI(f(t,x))
		\\
		&\quad + \sum_{\zeta \in \aA} \sum_{|k| < (\zeta+2) \wedge (\gamma+2)} \frac{\bXXX^k}{k!} \langle \Pi^{t,x} \qQ_{\zeta} f(t,x), \partial^k K^+((t,x) - \bigcdot) \rangle
		\\
		&\quad + 
		\sum_{|k| < \gamma+2} \frac{\bXXX^k}{k!} \langle \rR f - \Pi^{t,x} f(t,x), \partial^k K^+((t,x) - \bigcdot) \rangle\,,
		\\
		(\kK^- f)(t,x)&\eqdef
		\sum_{|k| < \gamma+2} \frac{\bXXX^k}{k!} \langle \rR f, \partial^k K^-((t,x) - \bigcdot) \rangle
	\end{equs}
	for $f\in\dD^{\gamma,\eta}_{T,\weight_{\mathrm{S}}}$ and $(t,x)\in\oO_T$, where $\rR$ is the reconstruction operator in \cite[Theorem~3.10]{Hai14} and $\oO_T=[0,T]\times\R^3$. We also set $\kK=\kK^++\kK^-$. The notation $\kK^\pm$ is used to indicate that the statement applies to both $\kK^+$ and $\kK^-$.
\end{defn}

\begin{lem}\label{lem:integration_K}
	The maps $\kK^+,\kK^-$ introduced above	are well defined and satisfy
	\begin{equ}\label{eq:abs_integration_identity}
		\rR\kK^\pm f = K^\pm \ast \rR f\,.
	\end{equ}
	We have
	\begin{equ}
		\VERT{\kK^\pm f}_{\gamma+2,\eta+2;T,\weight_\Pi^2\weight_{\mathrm{S}}} 
		\lesssim
		(1+\|(\Pi,\Gamma)\|_{T, \weight_\Pi})^2\,
		\VERT{f}_{\gamma,\eta;T, \weight_{\mathrm{S}} }\,,
	\end{equ}
	uniformly over $T\in(0,1]$, $(\Pi,\Gamma)\in\mM(\weight_\Pi)$ and $f\in\dD^{\gamma,\eta}_{T,\weight_{\mathrm{S}}}(\tT_{[\eta,\gamma)})$. Moreover,
	\begin{equ}
		\VERT{\kK^\pm f;\bar\kK^\pm \bar f}_{\gamma+2,\eta+2;T,\weight_\Pi^2\weight_{\mathrm{S}}} 
		\lesssim
		\VERT{f;\bar f}_{\gamma,\eta;T, \weight_{\mathrm{S}} }
		+
		\|(\Pi,\Gamma)-(\bar\Pi,\bar\Gamma)\|_{T, \weight_\Pi}\,,
	\end{equ}
	uniformly over $T\in(0,1]$ and locally uniformly over $(\Pi,\Gamma),(\bar\Pi,\bar\Gamma)\in\mM(\weight_\Pi)$, $f\in\dD^{\gamma,\eta}_{T,\weight_{\mathrm{S}}}(\tT_{[\eta,\gamma)},\Gamma)$, and $\bar f\in\dD^{\gamma,\eta}_{T,\weight_{\mathrm{S}}}(\tT_{[\eta,\gamma)},\bar\Gamma)$.
\end{lem}
\begin{proof}
	The statement concerning $\kK^+$ follows from Remark~\ref{rmk:norm_cubes} and~\cite[Proposition~6.16, Theorem~7.1]{Hai14}. To prove the estimates for $\kK^-$ we first use Lemma~\ref{lem:reconstruction} and~\eqref{eq:reconstruction_trivial} to show the bound
	\begin{equ}
		(\weight_\Pi^2\weight_{\mathrm{S}})(t,x)\,\langle \mathcal{R} f, \psi_{t,x}^r \rangle 
		\lesssim 
		r^\eta\, (1+\|\Pi,\Gamma\|_{T, \weight_\Pi})^2\,\VERT{f}_{\gamma,\eta;T, \weight_{\mathrm{S}} }
	\end{equ}
	uniform over all $\psi \in \mathcal{B}_-$, $t\in(0,T]$, $x\in\R^3$, $r\in(0,1]$ and $f\in\dD^{\gamma,\eta}_{T,\weight_{\mathrm{S}}}$ and the bound
	\begin{equs}
		(\weight_\Pi^2\weight_{\mathrm{S}})(t,x)\,\langle \mathcal{R} f& - \Pi^{t-r^2,x} f (t - r^2 , x), \psi_{t,x}^r \rangle 
		\\
		&\lesssim 
		r^{\gamma} t^{(\eta - \gamma)/2} (1+\|\Pi,\Gamma\|_{T, \weight_\Pi})^2\,\VERT{f}_{\gamma,\eta;T, \weight_{\mathrm{S}} }
	\end{equs}
	uniform over all $\psi \in \mathcal{B}_-$, $t\in[4r^2,T]$, $x\in\R^3$, $r\in (0,1]$ and $f\in\dD^{\gamma,\eta}_{T,\weight_{\mathrm{S}}}$. The estimates for $\kK^-$ follow now by the argument from the proof of Proposition~4.5 in~\cite{HL18}.
\end{proof}

\begin{lem}\label{lem:multiplication_exp}
	Let $\gamma,\eta\in\R$, $T\in(0,1]$ and $\weight_{\mathrm{S}},\weight_\Pi\in C(\R^3,(0,1]),\weight_{\mathrm{L/R}}\,:\,\{1,2\}\times \aA\to C(\R^{1+3},(0,1])$ satisfy~\eqref{eq:W-0} and \eqref{eq:W-5}.
	We have
	\begin{equs}
		\VERT{fg^2}_{\gamma-2\eta,3\eta;T, \weight_{\mathrm{R}} }
		\lesssim (1+\|(\Pi,\Gamma)\|_{T, \weight_\Pi})^4\,
		\VERT{f}_{\gamma,\eta;T, \weight_{\mathrm{L}} }\,\VERT{g}^2_{\gamma,\eta;T, \weight_{\mathrm{S}} }
	\end{equs}
	uniformly over $T\in(0,1]$, $(\Pi,\Gamma)\in\mM(\weight_\Pi)$ and $f,g\in\dD^{\gamma,\eta}_{T,\weight_{\mathrm{S}}}(\tT_{[\eta,\gamma)})$.
\end{lem}
\begin{proof}
	The statement follows from the proofs of Theorem~4.7 and Proposition~6.12 in~\cite{Hai14}. 
\end{proof}

\begin{thm}\label{thm:integration_K_exp}
	Recall the parameter $\fc>0$ introduced in Assumption~\ref{ass:weights}. Let $\gamma\in(0,1/4)$, $\eta>-2$  be such that $\gamma+2-\fc,\eta+2-\fc\notin\N_0$ and $\tT_{[\gamma+2-\fc,\gamma+2)}=\emptyset$, and $\weight_{\mathrm{S}},\weight_\Pi\in C(\R^3,(0,1]),\weight_{\mathrm{L/R}}\,:\,\{1,2\}\times \aA\to C(\R^{1+3},(0,1])$ satisfy~\eqref{eq:W-0}-\eqref{eq:W-4} and \eqref{eq:W-6}.	
	We have
	\begin{equ}\label{eq:integration_bound_exp}
		\VERT{\kK^\pm f}_{\gamma+2-\fc,\eta+2-\fc;T, \weight_{\mathrm{L}} } \lesssim (1+\|(\Pi,\Gamma)\|_{T, \weight_\Pi})^2\,\VERT{f}_{\gamma,\eta;T, \weight_{\mathrm{R}} }
	\end{equ}
	uniformly over $(\Pi,\Gamma)\in\mM(\weight_\Pi)$, $f\in\dD^{\gamma,\eta}_{T,\weight_{\mathrm{S}}}(\tT_{[\eta,\gamma)})$ and $T\in(0,1]$.
\end{thm}
\begin{proof}
	The proof of the bound for $\kK^+$ is almost identical to the proof of Theorem 4.3 in \cite{HL18} and we only discuss the necessary modifications.
	\begin{enumerate}
		\item We prove a bound for the integration operator $\kK^+$ whereas the bound in Theorem~4.3 in \cite{HL18} is for $\kK^+$ composed with multiplication by a noise $\Xi$, that is, \cite{HL18} proves a bound of the form $\|\kK^+ f\| \lesssim \|\Pi\|\,(1+\|\Gamma\|)\,\|u\|$ for $f=\Xi u$. The inspection of the proof therein reveals that all the estimates are actually written in terms of $f$ with the exception of two estimates for components of $\kK^+ f$ in sectors of non-integer regularity. The latter estimates can be trivially rewritten in terms of $f$ since in sectors of non-integer regularity $\kK^+ f=\iI f=\iI(\Xi u)$ and the operation $\kK^+$ amounts to a mere relabelling of the basis elements (there is no integration involved). 
		
		\item The norms that appear on both sides of our bound~\eqref{eq:integration_bound_exp} involve different weights whereas in Theorem~4.3 in \cite{HL18} the weights are the same. The choice of weights in the norms is determined by Assumption~3.6 and~4.1 therein. Upon replacing the conditions (W-0)--(W-4) formulated there by our conditions \eqref{eq:W-0}--\eqref{eq:W-4} the same proof gives a bound with our choice of weights in the norms. 
		
		\item Theorem~4.3 in \cite{HL18} is stated in the setting of $L^p$-Besov-type singular modelled distributions with finite $p$. In order to adapt the proof therein to our $L^\infty$-setting we have to first establish $L^\infty$-analogues of the estimates (4.4) and (4.5) in \cite{HL18} for the reconstruction operator. We replace (4.4) and (4.5) in~\cite{HL18} respectively by the bound
		\begin{equ}\label{eq:int-exp-bound-1}
			\overline\weight_{\mathrm{L}}(t,x)\,\langle \mathcal{R} f, \psi_{t,x}^r \rangle 
			\lesssim 
			r^{\eta-\fc}\, (1+\|\Pi,\Gamma\|_{T, \weight_\Pi})^2\,\VERT{f}_{\gamma,\eta;T, \weight_{\mathrm{R}} }
		\end{equ}
		uniform over all $\psi \in \mathcal{B}_-$, $t\in(0,T]$, $x\in\R^3$, $r\in(0,1]$ and $f\in\dD^{\gamma,\eta}_{T,\weight_{\mathrm{S}}}$ and the bound
		\begin{equs}\label{eq:int-exp-bound-2}
			\overline\weight_{\mathrm{L}}(t,x)\,\langle \mathcal{R} f& - \Pi^{t-r^2,x} f (t - r^2 , x), \psi_{t,x}^r \rangle 
			\\
			&\lesssim 
			r^{\gamma - \fc} t^{(\eta - \gamma)/2} (1+\|\Pi,\Gamma\|_{T, \weight_\Pi})^2\,\VERT{f}_{\gamma,\eta;T, \weight_{\mathrm{R}} }
		\end{equs}
		uniform over all $\psi \in \mathcal{B}_-$, $t\in[4r^2,T]$, $x\in\R^3$, $r\in (0,1]$ and $f\in\dD^{\gamma,\eta}_{T,\weight_{\mathrm{S}}}$. Assuming these bounds the rest of the proof is the same as the proof of Theorem 4.3 in \cite{HL18} upon replacing everywhere $L^p$-type norms of the  form
		\begin{equ}
			\left\|r^{-d-\gamma}\int_{\R^d} 1_{\{|y-x|\leq r\}} |f(x,y)| \, \md y \right\|_{L^p(\R^d,\md x)}
		\end{equ}
		by H{\"o}lder-type norms of the form
		\begin{equ}
			\sup_{\substack{x,y\in\R^d\\|x-y|<1}}\frac{|f(x,y)|}{|x-y|^\gamma}.
		\end{equ}
		\item The bounds~\eqref{eq:int-exp-bound-1} and \eqref{eq:int-exp-bound-2} are proved using the argument from the proof of Theorem~3.10 in~\cite{HL18} taking as input the bound for the reconstruction operator stated in Lemma~\ref{lem:reconstruction} below. Note that the bounds~\eqref{eq:int-exp-bound-1} and \eqref{eq:int-exp-bound-2} involve different weights than the corresponding bounds (4.4) and (4.5) in~\cite{HL18} but this only reflects our different assumptions about the weights and does not require any further comment. 
	\end{enumerate}
	The proof of the bound for $\kK^-$ is the same as the proof of  Proposition~4.5 in~\cite{HL18} with the exception that one has to use~\eqref{eq:int-exp-bound-1} instead of (3.13) therein.
\end{proof}

To complete the proof of the above theorem, it remains to establish bounds on the reconstruction operator. The following lemma is an $L^\infty$-analogue of Theorem~2.10 in~\cite{HL18} and should be viewed as a refinement of the original proof of the reconstruction theorem~\cite[Theorem 3.10]{Hai14}.

\begin{lem}\label{lem:reconstruction}
	Let $\gamma\in(0,1/4)$. We have
	\begin{equ}\label{eq:bound_reconstrction}
		\sup_{\psi \in \bB} |\langle \rR f - \Pi^{t,x}f(t,x), \psi_{t,x}^r \rangle|  \lesssim r^\gamma\, C_{t,x,r}(\Pi,\Gamma, f)
	\end{equ}
	uniformly over $r\in(0, 1]$, $(t, x)\in\R^{1+3}$, $f\in \dD^\gamma$ and $(\Pi, \Gamma)\in\mM$, where	
	\begin{equ}\label{eq:def_C_t_x_r}
		C_{t,x,r}(\Pi,\Gamma, f) = \sum_{2^{-n} \leq r} \left( \frac{2^{-n}}{r} \right)^{\gamma} \|\Pi\|_{B^n_{r, t, x}} (1 + \|\Gamma\|_{B^n_{r, t, x}}) \VERT{f}_{B^n_{r, t, x}}
	\end{equ}
	with $B^n_{r, t, x} = [t - 2r^2, t + r^2 - 2^{-2n}] \times B(x, 3)\subset\R^{1+3}$.	In particular, 
	\begin{equ}
		\sup_{\psi \in \bB^-} \langle \mathcal{R} f - \Pi^{t-r^2,x} f (t - r^2 , x), \psi_{t,x}^r \rangle 
		\lesssim r^\gamma\, C_{t,x,r}(\Pi,\Gamma, f)
	\end{equ}
	uniformly over $r\in(0, 1]$, $(t, x)\in\R^{1+3}$, $f\in \dD^\gamma$ and $(\Pi, \Gamma)\in\mM$.
\end{lem}
\begin{rmk}
	Note that we have the elementary bound
	\begin{equ}\label{eq:reconstruction_trivial}
		C_{t,x,r}(\Pi,\Gamma, f) \lesssim
		\|\Pi\|_{B_{r, t, x}} (1 + \|\Gamma\|_{B_{r, t, x}}) \VERT{f}_{B_{r, t, x}}
	\end{equ}
	with $B_{r, t, x} = [t - 2r^2, t + r^2] \times B(x, 3)\subset\R^{1+3}$. 
	However, this estimate is insufficient to establish the bounds~\eqref{eq:int-exp-bound-1} and \eqref{eq:int-exp-bound-2}. Instead, we must rely on~\eqref{eq:def_C_t_x_r}, following the approach in~\cite{HL18}.
\end{rmk}
\begin{proof}
	The proof is almost identical to the proof of \cite[Theorem~2.10]{HL18}. The only difference is that instead of~\cite[Proposition~2.11]{HL18} one has to use~\cite[Theorem~3.23]{Hai14}.
\end{proof}

Now we discuss the action of the Euclidean group $\R^3 \rtimes O(3)$.
\begin{defn}
\label{def:Euclidane_transform}
	For an element $\varrho=(a,A)$ of the Euclidean group $\R^3 \rtimes O(3)$, we denote by $x\mapsto \varrho\cdot x\eqdef Ax+a$ its canonical action on $\R^3$, by $\varrho\cdot(t,x)\eqdef(t,\varrho\cdot x)$ its action on $\R^{1+3}$, by $\varrho \cdot f$ its action on $f\,:\,\R^{1+3}\to\R$ defined by $(\varrho \cdot f)(t,x)=f(t,\varrho^{-1}\cdot x)$, and by $\tau\mapsto \tau\cdot\varrho$ its action on the regularity structure determined uniquely by the conditions: 
	\begin{enumerate}
		\item $\iI(\tau)\cdot\varrho=\iI(\tau\cdot\rho)$ for all $\tau\in\{\Xi\}\cup\tT$ and $\Xi\cdot\rho = \Xi$,
		\item $(\tau\bar\tau)\cdot\varrho=(\tau\cdot\varrho)(\bar\tau\cdot\varrho)$ for all $\tau,\bar\tau\in\tT$,
		\item $\oneb\cdot\varrho=\oneb$, $\bXXX^0\cdot\varrho= \bXXX^0$ and $(\bXXX^k\cdot\varrho)_{1 \leq k \leq 3}= (\sum_{j} A_{kj}\bXXX^j)_{1 \leq k \leq 3}$\,.
	\end{enumerate}
	For a model $(\Pi,\Gamma)$ we define the transformed model $(\varrho\cdot\Pi,\varrho\cdot \Gamma)$ by
	\begin{equ}
		\langle(\varrho\cdot\Pi)^z\tau,\psi\rangle \eqdef \langle\Pi^{\varrho^{-1}\cdot z}(\tau\cdot\varrho),\varrho^{-1}\cdot\psi\rangle\,,
		\quad
		(\varrho\cdot\Gamma)^{z;\bar z}\tau \eqdef (\Gamma^{\varrho^{-1}\cdot z;\varrho^{-1}\cdot \bar z}(\tau\cdot\varrho))\cdot\varrho^{-1}
	\end{equ}
	for all $\tau\in\tT$, $z,\bar z\in\R^{1+3}$ and $\psi\in C^\infty_{\mathrm{c}}(\R^{1+3})$. For a (singular) modelled distribution $f$ we define $\varrho\cdot f$ by $z\mapsto f(\varrho^{-1}\cdot z)\cdot\varrho^{-1}$. 
\end{defn}
\begin{rmk}
	One verifies $(\varrho\cdot\Pi,\varrho\cdot \Gamma)=(\Pi,\Gamma)$ on the polynomial sector of $\tT$.
\end{rmk}
\begin{rmk}
	Using the identity
	\begin{equ}
		(\varrho \cdot f)(\varrho\cdot\bar z) 
		- \big((\varrho \cdot \Gamma)^{\varrho\cdot\bar z,\varrho\cdot z} (\varrho \cdot f)\big)(\varrho\cdot z) 
		=
		(f(\bar z)
		-\Gamma^{\bar z,z} f(z))\cdot \varrho^{-1} \;,
	\end{equ}	
	one shows that if $f \in \dD^{\gamma,\eta}(\Gamma)$, then $\rho \cdot f \in \dD^{\gamma,\eta}(\varrho\cdot\Gamma)$.
\end{rmk}

\begin{rmk}	
	Let $(\Pi,\Gamma)$ be an admissible model. Using $K^+(\varrho\cdot z-\varrho\cdot \bar z)=K^+(z-\bar z)$ one checks that $(\varrho\cdot\Pi,\varrho\cdot \Gamma)$ is also an admissible model. We denote by $\rR,\kK^\pm$ and $\rR_\varrho,\kK_\varrho^\pm$ the  reconstruction and integration operators corresponding to models $(\Pi,\Gamma)$ and $(\varrho\cdot\Pi,\varrho\cdot \Gamma)$. By uniqueness of the reconstruction operator and the identity
	\begin{equ}
		\bracket{(\varrho \cdot \Pi)^{\varrho\cdot z} (\varrho \cdot f)(\varrho\cdot z), \varrho\cdot\psi_{z}^r } 
		=
		\bracket{\Pi^{z} f(z) ,\psi_z^r}
	\end{equ}
	we have $\langle \rR f,\psi\rangle=\langle\rR_\varrho(\varrho\cdot f),\varrho\cdot\psi\rangle$.
	By Lemma~\ref{lem:kernel_decomposition}, we have $K^\pm(\varrho\cdot z-\varrho\cdot \bar z)=K^\pm(z-\bar z)$.
	In consequence, it follows from Definition~\ref{def:integration_K} that $\kK^\pm_\varrho(\varrho\cdot f)=\varrho\cdot (\kK^\pm f)$.
\end{rmk}

\subsection{Initial data contribution}
\label{sec:init_data}

Let $\eta=-\f12-\kappa$ and recall that $w=\langle \bigcdot\rangle^{-\kappa}\in C(\R^3)$. The following lemma shows that for any initial condition $\phi \in \cC^\eta(w)$, we can find a sequence of smooth periodic functions $\phi_{\eps,\ell}\in C^\infty(\T_\ell^3)$ such that $\lim_{\ell\to\infty}\lim_{\eps\searrow0}\phi_{\eps,\ell} = \phi$ in $\cC^\eta(w)$.

\begin{defn}\label{def:map_T}
	Let $\chi\in C^\infty(\R^3,\R_>)$ be such that $\chi=1$ on $[-1/3,1/3]^3$, $\supp\chi\subset[-1,1]^3$ and the periodisation of $\chi$ with period $1$ coincides with the constant function~$1$. For $\ell\in\N_+$ we define $T^{(\ell)}\,:\,\cC^\eta(w)\to \cC^\eta(\T_\ell^3)$ to be the unique map such that for all $\phi\in\cC^\eta(w)$, $T^{(\ell)}\phi$ coincides with the periodisation of $\phi\chi(\bigcdot/\ell)$ with period $\ell$.
\end{defn}

\begin{lem}\label{lem:map_T}
	Let $\eta<0$ and $\phi\in\cC^\eta(w)$. For $\ell\in\N_+$ and $\eps\in(0,1]$ define $\phi_{\eps,\ell}=M_\eps\star T^{(\ell)}\phi\in C^\infty(\T_\ell^3)$, 
	where $\star$ denotes the convolution over $\R^3$ 
	and the mollifier $M_\eps\in C^\infty(\R^3)$ is given by $M_\eps(x)=\eps^{-3}M(\frac{x}{\eps})$ for $M\in C^\infty_{\mathrm c}(\R^3)$ such that $\int M(x)\,\md x=1$.
	Then $\lim_{\ell\to\infty}\lim_{\eps\searrow0}\phi_{\eps,\ell}=\phi$ in $\cC^\eta(w)$.
\end{lem}
\begin{proof}
	We use the fact that the Besov space $\cC^\eta(w)$ is defined to be the completion of $C^\infty_{\mathrm c}(\R^3)$, see~\cite[Lemma~13]{MW17a} for a very similar result.
\end{proof}

\begin{defn}\label{def:K_S}
	For $\phi\in\cC^\eta(w)$ and $h\in L^\infty(\R_\geq \times\R^3,w)$ we write
	\begin{equ}
		K(\phi)(t,x) \eqdef
		\int_{\R^3} K(t,x-y)\,\phi(y)\,\md y,
		\quad
		S(h,\phi)(t,x) \eqdef 
		(K\ast \one_> h)(t,x) + K(\phi)(t,x),
	\end{equ}
	where $\one_>$ is the characteristic function of $\R_>$.
\end{defn}

\begin{lem}\label{lem:S_phi_estimate}
	Let $\eta< 0$, $\gamma\in(0,2)$ and $T>0$. For $h\in L^\infty([0,T]\times\R^3,w)$ and $\phi\in\cC^\eta(w)$, the function $S(h,\phi)$ admits a lift to a polynomial sector in $\dD^{\gamma,\eta}_{T,w}$. Moreover, we have
	\begin{equ}
		\VERT{S(h,\phi)}_{\gamma,\eta;T,w}\lesssim
		\|\phi\|_{\cC^\eta(w)} 
		+
		\|h\|_{L^\infty([0,T]\times\R^3,w)}
	\end{equ}
	uniformly over $h\in L^\infty([0,T]\times\R^3,w)$ and $\phi\in\cC^\eta(w)$.
\end{lem}
\begin{proof}
	We note that $S(h,\phi)=S(h,0)+S(0,\phi)$ and study separately $S(h,0)$ and $S(0,\phi)$. For $S(h,0)$ the result follows from standard properties of the heat kernel $K$. The statement concerning $S(0,\phi)=K(\phi)$ is a very similar to Lemma 7.5 in \cite{Hai14}. The only difference is the presence of the weight and the fact we control $\phi$ using the weighted Besov norm $\|\phi\|_{\cC^\eta(w)}$ instead of the norm from~\cite[Definition~3.7]{Hai14}. Note that $\VERT{K(\phi)}_{\gamma,\eta;T,w}=\|K(\phi)\|_{\gamma,\eta;T,w}+[K(\phi)]_{\gamma,\eta;T,w}$. By Lemma~\ref{lem:heat_smoothing_L_inf} for all $k\in\N_0^3$ we have
	\begin{equ}
		\|\partial^k K(\phi)(t,\bigcdot)\|_{L^\infty(w)} \leq t^{\frac{\alpha-|k|}{2}} \|\phi\|_{\cC^\eta(w)}
	\end{equ}
	uniformly over $t\in(0,T]$ and $\phi\in \cC^\alpha(w)$. Using the fact that $(\partial_t-\Delta)K(\phi)=0$ we conclude an analogous bound for all $k\in\N_0^{1+3}$. This proves the bound for $\|K(\phi)\|_{\gamma,\eta;T,w}$ for any $\gamma>0$. The bound for $[K(\phi)]_{\gamma,\eta;T,w}$ follows from the bound for $\|K(\phi)\|_{\gamma,\eta;T,w}$ and
	the generalised Taylor expansion from \cite[Proposition A.1]{Hai14}.
\end{proof}

\begin{lem}\label{lem:S_phi_standard}
	Let $\eta< 0$. For $h\in L^\infty([0,T]\times\R^3,w)$ and $\phi\in\cC^\eta(w)$ we have $S(h,\phi)\in C(\R_\geq,\cC^\eta(w))\cap C(\R_>,L^\infty(w))$. Moreover, we have
	\begin{equ}
		\|S(h,\phi)(t,\bigcdot)\|_{\cC^\eta(w)}
		\vee
		t^{-\eta/2}\,\|S(h,\phi)(t,\bigcdot)\|_{L^\infty(w)}
		\lesssim
		\|\phi\|_{\cC^\eta(w)} 
		+
		\|h\|_{L^\infty([0,T]\times\R^3,w)}
	\end{equ}
	uniformly over $h\in L^\infty([0,T]\times\R^3,w)$, $\phi\in\cC^\eta(w)$ and $t\in(0,1]$.
\end{lem}
\begin{proof}
	This follows more or less immediately from Lemma~\ref{lem:heat_smoothing_L_inf}.
\end{proof}

\subsection{A priori bounds}\label{sec:a_priori_solution}

The aim of this section is to establish a priori bounds for the solutions $\Phi_{\eps, \ell}$ of the mild form of~\eqref{e:phi4_mod} with $H = 0$ and the initial data $\phi_{\eps,\ell}$,
\begin{equ}\label{eq:Phi4_mild}
	\Phi_{\eps, \ell} 
	=
	K\ast \one_>\Bigl(
	\xi_{\eps,\ell}-\lambda \Phi_{\eps, \ell}^3 
	+\big(3 \lambda C_{\eps, \ell}^{(1)}-9\lambda^2 C_{\eps,\ell}^{(2)}\big)\Phi_{\eps, \ell}
	\Bigr)
	+ K(\phi_{\eps,\ell})\,.
\end{equ}
These bounds will yield the compactness of the family $(\Phi_{\eps, \ell})_{\eps\in(0,1],\ell\in\N_+}$. Since the precise value of the prefactor $\lambda > 0$ in front of the cubic nonlinearity plays no role in the analysis, we set $\lambda = 1$ throughout this and the following subsection to simplify the notation.

The results of the previous subsections are general and do not rely on a specific choice of model for the regularity structure $(\aA, \tT,G)$ introduced in Section~\ref{sec:reg_structure}. 
From this point onward, however, we focus on a particular model relevant for solving~\eqref{eq:Phi4_mild}. Specifically, we denote by $(\Pi_{\eps, \ell}, \Gamma_{\eps, \ell}) \in \mM(\weight_\Pi)$ the canonical model constructed from the spatially smooth, periodic noise $\xi_{\eps, \ell}$, following the procedure introduced in~\cite[Section~9.2]{Hai14}. Recall the definition of the renormalisation group for the dynamical $\Phi^4_3$ model from~\cite[Section~9.2]{Hai14}. The model obtained by applying the renormalisation map with parameters $C^{(1)}_{\eps, \ell}$ and $C^{(2)}_{\eps, \ell}$ to $(\Pi_{\eps, \ell}, \Gamma_{\eps, \ell})$ is denoted by $(\hat\Pi_{\eps, \ell}, \hat\Gamma_{\eps, \ell}) \in \mM(\weight_\Pi)$. We write $\kK_{\eps, \ell}$ and $\rR_{\eps, \ell}$ for the abstract integration and reconstruction maps associated to $(\hat\Pi_{\eps, \ell}, \hat\Gamma_{\eps, \ell})$. For notational convenience, we omit the indices $\eps, \ell$ when referring to objects in the limit $\ell \to \infty$ and $\eps \searrow 0$. We denote $\gamma\eqdef\f32-5\kappa$, $\bar\gamma\eqdef\f32-6\kappa$, $\eta\eqdef-\f12-\kappa$, $\bar\eta\eqdef-\f12-2\kappa$.

\begin{defn}\label{def:remainders}
	We use the shorthands
	\begin{equ}
		\ou_{\eps, \ell} \eqdef K\ast \xi_{\eps, \ell}\,,
		\qquad
		\ou^\pm_{\eps, \ell} \eqdef K^\pm\ast \xi_{\eps, \ell}\,,
		\qquad
		S_{\eps,\ell}(\phi)\eqdef S(\lL \ou_{\eps,\ell}^-,\phi-\ou^+_{\eps,\ell}(0))\,,
	\end{equ}
	where the map $S$ was introduced in Definition~\ref{def:K_S}.
	We define $v^+_{\eps, \ell},v^\star_{\eps, \ell}\in C(\R_\geq,C(\T_\ell^3))$ by the equalities
	\begin{equ}\label{eq:reminders}
		\Phi_{\eps, \ell}=\ou^+_{\eps, \ell}+v^+_{\eps, \ell}
		=
		\ou^+_{\eps, \ell}+v^\star_{\eps, \ell}+S_{\eps,\ell}(\phi_{\eps,\ell})\,,
	\end{equ}
	where $\Phi_{\eps, \ell}$ is the solution of~\eqref{eq:Phi4_mild}.
\end{defn}

The main result of \cite[Theorem~2.1]{MW20} implies almost immediately the following a priori bound.
\begin{lem}
	\label{le:MW20}
	Let $T=1$. We have
	\begin{equ}
		\sup_{t\in(0,T]}t^{\frac{1}{2}}\, 
		\|v^+_{\eps, \ell}(t)\|_{L^{\infty}(w^{1/3})}
		\lesssim
		1+\|(\hat\Pi_{\eps, \ell},\hat\Gamma_{\eps, \ell})\|_{T, \weight_\Pi}^3
		+
		\|\lL \ou^-_{\eps, \ell}\|_{L^\infty([0,T]\times\R^3,\weight_{\Pi})}
	\end{equ}
	uniformly over $\eps\in(0,1]$, $\ell\in\N_+$, $\xi_{\eps,\ell}\in C^{-\f12-\kappa}(\R,C(\T_\ell^3))$ and $\phi_{\eps,\ell}\in C(\T_\ell^3)$.
\end{lem}
\begin{proof}
	Let us define $\oud^+_{\eps,\ell},\out^+_{\eps,\ell}\in C(\R,C(\T_\ell^3))$ by
	\begin{equ}
		\oud^+_{\eps,\ell} \eqdef (\ou^+_{\eps,\ell})^2 - \Cone \;, 
		\quad
		\out^+_{\eps,\ell} \eqdef (\ou^+_{\eps,\ell})^3 - 3\Cone \ou^+_{\eps,\ell}\,.
	\end{equ}
	Using~\eqref{eq:Phi4_mild} we obtain
	\begin{equ}\label{eq:v_+}
		v^+_{\eps, \ell} = K\ast \one_> \Bigl(
		 - (v^+_{\eps, \ell})^3-3(v^+_{\eps, \ell})^2 \ou^+_{\eps, \ell}-3v^+_{\eps, \ell}\oud_{\eps, \ell}^+- \out_{\eps, \ell}^+
		- 9 C_{\eps, \ell}^{(2)} (\ou^+_{\eps, \ell}+v^+_{\eps, \ell})
		\Bigr) + S_{\eps,\ell}(\phi_{\eps, \ell})\,.
	\end{equ}
	As a result, $v^+_{\eps, \ell}$ is a weak solution of \eqref{e:deterministic_phi4} on $\R_\geq\times\R^3$ with $h_3 = \lL \ou^-_{\eps, \ell}$, $h_4 = 0$, $\ou = \ou^+_{\eps, \ell}$, $\oud = \oud^+_{\eps, \ell}$, $\out = \out^+_{\eps, \ell}$, $C^{(2)} = C_{\eps, \ell}^{(2)}$. Let $\oudz$, $\outz$, $h_1,h_2$ be as in Remark~\ref{rmk:spactime_loc_h}. By the estimate~\eqref{eq:spacetime_localization} from Theorem~\ref{thm:spacetime_localization}, we obtain
	\begin{equ}
		\| v^+_{\eps, \ell} \|_{L^{\infty}(\fK_r)}
		\lesssim
		\frac1r\vee \tilde \fX(\fK,h)
		\lesssim
		1\vee \frac1r\vee 
		\Big(\|(\hat\Pi_{\eps, \ell},\hat\Gamma_{\eps, \ell})\|_{\fK}
		\vee
		\|\lL \ou^-_{\eps, \ell}\|^{1/3}_\fK
		\Bigr)^{\frac{2}{1 - 2\kappa}}
	\end{equ}
for all space-time cubes $\fK\subset\R_\geq\times\R^3$ and $r\in(0,1]$, where $\tilde\fX(\fK)$ and $\fK_r$ are defined at the beginning of Appendix~\ref{sec:spacetime_localization} and $\|f\|_{\fK}\eqdef \sup_{z\in\fK}|f(z)|$. For any $t \in (0, 1)$ and $x \in \R^3$, we take $r = \sqrt{t}$ and $\fK = [0, 1] \times B(x, 1)$. Then the previous estimate implies
	\begin{equ}
		|v^+_{\eps, \ell}(t, x)| \lesssim \frac{1}{\sqrt{t}} 
		\vee
		\Big(\bigl(
		\|(\hat\Pi_{\eps, \ell},\hat\Gamma_{\eps, \ell})\|_{T, \weight_{\Pi}}
		\vee
		\|\lL \ou^-_{\eps, \ell}\|_{L^\infty([0,T]\times\R^3,\weight_{\Pi})}^{1/3}
		\bigr)\,
		\weight^{-1}_{\Pi}(x)
		\Big)^{\frac{2}{1 - 2\kappa}} \;.
	\end{equ}
	The result follows, since $w^{1/3}\leq\weight_\Pi^{\frac{2}{1-2\kappa}}$.
\end{proof}

The main input in this section is the following new a priori bound, which provides an improved blow-up rate as $t \searrow 0$ compared to the estimate in~\cite{MW20}.
\begin{lem}\label{le:apriori_bound_Linfty}
	Let $T=1$. There exists $M>0$ such that
	\begin{equs}
		\sup_{t\in[0,T]}&
		\|v^\star_{\eps, \ell}(t) \|_{L^{\infty}(w^3)}
		\vee 
		\sup_{t\in(0,T]} 
		t^{-\frac{\eta}{2}}\|v^+_{\eps, \ell}(t)\|_{L^{\infty}(w^3)}
		\\
		&\lesssim
		1
		+\|\phi_{\eps,\ell}\|_{\cC^\eta(w)}^M
		+\|(\hat\Pi_{\eps, \ell},\hat\Gamma_{\eps, \ell})\|_{T, \weight_\Pi}^M
		+\|\ou^+_{\eps,\ell}(0)\|_{\cC^\eta(w)}^M
		+\|\lL \ou^-_{\eps, \ell}\|_{L^\infty([0,T]\times\R^3,\weight_{\Pi})}^M
	\end{equs}
	uniformly over $\eps\in(0,1]$, $\ell\in\N_+$, $\xi_{\eps,\ell}\in C^{-\f12-\kappa}(\R,C(\T_\ell^3))$ and $\phi_{\eps,\ell}\in C(\T_\ell^3)$.
\end{lem}
\begin{proof}
Let us define $\ou^{\star}_{\eps,\ell},\oud^{\,\star}_{\eps,\ell},\out^{\,\star}_{\eps,\ell}\in C(\R,C(\T_\ell^3))$ by
\begin{equ}
	\ou^{\star}_{\eps,\ell}\eqdef \one_>\ou^+_{\eps, \ell} + S_{\eps,\ell}(\phi_{\eps, \ell})\;,
	\quad
	\oud^{\,\star}_{\eps,\ell} \eqdef (\ou^{\star}_{\eps,\ell})^2 - \one_>\Cone \;, 
	\quad
	\out^{\,\star}_{\eps,\ell} \eqdef (\ou^{\star}_{\eps,\ell})^3 - 3\Cone \ou^{\star}_{\eps,\ell}\,.
\end{equ}
Note that the above trees vanish on $\R_\leq\times\R^3$. We define $v^\star_{\eps, \ell}\in C(\R,C(\T_\ell^3))$ by the equality~\eqref{eq:reminders} on $\R_>$ and $v^\star_{\eps, \ell}(t)=0$ for $t\leq 0$. Using~\eqref{eq:Phi4_mild} one shows that
\begin{equ}\label{e:tilde_v_eq}
	v^\star_{\eps, \ell} =  K \ast \Big( - (v^\star_{\eps, \ell})^3-3(v^\star_{\eps, \ell})^2 \ou^{\star}_{\eps, \ell}-3v^\star_{\eps, \ell}\oud^{\,\star}_{\eps, \ell}- \out^{\,\star}_{\eps, \ell} -9 C_{\eps, \ell}^{(2)}(\ou^{\star}_{\eps, \ell}+v^\star_{\eps, \ell}) \Big) \;.
\end{equ}
We stress that the above equation does not involve $\one_>$ and is satisfied in entire space-time. Therefore, $v^\star_{\eps, \ell}$ is a weak solution of \eqref{e:deterministic_phi4} on $\R^{1+3}$ with $h_3 = h_4 = 0$, $\ou = \ou^{\star}_{\eps, \ell}$, $\oud = \oud^{\,\star}_{\eps, \ell}$, $\out = \out^{\,\star}_{\eps, \ell}$, $C^{(2)} = \one_> C_{\eps, \ell}^{(2)}$. Note that in this case, we choose $C^{(2)}$ to be time-dependent rather than a constant as in the proof of Lemma~\ref{le:MW20}. Let $\oudz$, $\outz$, $h_1,h_2$ be as in Remark~\ref{rmk:spactime_loc_h}. By the estimate~\eqref{eq:spacetime_localization} from Theorem~\ref{thm:spacetime_localization}, we obtain
\begin{equ}
	\| v^\star_{\eps, \ell} \|_{L^{\infty}(\fK_r)}
	\lesssim
	\frac1r\vee \tilde \fX(\fK)^{\frac{2}{1 - 2\kappa}}
\end{equ}
for all space-time cubes $\fK$ and $r\in(0,1]$. Taking $r = \frac{1}{2}$ and $\fK = [-1, 1] \times B(x, 1)$ with $x \in \R^3$ and using Lemma~\ref{lem:convergence_model_phi_deterministic} with $S=S_{\eps,\ell}(\phi_{\eps, \ell})$, we get
\begin{equs}
	\sup_{t \in [0, 1]}& |v^\star_{\eps, \ell}(t, x)| 
	\\
	&\lesssim \Big(
	(1+\|(\hat\Pi_{\eps, \ell},\hat\Gamma_{\eps, \ell})\|_{T, \weight_\Pi})^{4}\,
	(1+\VERT{S_{\eps,\ell}(\phi_{\eps, \ell})}_{\gamma,\eta;T,w})
	\,w^{-1}(x) \weight^{-4}_\Pi(x) \Big)^{\frac{2}{1 - 2\kappa}} \;.
\end{equs}
Observe that $w^3\leq(w\weight_\Pi^4)^{\frac{2}{1-2\kappa}}$ and by Lemma~\ref{lem:S_phi_estimate} we have
\begin{equ}
	\VERT{S_{\eps,\ell}(\phi_{\eps, \ell})}_{\gamma,\eta;T,w}
	\lesssim 
	1+\|\phi_{\eps,\ell}-\ou^+_{\eps,\ell}(0)\|_{\cC^\eta(w)}
	+
	\|\lL \ou^-_{\eps, \ell}\|_{L^\infty([0,T]\times\R^3,\weight_{\Pi})}\,.
\end{equ}
Thus, the desired bound for $v^\star_{\eps, \ell}$ follows. The corresponding bound for $v^+_{\eps, \ell}$ is a~consequence of the estimate for $v^\star_{\eps, \ell}$ together with Lemma~\ref{lem:S_phi_standard}.
\end{proof}

Now we upgrade the $L^{\infty}$ bound to a regularity bound for the corresponding modelled distributions.
\begin{lem}\label{lem:V_modelled}
	Let $V_{\eps, \ell}\in\dD^{\gamma,\eta}$ be the solution to the abstract fixed point problem
	\begin{equ}\label{eq:V_cutoff}
		V_{\eps,\ell} = -\kK_{\eps, \ell} \one_>  (\bou+V_{\eps, \ell})^3  + S_{\eps,\ell}(\phi_{\eps, \ell})
	\end{equ}
	associated with the model $(\hat \Pi_{\eps, \ell}, \hat \Gamma_{\eps, \ell})$, where we identified $S_{\eps,\ell}(\phi_{\eps, \ell})$ with its lift to $\dD^{\gamma,\eta}$. Then we have $v^+_{\eps, \ell} = \rR_{\eps, \ell} V_{\eps, \ell}=\bracket{\oneb^*,V_{\eps, \ell}}$. Moreover, $V_{\eps, \ell}$ takes the form 
	\begin{equ}
	\label{e:V_form}
		V_{\eps, \ell} =  \one_> \big( v^+_{\eps,\ell}\oneb -\boutz - 3  v^+_{\eps,\ell} \boudz +  v^{\sharp}_{\eps, \ell}\cdot \bXXX \big) \;,
	\end{equ}
	where 
	\begin{equ}
		v_{\eps, \ell}^{\sharp} = \nabla v^+_{\eps, \ell} + K^+ * \nabla \out^+_{\eps, \ell} + 3v^+_{\eps, \ell} \big(K^+ * \nabla\oud^+_{\eps, \ell} \big) \;.
	\end{equ}
\end{lem}
\begin{proof}
	The proof follows the same argument as in \cite[Proposition~9.10]{Hai14}.
\end{proof}

\begin{prop}\label{pr:apriori_bound_modelled_distribution}
	Let $T=1$. There exists $M>0$ such that
	\begin{equs}
		\VERT{V_{\eps, \ell}}&_{\gamma, \eta; T, w^9}
		\\
		&\lesssim
		1
		+\|\phi_{\eps,\ell}\|_{\cC^\eta(w)}^M
		+\|(\hat\Pi_{\eps, \ell},\hat\Gamma_{\eps, \ell})\|_{T, \weight_\Pi}^M
		+\|\ou^+_{\eps,\ell}(0)\|_{\cC^\eta(w)}^M
		+\|\lL \ou^-_{\eps, \ell}\|_{L^\infty([0,T]\times\R^3,\weight_{\Pi})}^M
	\end{equs}
	uniformly over $\eps\in(0,1]$, $\ell\in\N_+$, $\xi_{\eps,\ell}\in C^{-\f12-\kappa}(\R,C(\T_\ell^3))$ and $\phi_{\eps,\ell}\in C(\T_\ell^3)$.
\end{prop}
\begin{proof}
Recall that $v^+_{\eps, \ell}$ solves~\eqref{eq:v_+}.
We also note that, in the notation of Theorem \ref{thm:spacetime_localization}, for $V_{\eps, \ell}$ of the form \eqref{e:V_form}, we have
\begin{equ}
	V_{\eps, \ell}(z) - \Gamma_{z\bar{z}} V_{\eps, \ell}(\bar{z}) = - \Bar U_{\eps, \ell}(z, \bar{z}) \oneb - 3 (v^+_{\eps, \ell}(z) - v^+_{\eps, \ell}(\bar{z})) \boudz + (v^{\sharp}_{\eps, \ell}(z) - v^{\sharp}_{\eps, \ell}(\bar{z})) \bXXX \;.
\end{equ}
Let
\begin{equs}
	h_3 = \lL \ou^-_{\eps, \ell}\;, 
	\quad
	C^{(2)} =  C_{\eps, \ell}^{(2)} \;,
	\quad
	\,&\ou = \ou_{\eps, \ell}^{+}\;, \quad \oud = \oud_{\eps, \ell}^{+} \;,
	\quad
	\out = \out_{\eps, \ell}^{+}
\end{equs}
and $\oudz$, $\outz$, $h_1,h_2,h_4$ be as in Remark~\ref{rmk:spactime_loc_h}. Applying the estimate~\eqref{eq:spacetime_localization_U} from Theorem~\ref{thm:spacetime_localization} 
in the compact set $\fK = [t,1] \times B(x, 2)$ with $x \in \R^3$ and choosing $r = \sqrt{t}$, we get that
\begin{equs}
	\|\bar U_{\eps, \ell}\|_{\gamma,[2t,1]\times B(x,1)}
	\lesssim
	\left(
	\tilde\fX(\fK,h)\vee\|v^+_{\eps, \ell}\|_{\fK}
	\right)
	\,
	\left(
	\frac{1}{\sqrt{t}}\vee\tilde\fX(\fK,h)\vee\|v^+_{\eps, \ell}\|_{\fK}
	\right)^{\gamma}.
\end{equs}
By Remark~\ref{rmk:spactime_loc_h} and the definition of the model, we have 
\begin{equ}
	\tilde \fX(\fK,h)
	\lesssim
	\Big(\bigl(
	1\vee
	\|(\hat\Pi_{\eps, \ell},\hat\Gamma_{\eps, \ell})\|_{T, \weight_{\Pi}}
	\vee
	\|\lL \ou^-_{\eps, \ell}\|_{L^\infty([0,T]\times\R^3,\weight_{\Pi})}^{1/3}
	\bigr)\,
	\weight^{-1}_{\Pi}(x)
	\Big)^{\frac{2}{1 - 2\kappa}}\,.
\end{equ}
Then using Lemma~\ref{le:apriori_bound_Linfty} to bound $\|v^+_{\eps, \ell}\|_{\fK}$, we get
\begin{equs}
	\| \bar U_{\eps, \ell}& \|_{\gamma, [2t, 1] \times B(x, 1)}
	\lesssim t^{\frac{\eta - \gamma}{2}} w(x)^{-9}
	\\
	&
	\times\bigl(1
	+\|\phi_{\eps,\ell}\|_{\cC^\eta(w)}^M
	+\|(\hat\Pi_{\eps, \ell},\hat\Gamma_{\eps, \ell})\|_{T, \weight_\Pi}^M
	+\|\ou^+_{\eps,\ell}(0)\|_{\cC^\eta(w)}^M
	+\|\lL \ou^-_{\eps, \ell}\|_{L^\infty([0,T]\times\R^3,\weight_{\Pi})}^M
	\bigr)
\end{equs}
for some $M>0$.
The result follows, by applying similar arguments to $v^+_{\eps, \ell}$, $v^{\sharp}_{\eps, \ell}$, and subsequently invoking Remark~\ref{rmk:norm_cubes}.
\end{proof}

\subsection{Proof of Theorem~\ref{thm:global_solution}}
\label{sec:proof_markov_process_construction}

In this section, we combine the results from the preceding sections to construct a solution to the $\Phi^4_3$ model in infinite volume. We begin by stating two auxiliary lemmas, which follow directly from the analysis in Section~\ref{sec:modelled_distributions}. Note that the parameters $\gamma=\f32-5\kappa$, $\bar\gamma=\f32-6\kappa$, $\eta=-\f12-\kappa$, $\bar\eta=-\f12-2\kappa$ satisfy the conditions $\bar\gamma+2\bar\eta+2,3\bar\eta+2\notin\N_0$ as well as $\tT_{[\bar\gamma,\gamma)}=\emptyset$ and $\gamma+2\eta+2-\fc,3\eta+2-\fc\notin\N_0$, $\tT_{[\gamma+2\eta-\fc,\gamma+2\eta)}=\emptyset$ with $\fc=2\bar\kappa$. These parameters are considered fixed throughout this section, so we will not
separately recall their values in each of the statements.
The same goes for the weights $w,\weight_\Pi,\weight_{\mathrm{S}},
\weight_{\mathrm{L/R}}$ as defined in the statement of Lemma~\ref{lem:weights_solution}.

\begin{lem}
	\label{le:abstract_convolution_estimates_new}
	For all $T\in(0,1]$ we have
	\begin{equ}
		\VERT{\kK f^3}_{\gamma,0;T,\weight_{\mathrm{S}}^3\weight_\Pi^7} 
		\lesssim 1
	\end{equ}
	and
	\begin{equ}
		\VERT{\kK f^3;\bar\kK \bar f^3}_{\gamma,0;T,\weight_{\mathrm{S}}^3\weight_\Pi^7} 
		\lesssim 
		\VERT{f;\bar f}_{\bar\gamma,\bar\eta;T, \weight_{\mathrm{S}} }+\|(\Pi,\Gamma)-(\bar\Pi,\bar\Gamma)\|_{T, \weight_\Pi}\,
	\end{equ}
	locally uniformly over
	\begin{equ}
		(\Pi,\Gamma),(\bar\Pi,\bar\Gamma)\in\mM(\weight_\Pi)
		\,,\quad
		f\in\dD^{\bar\gamma,\bar\eta}_{T,\weight_{\mathrm{S}}}(\tT_{[\bar\eta,\bar\gamma)},\Gamma)
		\,,\quad
		\bar f\in\dD^{\bar\gamma,\bar\eta}_{T,\weight_{\mathrm{S}}}(\tT_{[\bar\eta,\bar\gamma)},\bar\Gamma)\,.
	\end{equ}
\end{lem}
\begin{proof}
	The result follows from Lemmas~\ref{lem:embedding_modelled_distributions},~\ref{lem:multiplication},~\ref{lem:integration_K} and Remark~\ref{rmk:embedding_modelled_distributions}.
\end{proof}

\begin{lem}
\label{le:linearise_abstract_equ}
	Fix a model $(\Pi,\Gamma)\in\mM(\weight_\Pi)$. For all $T\in(0,1]$ we have
	\begin{equ}
		\VERT{f}_{\gamma,\eta;T, \weight_{\mathrm{L}}}
		\vee
		\VERT{f-K(\phi)}_{\gamma,0;T, \weight_{\mathrm{L}}} \lesssim \|\phi\|_{\cC^\eta(w)}
	\end{equ}
	locally uniformly over $f\in\dD^{\gamma,\eta}_{T,\weight_{\mathrm{S}}}(\tT_{[0,\gamma)})$, $g\in\dD^{\gamma,\eta}_{T,\weight_{\mathrm{S}}}(\tT_{[\eta,\gamma)})$ and $\phi\in\cC^\eta(w)$ satisfying the equation
	\begin{equ}
		f = K(\phi)-\kK (g^2 f)  \;.
	\end{equ}
\end{lem}
\begin{proof}
	First, note that by Lemmas~\ref{lem:multiplication},~\ref{lem:multiplication_exp} and Theorem~\ref{thm:integration_K_exp},
	\begin{equ}
		\VERT{\kK (g^2 f)}_{\gamma-2\eta+2-\fc,3\eta+2-\fc;T, \weight_{\mathrm{L}} }
		\lesssim
		\VERT{g}^2_{\gamma,\eta;T, \weight_{\mathrm{S}} }\,\,\VERT{f}_{\gamma,\eta;T, \weight_{\mathrm{L}} }
	\end{equ}
	uniformly over $(\Pi,\Gamma)\in\mM(\weight_\Pi)$, $g,f\in\dD^{\gamma,\eta}_{T,\weight_{\mathrm{S}}}(\tT_{[\eta,\gamma)})$ and $T\in(0,1]$. Hence, by Lemma~\ref{lem:embedding_modelled_distributions} and Remark~\ref{rmk:embedding_modelled_distributions}
	\begin{equ}\label{e:uniqueness_main}
		\VERT{\kK (g^2 f)}_{\gamma,\eta;T, \weight_{\mathrm{L}} }
		\lesssim
		\VERT{\kK (g^2 f)}_{\gamma,0;T, \weight_{\mathrm{L}} }
		\lesssim
		T^{\frac{\kappa}{2}}\,
		\VERT{g}^2_{\gamma,\eta;T, \weight_{\mathrm{S}} }\,
		\VERT{f}_{\gamma,\eta;T, \weight_{\mathrm{L}} }\,.
	\end{equ}
	Next, observe that by Lemma~\ref{lem:S_phi_estimate},
	\begin{equ}
		\VERT{K(\phi)}_{\gamma,\eta;T,w}\lesssim
		\|\phi\|_{\cC^\eta(w)} \,.
	\end{equ}
	This proves the bound
	\begin{equ}\label{eq:bound_f_phi}
		\VERT{f}_{\gamma,\eta;T, \weight_{\mathrm{L}}}
		\lesssim
		\|\phi\|_{\cC^\eta(w)}
	\end{equ}
	for $T^{\frac{\kappa}{2}}\leq 1\wedge \frac12\VERT{g}^{-2}_{\gamma,\eta;1, \weight_{\mathrm{S}}}$. To extend this bound to all $T \in (0,1]$, we employ a time iteration argument, analogous to the one used in the proof of~\cite[Theorem~5.2]{HL18}. Using~\eqref{e:uniqueness_main} and~\eqref{eq:bound_f_phi} we complete the proof.
\end{proof}

We adopt a deterministic perspective and construct the solution map pathwise, on the event
given by the following lemma.

\begin{lem}\label{lem:event}
	There exist 
	\begin{equ}
		(\hat\Pi,\hat\Gamma)\in\mM(\weight_\Pi)\,,
		\quad
		\lL\ou^-\in C(\R_\geq\times\R^3)\,,
		\quad
		\ou^+\in C(\R_\geq,\cC^{\eta+\frac\kappa2}(\weight_\Pi))
	\end{equ}
	such that almost surely, for all $T>0$, the stochastic data $(\hat\Pi_{\eps,\ell},\hat\Gamma_{\eps,\ell},\ou^+_{\eps,\ell},\lL\ou^-_{\eps,\ell})_{\eps\in(0,1],\ell\in\N_+}$ satisfies the conditions
	\begin{equs}
		\lim_{\ell\to\infty}\lim_{\eps\searrow0}\|(\hat\Pi,\hat\Gamma)-(\hat\Pi_{\eps,\ell},\hat\Gamma_{\eps,\ell})\|_{T, \weight_\Pi}&=0 \,,
		\label{e:good_event1} 
		\\
		\lim_{\ell\to\infty}\lim_{\eps\searrow0} \sup_{t \in [0, T]} \|\lL\ou^-(t) -\lL\ou^-_{\eps,\ell}(t)\|_{L^\infty(\weight_\Pi)}&=0 \,,
		\label{e:good_event2} 
		\\
		\lim_{\ell\to\infty}\lim_{\eps\searrow0}\sup_{t\in[0,T]}\|\ou^+(t)-\ou^+_{\eps,\ell}(t)\|_{\cC^{\eta +\frac\kappa2}(\weight_\Pi)}&=0 \,.
		\label{e:good_event3}
	\end{equs}
\end{lem}

\begin{proof}
	The statement relies crucially on the coupling of the family $(\xi_{\eps,\ell})_{\eps \in(0,1],\ell\in\N_+}$ with the space-time white noise $\xi$ introduced in Definition~\ref{def:smoothed_noise} and follows immediately from Lemmas~\ref{lem:convergence_stationary_model} and \ref{lem:stochastic_paracontrolled}. 
\end{proof}

Now we are ready to prove Theorem~\ref{thm:global_solution}. 
\begin{proof}[Proof of Theorem~\ref{thm:global_solution}]
	We work deterministically on the event of full measure on which the conclusions of Lemma~\ref{lem:event} hold.
	
	\step{Construction of $\Phi$.}\label{step:construction_Phi} Set $T=1$ and recall that $\mathrm{w}_{\mathrm{S}} \leq w^9$. Given $\phi\in \cC^{-\frac12-\kappa}(w)$ we define the initial data for the regularised dynamic $\phi_{\eps,\ell} \in C^\infty(\T_\ell^3)$ as in Lemma~\ref{lem:map_T}. Then $\lim_{\ell\to\infty}\lim_{\eps\searrow0}\|\phi-\phi_{\eps,\ell}\|_{\cC^\eta(w)}=0$, and by Lemma~\ref{lem:S_phi_estimate} we have $S(\phi)\in\dD^{\gamma,\eta}_{T,w}$ and
	\begin{equ}\label{eq:S_v}
		\lim_{\ell\to\infty}\lim_{\eps\searrow0}\VERT{S(\phi) - S_{\eps,\ell}(\phi_{\eps,\ell})}_{\gamma,\eta;T,w} =0 \;.
	\end{equ}
	Let $V_{\eps,\ell}\in\dD^{\gamma,\eta}$ be the solution to the abstract fixed point problem~\eqref{eq:V_cutoff}. By Proposition~\ref{pr:apriori_bound_modelled_distribution} and Lemma~\ref{le:compactness_modelled_distribution}, for every sequence $(\bar\eps_n,\bar\ell_n)_{n\in\N_+}$ there exists a subsequence $(\eps_n,\ell_n)_{n\in\N_+}$ and a singular modelled distribution $V\in\dD^{\gamma,\eta}_{T,\weight_{\mathrm S}}$ with respect to the model $(\hat \Pi, \hat \Gamma)$ such that 
	\begin{equ}\label{eq:V_accumulation}
		\lim_{n\to\infty}\VERT{V_{\eps_n, \ell_n}; V}_{\bar \gamma, \bar \eta;T,\weight_{\mathrm S}}=0\,.
	\end{equ}
	Using Lemma~\ref{le:abstract_convolution_estimates_new}, the fact that $V_{\eps,\ell}$ satisfies~\eqref{eq:V_cutoff} and the conditions \eqref{e:good_event1} and~\eqref{eq:S_v} we obtain
	\begin{equ}\label{e:limit_characterisation}
		V = \kK \one_> (\bou+V)^3 + S(\phi)
	\end{equ}
	on $[0,1]\times\R^3$ for all $V\in\dD^{\gamma,\eta}_{T,\weight_{\mathrm S}}$ such that~\eqref{eq:V_accumulation} holds. Now, suppose that $V, \bar V\in\dD^{\gamma,\eta}_{T,\weight_{\mathrm S}}$ solve~\eqref{e:limit_characterisation} with the initial data $\phi$ and $\bar \phi$, respectively. Then the difference $D = V - \bar V$ satisfies the equation
	\begin{equ}\label{e:difference_eq}
		D = \kK \one_>
		\Bigl(
		(\bou + V)^2 + (\bou + V)(\bou + \bar V) + (\bou + \bar V)^2 
		\Bigr) D
		+
		K(\phi - \bar \phi) \;.
	\end{equ}
	Applying Lemma~\ref{le:linearise_abstract_equ}, in the case $\phi = \bar \phi$ we get that $D=V-\bar V=0$, which implies that~\eqref{e:limit_characterisation} admits a unique solution in $\dD^{\gamma,\eta}_{T,\weight_{\mathrm S}}$. Therefore, there exists a unique modelled distribution $V\in \dD^{\gamma,\eta}_{T,\weight_{\mathrm S}}$ solving \eqref{e:limit_characterisation} such that
\begin{equ}\label{e:subsequential_limit}
	\lim_{\ell\to\infty}\lim_{\eps\searrow0}\VERT{V_{\eps,\ell};V}_{\bar\gamma,\bar\eta;T, \weight_{\mathrm{S}} } = 0 \;.
\end{equ}
Using \eqref{e:subsequential_limit}, \eqref{eq:S_v}, Lemma~\ref{le:abstract_convolution_estimates_new} and the fact that $V_{\eps,\ell}$  and $V$ satisfy~\eqref{eq:V_cutoff} and~\eqref{e:limit_characterisation}, respectively, one shows that
\begin{equ}
	\lim_{\ell\to\infty}\lim_{\eps\searrow0} \VERT{V_{\eps, \ell}-S_{\eps,\ell}(\phi_{\eps,\ell}); V-S(\phi)}_{\gamma,0; T, \weight_{\mathrm{S}}^3\weight_\Pi^7} = 0 \;.
	\end{equ}
Recall Definition~\ref{def:remainders}. Since $v^\star_{\eps, \ell} = \bracket{\oneb^*,V_{\eps, \ell}-S_{\eps,\ell}(\phi_{\eps,\ell})}\in C([0, 1] \times \T_\ell^3)$, using Definition~\ref{def:modelled_distributions} we obtain that there exists $v^\star\in C([0, 1] \times \R^3)$ such that 
\begin{equ}
\label{e:v_convergence}
	\lim_{\ell\to\infty}\lim_{\eps\searrow0} \sup_{t\in (0, 1]}  
	\|v^\star_{\eps, \ell}(t)-v^\star(t))\|_{L^{\infty}(\weight_{\mathrm{S}}^3\weight_\Pi^7)} 
	= 0 \;.
\end{equ}
By~\eqref{e:v_convergence}, Lemma~\ref{le:apriori_bound_Linfty} and $v^\star_{\eps, \ell}\in C([0, 1],L^{\infty}(w^3))$ we obtain
\begin{equ}
	\label{e:v_convergence_improved}
	\lim_{\ell\to\infty}\lim_{\eps\searrow0} \sup_{t\in [0,1]}
	\|v^\star_{\eps, \ell}(t)-v^\star(t))\|_{L^{\infty}(w^4)}
	= 0
\end{equ}
and
\begin{equ}
	v^\star \in C([0, 1],L^{\infty}(w^4))
	\subset C([0,1],\cC^\eta(w^4)) \;.
\end{equ}
By Lemma~\ref{lem:S_phi_standard} we have
\begin{equ}\label{e:convergence_S}
	\lim_{\ell\to\infty}\lim_{\eps\searrow0}\sup_{t>0}t^{-\frac\eta2}\|S_{\eps,\ell}(\phi_{\eps,\ell})(t)-S(\phi)(t)\|_{L^\infty(w)}=0
\end{equ}
and
\begin{equ}
	S(\phi) \in C(\R_\geq,\cC^\eta(w))
	\subset C(\R_\geq,\cC^\eta(w^4)) \;.
\end{equ}
By~\eqref{e:v_convergence_improved}, \eqref{e:convergence_S}, Lemma~\ref{le:MW20} and $v^+_{\eps,\ell}\in C([0,1],L^\infty(w^{1/3}))$ we furthermore infer that
\begin{equ}
	\label{e:v_+_convergence_improved}
	\lim_{\ell\to\infty}\lim_{\eps\searrow0} \sup_{t\in (0, 1]}
	t^{-\frac\eta2}\,\|v^+_{\eps, \ell}(t)-v^+(t))\|_{L^{\infty}(w^{1/2})} 
	= 0
\end{equ}
and
\begin{equ}
	v^+=v^\star+ S(\phi) \in C((0, 1],L^{\infty}(w^{1/2}))
	\subset C((0, 1],\cC^{\eta+\frac\kappa2}(w^{1/2}))\,.
\end{equ}
We also note that by Lemma~\ref{lem:event} we have
\begin{equ}
	\ou^+\in 
	C(\R_\geq,\cC^{\eta+\frac\kappa2}(\weight_\Pi))\subset
	C(\R_\geq,\cC^{\eta+\frac\kappa2}(w^{1/2}))
	\subset C(\R_\geq,\cC^{\eta}(w^4))\,.
\end{equ}
For $t \in [0, 1]$ and a realisation $\xi$ of the white noise, we define
\begin{equ}
	\Phi(\phi;\bigcdot)\equiv\Phi(\phi,\xi;\bigcdot) \eqdef \ou^+ + v^+
	=
	\ou^+ + v^\star +S(\phi)
\end{equ}
and from the previous discussions we conclude immediately that
\begin{equ}
	\Phi(\phi;\bigcdot)\in C([0,1],\cC^\eta(w^4))\cap C((0,1],\cC^{\eta+\frac{\kappa}{2}}(w^{\frac12}))
\end{equ}
This proves~\eqref{e:solMap_inf} restricted to the time interval $[0,1]$. The convergence \eqref{e:Phi_convergence} for $t\in(0,1]$ follows from~\eqref{e:v_+_convergence_improved} and~\eqref{e:good_event3}. The bound \eqref{e:bound} follows directly from the definition of $\Phi$, \eqref{e:v_+_convergence_improved} and Lemmas~\ref{le:MW20}, \ref{lem:convergence_stationary_model} and~\ref{lem:stochastic_paracontrolled}.

From \eqref{e:Phi_convergence} and the properties of the finite volume dynamic we deduce that $\Phi$ satisfies the cocycle property:
\begin{equ}
\label{e:cocycle}
	\Phi(\phi,\xi; t_1 + t_2) = \Phi(\Phi(\phi,\xi; t_1),\theta(t_1)\,\xi; t_2)
\end{equ}
for all $0 < t_1 + t_2 \leq 1$ and $\phi \in \cC^{-\frac{1}{2} - \kappa}(w)$, where $\theta(t)\,\xi$ denotes the noise obtained by shifting $\xi$ by $t$ into the past. Using \eqref{e:cocycle} and the definition of $\Phi(\phi;t)$ for $t \in [0, 1]$, one defines $\Phi(\phi; t)$ iteratively for all $t\geq 0$ and verifies that it satisfies~\eqref{e:solMap_inf} and~\eqref{e:Phi_convergence}.

\step{Continuity of $\Phi$ with respect to initial data.}
By \eqref{e:cocycle}, it suffices to study the time interval $[0, 1]$. Let $V, \bar V\in\dD^{\gamma,\eta}_{T,\weight_{\mathrm S}}$ be solutions of~\eqref{e:limit_characterisation} with initial data $\phi,\bar\phi\in\cC^\eta(w)$, respectively. Set
\begin{equ}
	v^+ = \bracket{\oneb^*,V}\,,
	\quad
	v^\star = v^+-S(\phi)\,,
	\quad
	\bar v^+ = \bracket{\oneb^*,\bar V}\,,
	\quad
	\bar v^\star = \bar v^+-S(\bar\phi)\,.
\end{equ}
Applying Lemma~\ref{le:linearise_abstract_equ} to \eqref{e:difference_eq} and using Definition~\ref{def:modelled_distributions} we obtain
\begin{equs}
	\sup_{t\in(0,1]}
	t^{-\frac\eta2}\,\|v^+(t)-\bar v^+(t)\|_{L^{\infty}(\weight_{\mathrm{L}}(t))}
	\vee
	\sup_{t\in(0,1]}
	\|v^\star(t)-\bar v^\star(t)\|_{L^{\infty}(\weight_{\mathrm{L}}(t))}
	\lesssim \|\phi- \bar\phi\|_{\cC^{\eta}(w)} \;.
\end{equs}
By~\eqref{e:v_convergence_improved} and Lemma~\ref{le:apriori_bound_Linfty} we have $\sup_{t\in[0,1]}
\|v^\star(t)-\bar v^\star(t)\|_{L^{\infty}(w^3)}\lesssim 1$ locally uniformly over initial data $\phi,\bar\phi$. Similarly, by~\eqref{e:v_+_convergence_improved} and Lemma~\ref{le:MW20} we have $\sup_{t\in(0,1]}
t^{1/2}\,\|v^+(t)-\bar v^+(t)\|_{L^{\infty}(w^{1/3})}\lesssim 1$ uniformly over initial data $\phi,\bar\phi$.

The continuity of the map~\eqref{e:solMap_inf} follows now from Lemma~\ref{lem:S_phi_standard} and the following observation: given $\beta>\alpha\geq 0$ and an interval $I\subset\R$,
if $\lim_{n\to\infty}\sup_{t\in I}\|f_n(t)\|_{L^\infty(\weight_{\mathrm{L}}(t))}=0$ and $\sup_{n\in\N_+}\sup_{t\in I}\|f_n(t)\|_{L^\infty(w^\alpha)}<\infty$, then $\lim_{n\to\infty}\sup_{t\in I}\|f_n(t)\|_{L^\infty(w^\beta)}=0$.

\step{Euclidean invariance.} From~\ref{step:construction_Phi} we know that for any $\phi \in \cC^{\eta}(w)$ there exists a unique singular modelled distribution $V$ such that \eqref{e:limit_characterisation} holds, where $S(\phi) = \kK \one_> \lL \ou^- + K(\phi - \ou^+(0))$. Recall Definition~\ref{def:Euclidane_transform}. For any $\varrho \in \R^3 \rtimes O(3)$, by~(A) and direct calculation, we get that $\varrho\cdot V$ is the unique solution of the equation
\begin{equ}
	\varrho\cdot V = \kK_{\varrho} \one_> (\bou + \varrho\cdot V)^3 + \kK_{\varrho} \one_> (\varrho\cdot \lL\ou^- ) + K( \varrho\cdot \phi - \varrho\cdot \ou^+(0)) \;.
\end{equ}
Here both sides are singular modelled distribution with respect to $(\varrho \cdot \hat \Pi, \varrho \cdot \hat \Gamma)$. From our definition of $\Phi$ and Lemma~\ref{lem:convergence_stationary_model} we have 
\begin{equ}
	\Phi(\varrho \cdot \phi; \bigcdot, \varrho \cdot \xi) = \bracket{\oneb^*, \varrho\cdot V} + \varrho\cdot \ou^+ = \varrho \cdot \Phi(\phi; \bigcdot, \xi) \;.
\end{equ}
The result then follows, since $\varrho \cdot \xi \eqlaw \xi$.
\end{proof}

\appendix

\section{Space-time localisation bound}\label{sec:spacetime_localization}

Recall that $|(t,x)-(s,y)|\eqdef \max\{|t-s|^{1/2},|x-y|\}$ is the usual parabolic distance of $(t,x),(s,y)\in\R^{1+3}$. For $z=(t,x)\in\R^{1+3}$ we set $X(z)=x$. We denote by $B_-(z,r)$ the parabolic ball of centre $z=(t,x)$ and radius $r>0$ with respect to the parabolic distance looking only into the past. We define the parabolic boundary of a subset $\fK$ of space-time as the set of points $z$ in the closure $\fK^{\mathrm{cl}}$ of $\fK$ such that $B_-(z,r)\not\subset\fK^{\mathrm{cl}}$ for all $r>0$. For $r>0$ we define $\fK_r\subset \fK$ as the set at distance $r$ from the parabolic boundary. We call a set $\fK\subset\R^{1+3}$ a space-time cube if $\fK=I_0\times \ldots\times I_3$ for some closed intervals $I_0,\ldots,I_3\subset \R$. For $\fK\subset\R^{1+3}$, $\alpha > 0$ and a weight $\weight\in C(\R^{1+3},\R_>)$, we let $\|f\|_{\fK,\weight}\eqdef \sup_{z\in\fK} \weight(z)|f(z)|$ and denote by $[f]_{\alpha,\fK,\weight}$ the weighted $\alpha$-Hölder seminorm restricted to points in $\fK$ defined with the use of the parabolic distance. Let $(\psi^r)_{r\in(0,1]}$ be a family of smooth compactly supported test functions over space-time with a~semigroup property at dyadic scales constructed in \cite[Section~2]{MW20}. For an open set $\fK\subset\R^{1+3}$, a weight $\weight\in C(\R^{1+3},\R_>)$ and $\alpha<0$ we define the local Besov $\cC^\alpha$ norm of a distribution $f\in\dD'(\fK)$ by 
\begin{equ}\label{e:weighted_Holder_norm}
 [f]_{\alpha,\fK,\weight}\eqdef\sup_{r\in(0,1]} r^{-\alpha} \|\weight(\psi^r\ast f)\|_{\fK}\,.
\end{equ}
We omit the weight $\weight$ if $\weight=1$.

\begin{defn}\label{def:enhanced_noise_spacetime}
Given a space-time region $\fK\subset\R^{1+3}$ and functions $\ou,\oud,\out,\oudz,\outz,\weight,C^{(2)}$ we define
\begin{equs}[eq:def_X_tilde_st]
	\tilde\fX&(\ou,\oud,\out,\oudz,\outz,C^{(2)},\fK,\weight)
	\eqdef
	\max\Bigl\{
	[\ou]_{|\bou|,\fK,\weight},
	[\oud]_{|\boud|,\fK,\weight^2}^{1/2},
	[\out]_{|\bout|,\fK,\weight^3}^{1/3},
	\\
	&	
	\|\oudz\|_{\fK,\weight^2}^{1/2},
	[\oudz]_{|\boudz|,\fK,\weight^2}^{1/2},
	\|\outz\|_{\fK,\weight^3}^{1/3},
	[\outz]_{|\boutz|,\fK,\weight^3}^{1/3},
	[\oud \XXX]_{\fK,\weight^2}^{1/2},
	[\oudd]_{\fK,\weight^4}^{1/4},
	[\outu]_{\fK,\weight^4}^{1/4},
	[\outd]_{\fK,\weight^5}^{1/5}\Bigr\}\,,
\end{equs}
where
\begin{equs}{}
[\oud \XXX]_{\fK,\weight}&\eqdef\sup_{r\in(0,1]}r^{-|\boud \bXXX|}
\left\|\weight\int X(z-\bigcdot)\oud(z)\psi^r(z-\bigcdot)\,\md z\right\|_{\fK}\,,
\\{}
[\oudd]_{\fK,\weight}&\eqdef\sup_{r\in(0,1]}r^{-|\boudd|}
\left\|\weight\int((\oudz(z)-\oudz(\bigcdot))\oud(z)-C^{(2)}(z))\psi^r(z-\bigcdot)\,\md z\right\|_{\fK}\,,
\\{}
[\outu]_{\fK,\weight}&\eqdef\sup_{r\in(0,1]}r^{-|\boutu|}
\left\|\weight\int(\outz(z)-\outz(\bigcdot))\ou(z)\psi^r(z-\bigcdot)\,\md z\right\|_{\fK}\,,
\\{}
[\outd]_{\fK,\weight}&\eqdef\sup_{r\in(0,1]}r^{-|\boutd|}
\left\|\weight\int((\outz(z)-\outz(\bigcdot))\oud(z)-3C^{(2)}(z)\ou(z))\psi^r(z-\bigcdot)\,\md z\right\|_{\fK}\,.
\end{equs}
Here $|\bou|$, $|\boud|$, etc.\ refer to the grading defined in Section~\ref{sec:reg_structure}. We omit $\weight$ if $\weight=1$.
\end{defn}

\begin{thm}\label{thm:spacetime_localization}
There exists a constant $C>0$ such that the following statement is true. Let $\fK\subset\R^{1+3}$ be a space-time cube. Recall that $\lL=\partial_t-\Delta+1$. Suppose that $\ou,\oud,\out,C^{(2)},h_1,h_2,h_3,h_4\in L^\infty_{\mathrm{loc}}(\R^{1+3})$ and $\oudz,\outz,v\in C^{0,1}(\R^{1+3})$ satisfy the relations
\begin{equ}
\label{e:relation_trees}
\lL\oudz=\oud+h_1\,,
\qquad
\lL\outz=\out+h_2\,,
\end{equ}
\begin{equ}
\label{e:deterministic_phi4}
\lL v=-v^3-3v^2\ou-3v\oud-\out-9C^{(2)}(v+\ou)
+h_3+h_4 v
\end{equ}
in the weak sense in $\fK$. Recall that $\tilde\fX(\fK)=\tilde\fX(\ou,\oud,\out,\oudz,\outz,C^{(2)},\fK)$ was introduced in Definition~\ref{def:enhanced_noise_spacetime} and set
\begin{equ}\label{eq:tilde_fX_h}
\tilde\fX(\fK,h)
\eqdef
\left(\tilde\fX(\fK)\vee\|h_1\|_\fK^{1/2}\vee\|h_2\|_\fK^{1/3}\vee \|h_3\|_\fK^{1/3}\vee\|h_4\|_\fK^{1/2}\right)^{2/(1-2\kappa)}\,.
\end{equ} 
Then
\begin{equ}
\label{eq:spacetime_localization}
\|v\|_{\fK_r}\leq C\,(1/r\vee \tilde\fX(\fK,h))
\end{equ}
and
\begin{equs}
\label{eq:spacetime_localization_holder}
[v]_{1/2-3\kappa,\fK_r}
&\leq C\,\left(\tilde\fX(\fK,h)\vee\|v\|_{\fK}\right)\,
\left(\frac{1}{r}\vee \tilde\fX(\fK,h)\vee\|v\|_{\fK}\right)^{1/2-3\kappa}\,,
\\{}
\label{eq:spacetime_localization_holder2}
[v+\outz]_{1-2\kappa,\fK_r}
&\leq C\,\left(\tilde\fX(\fK,h)\vee\|v\|_{\fK}\right)\,
\left(\frac{1}{r}\vee \tilde\fX(\fK,h)\vee\|v\|_{\fK}\right)^{1-2\kappa}\,,
\\{}
\label{eq:spacetime_localization_U}
[\bar U]_{3/2-5\kappa,\fK_r}
&\leq C\,\left(\tilde\fX(\fK,h)\vee\|v\|_{\fK}\right)\,
\left(\frac{1}{r}\vee \tilde\fX(\fK,h)\vee\|v\|_{\fK}\right)^{3/2-5\kappa}\,,
\\{}
\label{eq:spacetime_localization_V}
\|v^\sharp\|_{\fK_r}&\leq C\,\left(\tilde\fX(\fK,h)\vee\|v\|_{\fK}\right)\,
\left(\frac{1}{r}\vee \tilde\fX(\fK,h)\vee\|v\|_{\fK}\right)\,,
\\{}
\label{eq:spacetime_localization_V_holder}
[v^\sharp]_{1/2-5\kappa,\fK_r}
&\leq C\,\left(\tilde\fX(\fK,h)\vee\|v\|_{\fK}\right)\,
\left(\frac{1}{r}\vee \tilde\fX(\fK,h)\vee\|v\|_{\fK}\right)^{3/2-5\kappa}
\end{equs}
for all $r\in(0,1]$, where
\begin{equs}
v^\sharp(z) \eqdef& \nabla v(z)+\nabla\outz(z)+3v(z)\nabla\oudz(z),
\\	
\bar U(z,\bar z)\eqdef& v(\bar z)-v(z) + \outz(\bar z)-\outz(z) + 3v(z)(\oudz(\bar z)-\oudz(z))-v^\sharp(z)\cdot X(\bar z-z)
\end{equs}
for $z,\bar z\in \fK$ and
\begin{equ}{}
[\bar U]_{\alpha,\fK}\eqdef\sup_{\substack{z,\bar z\in \fK\\z\neq \bar z}} \frac{|\bar U(z,\bar z)|}{|z-\bar z|^\alpha}\,.
\end{equ}
\end{thm}
\begin{rmk}\label{rmk:spactime_loc_h}
	Recall the decomposition of the heat kernel $K=K^++K^-$ from Lemma~\ref{lem:kernel_decomposition}. We will always apply the above theorem with
	\begin{equ}
		\oudz=K^+\ast \oud\,,
		\quad
		\outz=K^+\ast \out\,,
		\quad		
		h_1=\lL K^-\ast \oud\,,
		\quad
		h_2=\lL K^-\ast \out\,,
		\quad
		h_4=0
	\end{equ}
	for some $\oud$ and $\out$.
	In this situation, we have the estimate
	\begin{equ}
		\tilde\fX(\fK,h)\lesssim \left(\tilde\fX(\fK)\vee\|h_3\|_\fK^{1/3}\right)^{2/(1-2\kappa)},
	\end{equ}
	which follows trivially from~\eqref{eq:tilde_fX_h} and the fact that $\lL K^-=\delta-\lL K^+$ is smooth and supported in the unit ball.
\end{rmk}

\begin{proof}
The general strategy of the proof is to combine a bound for the high regularity norm of the solution provided by a local Schauder estimate stated as Lemma 2.11 in~\cite{MW20} with a coercive bound for the $L^\infty$ norm stated as Lemma~\ref{lem:maximum} below, which is a slight generalisation of Lemma~2.7 in~\cite{MW20}, whose proof is based on the maximum principle and crucially exploits the cubic term in~\eqref{e:deterministic_phi4}.

Since a nearly identical result was established in~\cite{MW20}, we only discuss necessary modifications of the original proof. 
\begin{enumerate}
	\item Our constant $C>0$ does not depend on the space-time cube. The constant of proportionality in the local Schauder estimate states in~\cite{MW20} does not depend on the space-time region. The same is true for the constant of proportionality in Lemma~\ref{lem:maximum}.
	
	\item Our equations~\eqref{e:relation_trees} and~\eqref{e:deterministic_phi4} involve functions $h_1,h_2,h_3,h_4$ and are assumed to hold in the weak sense. 
	By redefining the trees $\oud$ and $\out$ we can reduce to the case $h_1=h_2=0$ and the extra contributions coming from $h_3$ and $h_4$ can be easily bounded using $\|\psi^r\ast h_3\|_{\bar\fK_r}\leq\|h_3\|_{\bar\fK}$ and $\|\psi^r\ast (v h_4)\|_{\bar\fK_r}\leq \|v\|_{\bar\fK}\|h_4\|_{\bar\fK}$ for all $r>0$ and all $\bar\fK\subset\R^{1+3}$. Even though~\cite{MW20} assumes that~\eqref{e:relation_trees} and~\eqref{e:deterministic_phi4} hold pointwise, actually only the equations obtained by convolving both side of~\eqref{e:relation_trees} and~\eqref{e:deterministic_phi4} with a smooth test function $\psi^r$ are used in the proof.
	
	\item $C^{(2)}\in L^\infty(\fK)$, $\ou,\oud,\out\in L^1(\fK)$, $\oudz,\outz,v,h_1,h_2,h_3,h_4\in C^{0,1}(\fK)$ are not assumed to be smooth\footnote{Note that~\eqref{e:phi4_mod} involve space-time white noise mollified only in space.}. The regularity assumption we made is sufficient to ensure that all operations are well-defined. 
	
	\item We work with the massive parabolic differential operator $\lL=\partial_t-\Delta+1$ whereas in~\cite{MW20} the massless operator is used. However, the statement for $\lL=\partial_t-\Delta+1$ follows immediately from the statement for $\lL=\partial_t-\Delta$ since the mass term can be absorbed in $h_1,h_2$ and $h_4$. Hence, in the remaining part of the proof we assume that $\lL=\partial_t-\Delta$. Note that, by the argument we present below, in the massless case the bounds~\eqref{eq:spacetime_localization}-\eqref{eq:spacetime_localization_V_holder} are true even when $\|\oudz\|_{\fK,w}^{1/2}$ and $\|\outz\|_{\fK,w}^{1/3}$ are removed from the maximum in~\eqref{eq:def_X_tilde_st}.
\end{enumerate}

Let us demonstrate the bounds~\eqref{eq:spacetime_localization_holder}-\eqref{eq:spacetime_localization_V_holder}. To this end, we use the fact that $v$ satisfies~\eqref{e:deterministic_phi4} and apply a local Schauder estimate stated as Lemma 2.11 in~\cite{MW20}. Fix a space-time cube $\bar\fK\subset\fK$ and let 
\begin{equ}{}
	\frac{1}{r_0}\eqdef c\,\left(\|v\|_{\bar\fK}\vee \tilde\fX(\bar\fK,h)\right)
\end{equ}
with a small constant $c>0$. By estimates analogous to the estimates (4.2)-(4.18) in Section 4 of~\cite{MW20} we prove that there is a universal constant $C>0$ such that 
\begin{equs}
	\label{eq:aux_reg1}
	r^{1/2-3\kappa}\,[v]_{1/2-3\kappa,\bar\fK_r,r}
	&\leq C\,\left(\tilde\fX(\bar\fK,h)\vee\|v\|_{\bar\fK}\right)\,,
	\\{}
	\label{eq:aux_reg2}
	r^{1-2\kappa}\,[v+\outz]_{1-2\kappa,\bar\fK_r,r}
	&\leq C\,\left(\tilde\fX(\bar\fK,h)\vee\|v\|_{\bar\fK}\right)\,,
	\\{}
	\label{eq:aux_reg3}
	r^{3/2-5\kappa}\,[\bar U]_{3/2-5\kappa,\bar\fK_r}
	&\leq C\,\left(\tilde\fX(\bar\fK,h)\vee\|v\|_{\bar\fK}\right)\,,
	\\{}
	\label{eq:aux_reg4}
	r\,\|v^\sharp\|_{\bar\fK_r}&\leq C\,\left(\tilde\fX(\bar\fK,h)\vee\|v\|_{\bar\fK}\right)\,,
	\\{}
	\label{eq:aux_reg5}
	r^{3/2-5\kappa}\,[v^\sharp]_{1/2-5\kappa,\bar\fK_r,r}
	&\leq C\,\left(\tilde\fX(\bar\fK,h)\vee\|v\|_{\bar\fK}\right)
\end{equs}
for all $r\in(0,r_0)$ provided the constant $c>0$ in the definition of $r_0$ is small enough. We denoted by $[\bigcdot]_{\alpha,\bar\fK,r}$ in~\eqref{eq:aux_reg1}, \eqref{eq:aux_reg2} and \eqref{eq:aux_reg5} the usual $\alpha$-Hölder seminorm defined with the use of the parabolic distance restricted to points in $\bar\fK$ at the distance not bigger than $r$.
Note that in contrast to Section~4 of~\cite{MW20}, in this paragraph\footnote{However, we introduce this assumption in the next paragraph to prove the bound~\eqref{eq:spacetime_localization}.}, we do not assume that $\tilde\fX(\bar\fK,h)\lesssim \|v\|_{\bar\fK}$. Consequently, to prove the bounds~\eqref{eq:aux_reg1}-\eqref{eq:aux_reg5} we have to undo the simplifications in the estimates (4.2)-(4.18) in Section~4 of~\cite{MW20} due to this assumption. This amounts to replacing in all these estimates the $L^\infty$ norm of $v$ over the region of interest by $\|v\|_{\bar\fK}\vee \tilde\fX(\bar\fK,h)$. The bounds~\eqref{eq:aux_reg1}-\eqref{eq:aux_reg5} are analogs of (4.23), (4.22), (4.20), (4.21) and (4.24) in~\cite{MW20}, respectively. Next, we observe that for $r>0$ we have trivial estimates
\begin{equs}{}
	[v]_{1/2-3\kappa,\bar\fK_r} &\leq [v]_{1/2-3\kappa,\bar\fK_r,r} \vee \frac{2\|v\|_{\bar\fK}}{r^{1/2-3\kappa}} \;,
	\\{}
	[v+\outz]_{1-2\kappa,\bar\fK_r} &\leq  [v+\outz]_{1-2\kappa,\bar\fK_r,r} \vee \left( \frac{2\|v\|_{\bar\fK}}{r^{1-2\kappa}} + \frac{[\outz]_{1/2-3\kappa,\bar\fK}}{r^{1/2+\kappa}}\right) \;,
	\\{}
	[v^\sharp]_{1/2-5\kappa,\bar\fK_r} &\leq  [v^\sharp]_{1/2-5\kappa,\bar\fK_r,r} \vee \frac{2\|v^\sharp\|_{\bar\fK_r}}{r^{1/2-5\kappa}} \;.
\end{equs}
Combining these estimates and \eqref{eq:aux_reg1}-\eqref{eq:aux_reg5}, we get that there is a universal constant $C>0$ such that 
\begin{equs}
	\label{eq:aux_reg21}
	(r_0\wedge r)^{1/2-3\kappa}\,[v]_{1/2-3\kappa,\bar\fK_r}
	&\leq C\,\left(\tilde\fX(\bar\fK,h)\vee\|v\|_{\bar\fK}\right)\,,
	\\{}
	\label{eq:aux_reg22}
	(r_0\wedge r)^{1-2\kappa}\,[v+\outz]_{1-2\kappa,\bar\fK_r}
	&\leq C\,\left(\tilde\fX(\bar\fK,h)\vee\|v\|_{\bar\fK}\right)\,,
	\\{}
	\label{eq:aux_reg23}
	(r_0\wedge r)^{3/2-5\kappa}\,[\bar U]_{3/2-5\kappa,\bar\fK_r}
	&\leq C\,\left(\tilde\fX(\bar\fK,h)\vee\|v\|_{\bar\fK}\right)\,,
	\\{}
	\label{eq:aux_reg24}
	(r_0\wedge r)\,\|v^\sharp\|_{\bar\fK_r}&\leq C\,\left(\tilde\fX(\bar\fK,h)\vee\|v\|_{\bar\fK}\right)\,,
	\\{}
	\label{eq:aux_reg25}
	(r_0\wedge r)^{3/2-5\kappa}\,[v^\sharp]_{1/2-5\kappa,\bar\fK_r}
	&\leq C\,\left(\tilde\fX(\bar\fK,h)\vee\|v\|_{\bar\fK}\right)\,.
\end{equs}
for all $r \in (0, 1)$, where in the case $r > r_0$ we used the inclusion $\fK_r \subset \fK_{r_0}$. Choosing $\bar\fK=\fK$ we obtain the bounds~\eqref{eq:spacetime_localization_holder}-\eqref{eq:spacetime_localization_V_holder}. Observe that in this part of the proof we only used the local Schauder estimate.

In order to prove the bound~\eqref{eq:spacetime_localization} we apply the estimate stated in Lemma~\ref{lem:maximum} to the equation
\begin{equ}\label{eq:phi4_eq_st_loc_proof}
 	(\partial_t-\Delta)(\psi^{\hat r}\ast v)+(\psi^{\hat r}\ast v)^3=g\,,
\end{equ}
obtained by convolving both sides of~\eqref{e:deterministic_phi4} with a test function $\psi^{\hat r}$ of characteristic length scale $\hat r>0$ and support contained in $B_-(0,\hat r)$, where
\begin{equ}
	g=(\psi^{\hat r}\ast v)^3+\psi^{\hat r}\ast(-v^3-3v^2\ou-3v\oud-\out-9C^{(2)}(v+\ou)+h_3+h_4 v)\,.
\end{equ}
Note that although \eqref{e:deterministic_phi4} only holds in a weak sense in $\fK$, \eqref{eq:phi4_eq_st_loc_proof} holds in a strong sense in $\fK_{\hat r}$. By Lemma~\ref{lem:maximum}, we obtain
\begin{equs}[eq:bound_st_loc_proof]
	\|\psi^{\hat r}&\ast v\|_{\bar\fK_R} \lesssim \max \bigg\{\frac{1}{R-R'},
	\|(\psi^{\hat r}\ast v)^3-\psi^{\hat r}\ast v^3\|^{1/3}_{\bar\fK_{R'}},
	\|\psi^{\hat r}\ast (v^2\ou)\|^{1/3}_{\bar\fK_{R'}},
	\\
	&\|\psi^{\hat r}\ast(v\oud+3C^{(2)}(v+\ou))\|^{1/3}_{\bar\fK_{R'}},
	\|\psi^{\hat r}\ast\out\|^{1/3}_{\bar\fK_{R'}},
	\|\psi^{\hat r}\ast (h_3+h_4 v)\|^{1/3}_{\bar\fK_{R'}}
	\bigg\} 
\end{equs}
uniformly in $0<R'<R$ and $\hat r>0$ such that $\bar\fK_{R'}\subset\fK_{\hat r}$. To deduce the bound~\eqref{eq:spacetime_localization} it is convenient to proceed as in~\cite{MW20} and assume that $\tilde\fX(\fK,h)\leq c\,\|v\|_{\bar\fK}$ for some domain $\bar\fK=\fK_{\bar r}$ with $\bar r\geq 0$, where $c>0$ is a fixed small constant. If the above bound is false for all $\bar r\geq 0$ then~\eqref{eq:spacetime_localization} holds true and the proof is finished. Hence, it remains to prove that if $\tilde\fX(\fK,h)\leq c\,\|v\|_{\bar\fK}$ for some $\bar\fK=\fK_{\bar r}$, then $\|v\|_{\bar\fK}\leq\,\frac{C}{\bar r}$. To this end, we simplify the right-hand side of the bound~\eqref{eq:bound_st_loc_proof} using~\eqref{eq:aux_reg1}-\eqref{eq:aux_reg5} and subsequently use an iteration argument. The details can be found in Section~4.4-4.6 of~\cite{MW20}. We note that $D$, $T$ and $r$ therein correspond to our $\bar\fK$, $\hat r$ and $\bar r$ and $R_N=1/2$ therein has to be replaced by half the diameter of $\fK$. This completes the proof.
\end{proof}

\begin{lem}[]\label{lem:maximum}
There exists a constant $C>0$ such that the following statement is true.
Let $\fK\subset\R^{1+d}$ be a space-time cube and $u,g\in C(\fK)$ be such that the following equality
$
(\partial_t-\Delta)u+u^3=g
$
holds pointwise in $\fK$. Then
\begin{equ}
 \|u\|_{\fK_r} \leq C\left(\frac{1}{r}\vee \|g\|^{1/3}_\fK\right),
\end{equ}
for all $r\in(0,1]$.
\end{lem}
\begin{proof}
	The statement is a generalisation of Lemma~2.7 in~\cite{MW20}, which is stated only for $\fK=[0,1]\times[-1,1]^d$. To show the result for an arbitrary space-time cube $\fK=[a_0,b_0]\times\ldots\times[a_d,b_d]\subset\R^{1+d}$ it is enough to apply the argument from the proof Lemma~2.7 in~\cite{MW20} with a function
	\begin{equ}
		\eta(t,x)\eqdef\frac{\frac{1}{5}}{\frac{1}{5}\|g\|^{1/3}_\fK+\frac{1}{\sqrt{t-a_0}}+\sum_{i=1}^d \frac{1}{x_i-a_i}+\sum_{i=1}^d \frac{1}{b_i-x_i}}\,,
	\end{equ}
	replacing $\eta$ defined by~(5.17) therein, and noting that, since\footnote{The last formula on the page 2553 of~\cite{MW20} could suggest that the estimate for the expression appearing on the left-hand side is not uniform in $-\infty<a_i<b_i<\infty$, $i\in\{0,\ldots,d\}$. However, there is a typo in this formula.}
	\begin{equs}
		(2\eta& (\partial_t - \Delta) \eta + 4 |\nabla \eta|^2)(t,x) 
	 	\\
	 	&= \left(\frac{5}{\sqrt{t-a_0}^{3}} + 20 \sum_{i=1}^{3} \left( \frac{1}{(a_i-x_i)^3} + \frac{1}{(b_i - x_i)^3} \right)\right)\eta(t,x)^3 \leq 1\,,
	\end{equs}
	 their bound~(5.15) is satisfied in the interior of $\fK$.
\end{proof}

\begin{proof}[Proof of Lemma~\ref{le:psi}]
	Recall the notation introduced at the beginning of Section~\ref{sec:stochastic_objects}. For $x\in\R^3$, let $\fK=[s,t]\times B(x,2)$, 
	\begin{equ}
		\ou=\lambda^{1/2}\ou_{\eps,\ell,s}\,,
		\quad
		\oud=\lambda\tilde\oud_{\eps,\ell,s}\,,
		\quad
		\out=\lambda^{3/2}\tilde\out_{\eps,\ell,s}\,,
		\quad
		C^{(2)}=\lambda^{2}\one_>\Ctwo\,,
	\end{equ}
	\begin{equ}
		h_3=S,
		\quad
		v=\lambda^{1/2}\PSImod=\lambda^{1/2}(\PSI-\lambda\outz_{\eps,\ell,s})
	\end{equ}
	and $\oudz$, $\outz$, $h_1,h_2,h_4$ be as in Remark~\ref{rmk:spactime_loc_h}. Since $\PSImod$ satisfies~\eqref{e:psi_mod} it is easy to see that the assumptions of Theorem~\ref{thm:spacetime_localization} hold true. As a result, by~\eqref{eq:spacetime_localization} there is a universal constant $C>0$ such that
	\begin{equs}
		\lambda^{1/2}\,\|\PSImod\|_{\{t\}\times B(x,1)}&\leq C\,((t-s)^{-1/2}\vee \tilde\fX_{\eps,\ell,s,t,x}^{2/(1-2\kappa)})\,,
		\\{}
		\lambda^{1/2}\,[\PSImod(t)-\lambda\tilde\outz_{\eps,\ell,s}]_{1-2\kappa,\{t\}\times B(x,1)}
		&\leq C\,((t-s)^{-1/2}\vee\tilde\fX_{\eps,\ell,s,t,x}^{2/(1-2\kappa)})^{2-2\kappa}
	\end{equs}
	for all $\lambda\in(0,1]$, $\eps\in(0,1]$, $\ell\in\N_+$, $s\in\R$, $t\in(s,s+1]$, $x,\zzz\in\R^3$, where
	\begin{equs}
		&\tilde\fX_{\eps,\ell,s,t,x}
		\\
		&\eqdef\tilde\fX(\lambda^{1/2}\ou_{\eps,\ell,s},\lambda\tilde\oud_{\eps,\ell,s},\lambda^{3/2}\tilde\out_{\eps,\ell,s},\lambda \tilde\oudz_{\eps,\ell,s},\lambda^{3/2}\tilde\outz_{\eps,\ell,s},\lambda^2\one_>\Ctwo,[s,t]\times B(x,2))
	\end{equs}
	and the trees appearing above were introduced in Definition~\ref{def:enhanced_noise}. Note that there is $c>0$ such that
	\begin{equ}
		\frac{1}{c}\leq \frac{w_\zzz(y)}{w_\zzz(x)}\leq c
	\end{equ}
	for all $x,\zzz\in\R^3,y\in B(x,2)$	and recall $\fX_{\eps,\ell,s,t,\zzz}$ introduced in Definition~\ref{def:size_enhanced_noise}. There exists $C>0$ such that
	\begin{equ}
		\tilde\fX_{\eps,\ell,s,t,x}
		\leq C\,
		w_\zzz(x)^{-1}
		\fX_{\eps,\ell,s,t,\zzz}
	\end{equ}
	for all $\lambda\in(0,1]$, $\eps\in(0,1]$, $\ell\in\N_+$, $s\in\R$, $t\in(s,s+1]$, $x,\zzz\in\R^3$. Using the fact that for $\alpha,\beta>0$ the norm $\|\bigcdot\|_{\cC^\alpha(w^\beta)}$ is equivalent to the norm $\|\bigcdot\|_{L^\infty(w^\beta)}+[\bigcdot]_{\alpha,\R^3,w^\beta}$ we obtain
	\begin{equs}
		\lambda^{1/2}\,\|\PSI(t)\|_{L^\infty(w_{\zzz}^{2/(1-2\kappa)})}
		&\leq C\,(t-s)^{-1/2}\vee C\, \fX_{\eps,\ell,s,t,\zzz}^{2/(1-2\kappa)},
		\\
		\lambda^{1/2}\,\|\PSI(t)\|_{\cC^{1/2 + 4\kappa}(w_\zzz^{(3+4\kappa)/(1-2\kappa)})}
		&\leq C\,(t-s)^{-3/4-2\kappa}\vee C\, \fX_{\eps,\ell,s,t,\zzz}^{(3+8\kappa)/(1-2\kappa)}
	\end{equs}
	with a universal constant $C>0$, which implies the bounds stated in the lemma.
\end{proof}

\section{Stochastic estimates}\label{sec:stochastic_estimates}
\begin{lem}\label{lem:convergence_stationary_model}
	Let $\weight_\Pi=\bracket{\bigcdot}^{-a}\in C(\R^3)$ and $\xi_{\eps,\ell}$ be constructed from the space-time white noise $\xi$ as in Definition~\ref{def:smoothed_noise}. Suppose that $(\Pi_{\eps,\ell},\Gamma_{\eps,\ell})\in\mM$ is the canonical model constructed in terms of $\xi_{\eps,\ell}$ and $(\hat\Pi_{\eps,\ell},\hat\Gamma_{\eps,\ell})\in\mM$ is the corresponding model obtained by application of the renormalisation map with parameters $\Cone$ and $\Ctwo$ introduced in Definition~\ref{def:enhanced_noise}. Then there exists a random model $(\hat\Pi,\hat\Gamma)\in\mM(\weight_\Pi)$ independent of the choice of the mollifier used in the definition of $\xi_{\eps,\ell}$ such that $\|(\hat\Pi,\hat\Gamma)\|_{T, \weight_\Pi}\in L^p(\Omega)$ and
	\begin{equ}
		\lim_{\ell\to\infty}\lim_{\eps\searrow0}\|(\hat\Pi,\hat\Gamma)-(\hat\Pi_{\eps,\ell},\hat\Gamma_{\eps,\ell})\|_{T, \weight_\Pi}=0
	\end{equ}
	almost surely and in $L^p(\Omega)$ for every $T>0$, $a>0$ and $p\geq1$. Furthermore, for every element of the Euclidean group $\varrho$, the transformed model $(\varrho\cdot\hat\Pi,\varrho\cdot \hat\Gamma)$ coincides with the model $(\hat\Pi,\hat\Gamma)$ constructed using the transformed white noise $\varrho\cdot\xi$; see the notation introduced in Definition~\ref{def:Euclidane_transform}.
\end{lem}

\begin{proof}
	Let $(\bar\Pi_{\eps,\ell},\bar\Gamma_{\eps,\ell})=M(\bar C^{(1)}_{\eps},\bar C^{(2)}_{\eps})(\Pi_{\eps,\ell},\Gamma_{\eps,\ell})\in\mM$ be the so-called BPHZ model, which is obtained by the application of the renormalisation map $M(\bar C^{(1)}_{\eps},\bar C^{(2)}_{\eps})$ with parameters defined by~\eqref{eq:C_hat} to the model $(\Pi_{\eps,\ell},\Gamma_{\eps,\ell})$. By\footnote{Actually, \cite[Section~10]{Hai14} uses a truncated massless heat kernel and a noise mollified in both space and time in the construction of the canonical model. However, the arguments presented therein apply also to a truncated massive heat kernel and a spatially mollified noise. Even though our $\Pi^{t,x}(\tau)$ is not smooth in time, it is a continuous function over space-time, which is sufficient to define the canonical model.}~\cite[Section~10]{Hai14} for every $\ell\in\N_+$ there exists $(\bar\Pi_\ell,\bar\Gamma_\ell)\in\mM$ such that $\|(\bar\Pi_\ell,\bar\Gamma_\ell)\|_{\fK}\in L^p(\Omega)$ and
	\begin{equ}
		\lim_{\eps\searrow0}\|(\bar\Pi_\ell,\bar\Gamma_\ell)-(\bar\Pi_{\eps,\ell},\bar\Gamma_{\eps,\ell})\|_{\fK}=0
	\end{equ}
	almost surely and in $L^p(\Omega)$ for every compact set $\fK\subset\R^{1+3}$ and $p\geq 1$. By assumption,
	\begin{equ}
		(\hat\Pi_{\eps,\ell},\hat\Gamma_{\eps,\ell})
		=
		M(\Cone,\Ctwo)(\Pi_{\eps,\ell},\Gamma_{\eps,\ell})
		=
		M(\Cone-\bar C^{(1)}_{\eps},\Ctwo-\bar C^{(2)}_{\eps})(\bar\Pi_{\eps,\ell},\bar\Gamma_{\eps,\ell})\,.	
	\end{equ}
	Since by~\eqref{eq:C_diff_convergence} we have
	\begin{equ}
		\lim_{\ell\to\infty}\lim_{\eps\searrow0} M(\Cone-\bar C^{(1)}_{\eps},\Ctwo-\bar C^{(2)}_{\eps})
		=M(\bar C^{(1)},\bar C^{(2)})\,,
	\end{equ}
	it follows from~\cite[Theorem~6.16]{BHZ19} that for every $\ell\in\N_+$ there exists $(\hat\Pi_\ell,\hat\Gamma_\ell)\in\mM$ such that $\|(\hat\Pi_\ell,\hat\Gamma_\ell)\|_{\fK}\in L^p(\Omega)$ and
	\begin{equ}
		\lim_{\eps\searrow0}\|(\hat\Pi_\ell,\hat\Gamma_\ell)-(\hat\Pi_{\eps,\ell},\hat\Gamma_{\eps,\ell})\|_{\fK}=0
	\end{equ}
	almost surely and in $L^p(\Omega)$ for every compact set $\fK\subset\R^{1+3}$ and $p\geq 1$.
	
	Exploiting the fact that the kernel $K^+$ used in the construction of the model $(\hat\Pi_{\eps,\ell},\hat\Gamma_{\eps,\ell})$ is supported in the unit ball centred at the origin,
	one shows that, given any compact set $\fK\subset\R^{1+3}$, $(t,x),(s,y)\in\fK$ and $\psi\in\bB$,
	the random variables
	$\hat\Pi^{t,x}_{\eps,\ell}(\psi^r_{t,x})$ and $\hat\Gamma^{t,x;s,y}_{\eps,\ell}$ are measurable with respect to the $\sigma$-algebra generated by the noise $\xi_{\eps,\ell}$ restricted to $N$-fattening $\hat\fK$ of $\fK$, where $N\in\N_+$ is some fixed constant that depends only on the level of truncation of the regularity structure. In particular, since $\xi_{\eps,\ell}=\xi$ on $\R \times [-\frac{\ell}{2},\frac{\ell}{2})^3$, the functions $\ell\mapsto\hat\Pi^{t,x}_{\eps,\ell}(\psi^r_{t,x})$ and $\ell\mapsto\hat\Gamma^{t,x;s,y}_{\eps,\ell}$ are constant  for all $\ell\in\N_+$ such that $\hat\fK \subset \R \times [-\frac{\ell}{2},\frac{\ell}{2})^3$. The same is true for $(\hat\Pi_\ell,\hat\Gamma_\ell)$. Hence, on compact subsets of $\R^{1+3}$ the infinite volume limit is trivial and the model $(\hat\Pi_\ell,\hat\Gamma_\ell)$ on $\T_\ell^3$ automatically yields a candidate model $(\hat\Pi,\hat\Gamma)\in\mM$. Using stationarity of the models in space, the assumed form of the weight and Remark~\ref{rmk:norm_cubes}, we show that $\|(\hat\Pi,\hat\Gamma)\|_{T, \weight_\Pi}\in L^p(\Omega)$ and
	\begin{equ}
		\lim_{\ell\to\infty}\|(\hat\Pi,\hat\Gamma)-(\hat\Pi_\ell,\hat\Gamma_\ell)\|_{T, \weight_\Pi}=0
	\end{equ}
	almost surely and in $L^p(\Omega)$ for every $T>0$, $a>0$ and $p\geq1$.
	
	To prove Euclidean invariance one uses the representation of $(\hat\Pi,\hat\Gamma)$ as an element of an inhomogeneous Wiener chaos of finite order and the fact that $K^+(t-s,\varrho\cdot x-\varrho\cdot y)=K^+(t-s,x-y)$, which follows from Lemma~\ref{lem:kernel_decomposition}.
\end{proof}

\begin{defn}
	Let $\eta,\gamma\in\R$ and $\fF\subset\tT$. Recall that $\dD^\gamma(\fF)$ and $\dD^{\gamma,\eta}(\fF)$ are the locally convex spaces introduced in~\cite[Definition~3.1]{Hai14} and~\cite[Definition~6.2]{Hai14} equipped with families of seminorms $\VERT{\bigcdot}_{\gamma;\fK}$ and $\VERT{\bigcdot}_{\gamma,\eta;\fK}$ indexed by compact sets $\fK\subset\R^{1+3}$. We define $\bar\dD^{\gamma,\eta}(\fF)$ to be the space consisting of $f\in\dD^{\gamma,\eta}(\fF)$ such that $\qQ_{<\eta}f\in\dD^\eta$ and denote by $\hat\dD^{\gamma,\eta}(\fF)$ its subspace consisting of $f\in\bar\dD^{\gamma,\eta}(\fF)$ such that $f(t,x)=0$ for $t\leq0$. We omit $\fF$ if it is clear from the context.
\end{defn}
\begin{rmk}\label{rmk:bar_dD}
	One shows that $\bar\dD^{\gamma,\eta}$ and $\hat\dD^{\gamma,\eta}$ are closed subsets of $\dD^{\gamma,\eta}$ and $\VERT{\qQ_{<\eta}f}_{\eta;\fK}\lesssim(1+\|(\Pi,\Gamma)\|_\fK)\,\VERT{f}_{\gamma,\eta;\fK}$ uniformly over $f\in\bar\dD^{\gamma,\eta}$, $(\Pi,\Gamma)\in\mM$ and compact $\fK\subset\R^{1+3}$.
\end{rmk}

\begin{thm}[Theorem~A.6 in~\cite{CCHS22}]\label{thm:reconstruction_singular}
	Let $\gamma > 0$ and $\eta \in (-2, \gamma]$. The reconstruction operator $\rR$ satisfies the bound
	\begin{equ}
		(\rR f-\Pi^z\qQ_{<\eta}f(z))(\psi^z_r) \lesssim r^\eta\,(1+\|(\Pi,\Gamma)\|_{B(z,1)})^2\,\VERT{f}_{\gamma,\eta;B(z,1)}
	\end{equ}
	uniformly over $r\in(0,1]$, $z\in\R^{1+3}$, $\psi\in\mathcal{B}$, $(\Pi,\Gamma)\in\mM$ and $f\in\bar\dD^{\gamma,\eta}$. 
\end{thm}

\begin{lem}\label{lem:convergence_model_phi_deterministic}
	Let $(\Pi,\Gamma)\in\mM(\weight_\Pi)$ be the canonical model constructed in terms of a~regular in space noise $\xi\in C^{-\f12-\kappa}(\R,C_{\mathrm{b}}(\R^3))$ and $(\hat\Pi,\hat\Gamma)\in\mM(\weight_\Pi)$ be the corresponding model obtained by application of the renormalisation map with parameters $C^{(1)}$ and $C^{(2)}$. For $S\in L^\infty_{\mathrm{loc}}(\R^{1+3})$ define $\ou_S,\oud_S,\out_S\in L^\infty_{\mathrm{loc}}(\R^{1+3})$ and $\oudz_S,\outz_S\in C^{0,1}(\R^{1+3})$ by
		\begin{equs}
		\ou_S &\eqdef \one_>\rR\bou + S\,,
		\\
		\oud_S &\eqdef \one_>\rR\boud  + 2\rR\bou S + S^2\,,
		\\
		\out_S &\eqdef \one_>\rR\bout  + 3\rR\boud S + 3\rR \bou S^2 +  S^3\,,
		\\
		\oudz_S &\eqdef K^+\ast \oud_S\,,
		\\
		\outz_S &\eqdef K^+\ast \out_S\,,	
	\end{equs}
	where $\rR$ is the reconstruction operator associated to the model $(\hat\Pi,\hat\Gamma)$. Recall Definition~\ref{def:enhanced_noise_spacetime}. Let $N=4$ and $\hat w=\bracket{\bigcdot}^{-a}\in C(\R^3)$ with $a\geq 0$. We have
	\begin{equ}
		\tilde\fX(\ou_S,\oud_S,\out_S,\oudz_S,\outz_S,\one_>C^{(2)},\bar\oO_T,\tilde w) 
		\lesssim (1+\|(\hat\Pi,\hat\Gamma)\|_{T, \weight_\Pi})^N\, (1+\VERT{S}_{\gamma,\eta;T,\hat w})
	\end{equ} 
	uniformly over all $C^{(1)},C^{(2)}\in\R$, $\xi\in \cC^{-\f12-\kappa}(\R,C_{\mathrm{b}}(\R^3))$ and $S\in L^\infty(\R^{1+3})$ that admit a~lift to a polynomial sector in $\hat\dD^{\gamma+\eta,\eta}_{T,\hat w}$ with $\gamma=\f32+4\kappa$ and $\eta=-\f12-\kappa$, where $\bar\oO_T=[-1,T]\times\R^3$ and $\tilde w=\hat w \weight_\Pi^N$.
\end{lem}
\begin{proof}
It is enough to show that we have
\begin{equ}
	\tilde\fX(\fK)=	\tilde\fX(\ou_S,\oud_S,\out_S,\oudz_S,\outz_S,\one_>C^{(2)},\fK) 
	\lesssim (1+\|(\hat\Pi,\hat\Gamma)\|_{\hat\fK})^N\, (1+\VERT{S}_{\gamma,\eta;\hat\fK})
\end{equ} 
uniformly over compact $\fK\subset\R^{1+3}$ and $C^{(1)},C^{(2)}$, $\xi$, $S$ as in the statement of the lemma. Here and in what follows, $\bar\fK$ and $\hat\fK$ denote the $1$- and $2$-fattenings of $\fK$, respectively.

\step{Expression~\eqref{eq:X_tilde_proof} for $\tilde\fX(\fK)$.} For $(t,x)\in\R^{1+3}$, we introduce the following singular modelled distributions
\begin{equs}
	F^{t,x}(\bou)&=F(\bou)& &\eqdef \one_>\bou + S\,,
	\\\
	F^{t,x}(\boud)&=F(\boud)& &\eqdef \one_>\boud + 2S \bou + S^2\,,
	\\
	F^{t,x}(\bout)&=F(\bout)& &\eqdef \one_> \bout + 3S \boud + 3S^2 \bou + S^3\,.
\end{equs}
Moreover, we define
\begin{equs}
	F^{t,x}(\boud\bXXX) &\eqdef F(\boud)\,\bXXX + F(\boud)\,X(\bigcdot-x)\,,
	\\
	F^{t,x}(\boudd) &\eqdef F(\boud)\, (\kK^+ F(\boud) - \oudz_S(t,x)\oneb)\,,
	\\
	F^{t,x}(\boutu) &\eqdef F(\bou)\,(\kK^+ F(\bout) - \outz_S(t,x)\oneb) \,,
	\\
	F^{t,x}(\boutd) &\eqdef F(\boud)\,(\kK^+ F(\bout) - \outz_S(t,x)\oneb) \,.
\end{equs}
Using Definition~\ref{def:integration_K} of $\kK^+$, the identity $\rR\kK^+ f = K^+ \ast \rR f$ and
\begin{equ}
	\rR\boudz=
	\rR\boutz=
	\rR\boud\bXXX=
	\rR\boutu=0\,,\quad
	\rR\boudd=-C^{(2)}\,,\qquad
	\rR\boutd=-3C^{(2)}\rR\bou \,,
\end{equ}
it is straightforward to check that
\begin{equ}
	\rR F^{t,x}(\bou)=\ou_S\,,
	\qquad
	\rR F^{t,x}(\boud)=\oud_S\,,
	\qquad
	\rR F^{t,x}(\bout)=\out_S\,,
\end{equ}
and
\begin{equs}	
	\rR F^{t,x}(\boud\bXXX) &= X(\bigcdot-(t,x)) \oud_S\,,
	\\
	\rR F^{t,x}(\boudd) &= \oud_S (\oudz_S-\oudz_S(t,x))-\one_> C^{(2)}\,,
	\\
	\rR F^{t,x}(\boutu) &= \ou_S (\outz_S-\outz_S(t,x))\,,
	\\
	\rR F^{t,x}(\boutd) &= \oud_S (\outz_S-\outz_S(t,x))-3 C^{(2)}\ou_S\,.	
\end{equs}
Note that for all $\tau$ the support of $\rR F^{t,x}(\tau)$ is contained in $\R_\geq\times\R^3$. One checks that
\begin{equ}
	\qQ_{<|\tau|} F^{t,x}(\tau)(t, x)=0
\end{equ}
for $\tau\in\tilde\tT^\circ$. Moreover, by Definition~\ref{def:integration_K} we have
\begin{equ}
	\qQ_{<|\boudz|}\kK^+ F(\boud) =\oudz_S\oneb \,,
	\qquad
	\qQ_{<|\boutz|}\kK^+ F(\bout) =\outz_S\oneb \,.
\end{equ}
Hence,
\begin{equ}
	\|\oudz_S\|_{\fK}
	\vee
	[\oudz_S]_{|\boudz|,\fK}
	\simeq
	\VERT{\qQ_{<|\boudz|}\kK^+ F(\boud)}_{|\boudz|;\fK}\,,
	\quad
	\|\outz_S\|_{\fK}
	\vee
	[\outz_S]_{|\boutz|,\fK}
	\simeq
	\VERT{\qQ_{<|\boutz|}\kK^+ F(\bout)}_{|\boutz|;\fK}\,,
\end{equ}
where $\VERT{\bigcdot}_{\gamma;\fK}$ is the norm in the space of modelled distributions introduced in~\cite[Definition~3.1]{Hai14} and $[\bigcdot]_{\alpha, \fK}$ is the H\"older semi-norm. As a result, by Definition~\ref{def:enhanced_noise_spacetime} we obtain 
\begin{equs}[eq:X_tilde_proof]
	\tilde\fX(\fK) 
	\simeq&
	\sup_{\tau\in\{\boudz,\boutz\}}\VERT{\qQ_{<|\tau|}F(\tau)}_{|\tau|;\fK}^{1/n(\tau)}
	\\
	&\vee
	\sup_{\tau\in\tT_<^\circ}
	\left(\sup_{r\in(0,1]}\sup_{(t,x)\in\fK}r^{-|\tau|}\,
	|(\rR F^{t,x}(\tau))(\psi^r_{t,x})|\right)^{1/n(\tau)}
	\;,
\end{equs}
where
$
	\tT_<^\circ\eqdef\{\bout,\boud,\boutd,\bou,\boudd,\boutu,\boud\bXXX\}
$
and $n(\tau)$ denotes the number of leaves of $\tau$. To bound the second line in~\eqref{eq:X_tilde_proof} we will prove that for all $\tau\in\tT_<^\circ$ there is $N(\tau)\leq 4n(\tau)$ such that we have\footnote{Actually, it would suffice to prove this for $\psi\in\bB_-\subset\bB$ fixed at the beginning of Appendix~\ref{sec:spacetime_localization}.}
\begin{equs}[eq:stochastic_proof_bound]
	r^{-|\tau|}
	|(\rR F^{t,x}(\tau))(\psi^r_{t,x})|
	\lesssim (1+\|(\hat\Pi,\hat\Gamma)\|_{\hat\fK})^N(1+\VERT{S}_{\gamma,\eta;\hat\fK})^{n(\tau)}
\end{equs}
uniformly over compact $\fK\subset\R^{1+3}$, $r\in(0,1]$, $\psi\in\bB$ and $(t,x)\in\fK$.

\step{Bounds for $\tau \in \{ \bou, \boud, \bout, \boud\bXXX \}$.} By the estimate for the product of singular modelled distributions~\cite[Lemma~A.8]{CCHS22}, we have $F(\tau)\in \hat\dD^{\gamma+|\tau|,|\tau|}$ and
\begin{equ}\label{eq:estimates_first_level}
	\VERT{F(\tau)}_{\gamma+|\tau|,|\tau|;\fK} \lesssim 
	(1+\|(\hat\Pi,\hat\Gamma)\|_{\fK})^{2(n(\tau)-1)}
	\,(1+\VERT{S}_{\gamma,\eta;\fK})^{n(\tau)}
\end{equ}
for $\tau\in\{\bou,\boud,\bout\}$ uniformly over compact $\fK\subset\R^{1+3}$. Hence, for $\tau\in\{\bout,\boud,\bou\}$ the bound~\eqref{eq:stochastic_proof_bound} with $N(\tau)=2n(\tau)$ is a consequence of Theorem~\ref{thm:reconstruction_singular}. 
The bound for the contribution coming from $\tau=\boud\bXXX$ follows trivially from the bound for the term $\tau=\boud$.

\step{Bounds for $\tau\in\{\boudz,\boutz\}$.} By the estimate for $\kK^+$ from~\cite[Theorem~A.9]{CCHS22}, Remark~\ref{rmk:bar_dD} and~\eqref{eq:estimates_first_level}, we have $\kK^+ F(\tau)\in \hat\dD^{\gamma+|\tau|+2,|\tau|+2}$ and
\begin{equ}
	\VERT{\qQ_{<|\tau|}\kK^+ F(\tau)}_{|\tau|;\fK} \lesssim 
	(1+\|(\hat\Pi,\hat\Gamma)\|_{\bar\fK})^7
	\,(1+\VERT{S}_{\gamma,\eta;\bar\fK})^{n(\tau)}
\end{equ}
for $\tau\in\{\boud,\bout\}$ uniformly over compact $\fK\subset\R^{1+3}$. This implies the desired bound for the first line of~\eqref{eq:X_tilde_proof}.

\step{Bounds for $\tau \in \{\boutd,\boudd,\boutu \}$.}
By the estimate for the product of singular modelled distributions~\cite[Lemma~A.8]{CCHS22}, the estimate for $\kK^+$ from~\cite[Theorem~A.9]{CCHS22} and~\eqref{eq:estimates_first_level}, we obtain that $F(\boud)\,\kK^+ F(\bout)\in\hat\dD^{\gamma+|\boud|,|\boutd|}$ and
\begin{equs}{}
	&\VERT{F(\boud)\,\kK^+ F(\bout)}_{\gamma+|\boud|,|\boutd|;\bar\fK}
	\\
	&\qquad\lesssim
	(1+\|(\hat\Pi,\hat\Gamma)\|_{\bar\fK})^2\,
	\VERT{F(\boud)}_{\gamma+|\boud|,|\boud|;\bar\fK}\,
	\VERT{\kK^+ F(\bout)}_{\gamma+|\boutz|,|\boutz|;\bar\fK}
	\\
	&\qquad\lesssim
	(1+\|(\hat\Pi,\hat\Gamma)\|_{\hat\fK})^4\,
	\VERT{F(\boud)}_{\gamma+|\boud|,|\boud|;\bar\fK}\,
	\VERT{F(\bout)}_{\gamma+|\bout|,|\bout|;\hat\fK}\,
	\\
	&\qquad\lesssim
	(1+\|(\hat\Pi,\hat\Gamma)\|_{\hat\fK})^{12}
	\,(1+\VERT{S}_{\gamma,\eta;\hat\fK})^{n(\boutd)}
\end{equs}
uniformly over compact $\fK\subset\R^{1+3}$. Let $\chi\in C^\infty(\R)$ be such that $\chi=1$ on $[-3,3]$ and $\chi=0$ on $\R\setminus[-4,4]$. Let $\vartheta\in C^\infty(\R)$ be such that $\vartheta=1$ on $[\f12,\f32]$ and $\vartheta=0$ on $\R\setminus[\f14,\f74]$. We can view $(\chi(\bigcdot/r^2))_{r\in(0,1)}$ and $(\vartheta(\bigcdot/t))_{t\in(0,1)}$ as uniformly bounded families of elements of $\bar\dD^{\gamma,0}$. We have
\begin{equ}
	(\rR F^{t,x}(\boutd))(\psi_r^{t,x})
	=
	1_{t\leq 2r^2}\,(\rR \chi(\bigcdot/r^2)F^{t,x}(\boutd))(\psi_r^{t,x})
	+ 
	1_{2r^2<t}\,(\rR\vartheta(\bigcdot/t)F^{t,x}(\boutd))(\psi_r^{t,x}) \,,
\end{equ}
where we used the fact that $(\rR f)(\psi)=(\rR g)(\psi)$ if $f=g$ on $\supp\psi$. Consequently, by Theorem~\ref{thm:reconstruction_singular} we obtain
\begin{equs}[eq:stochastic_proof_boudd]
	&|(\rR F^{t,x}(\boutd))(\psi_r^{t,x})| 
	\lesssim (1+\|(\hat\Pi,\hat\Gamma)\|_{\bar\fK})^2
	\\
	&\times\Big( 
	1_{t\leq 2r^2} r^{|\boud|}  
	\VERT{\chi(\bigcdot/r^2)F^{t,x}(\boutd)}_{\gamma+|\boud|,|\boud|;\bar\fK}
	+ 
	r^{|\boutd|}
	\VERT{\vartheta(\bigcdot/t)F^{t,x}(\boutd)}_{\gamma+|\boud|,|\boutd|;\bar\fK} 
	\Big)
\end{equs}
uniformly over $(t,x)\in\fK$ and compact $\fK\subset\R^{1+3}$. Note that for a fixed $\delta\geq0$ we have $\VERT{F}_{\gamma+|\boud|,\eta;\bar\fK}\lesssim r^{\delta}\, \VERT{F}_{\gamma+|\boud|,\eta+\delta;\bar\fK}$ uniformly over $r\in(0,1]$ and $F\in\hat\dD^{\gamma+|\boud|,\eta}(\tT_{\geq\eta})$ supported in $[0,r^2]\times\R^3$. Hence,
\begin{equs}
	&\VERT{\chi(\bigcdot/r^2) F(\boud)\,\kK^+ F(\bout)}_{\gamma+|\boud|,|\boud|;\bar\fK}
	\\
	&\qquad\lesssim
	r^{|\boutd|-|\boud|}\,\VERT{\chi(\bigcdot/r^2) F(\boud)\,\kK^+ F(\bout)}_{\gamma+|\boud|,|\boutd|;\bar\fK}
	\\
	&\qquad\lesssim
	r^{|\boutd|-|\boud|}\,(1+\|(\hat\Pi,\hat\Gamma)\|_{\bar\fK})^2
	\,\VERT{F(\boud)\,\kK^+ F(\bout)}_{\gamma+|\boud|,|\boutd|;\bar\fK}\,,
\end{equs}
where in the second line we used the estimate for the product of singular modelled distributions~\cite[Lemma~A.7]{CCHS22}. In consequence, for $t\leq 2r^2$ we have
\begin{equs}  
	r^{|\boud|}\,  &\VERT{\chi(\bigcdot/r^2)F^{t,x}(\boutd)}_{\gamma+|\boud|,|\boud|; \bar\fK}
	\\
	&=
	r^{|\boud|}\,\VERT{\chi(\bigcdot/r^2) F(\boud)\,(\kK^+ F(\bout) - \outz_S(t,x)\oneb)}_{\gamma+|\boud|,|\boud|; \bar\fK}
	\\
	&\lesssim
	r^{|\boutd|}(1+\|(\hat\Pi,\hat\Gamma)\|_{\bar\fK})^2\,\VERT{F(\boud)\,\kK^+ F(\bout)}_{\gamma+|\boud|,|\boutd|; \bar\fK}
	\\
	&\qquad+
	r^{|\boutd|}\VERT{F(\boud)}_{\gamma+|\boud|,|\boud|;\bar\fK}~ t^{-|\boutz|/2}|\outz_S(t,x)|\,.
\end{equs}
For the second term, observe that for a fixed $\delta\geq0$ we have $\VERT{F}_{\gamma+|\boud|,\eta;\bar\fK}\lesssim r^{-\delta}\, \VERT{F}_{\gamma+|\boud|,\eta-\delta;\bar\fK}$ uniformly over $r\in(0,1]$ and $F\in\hat\dD^{\gamma+|\boud|,\eta}(\tT_{\geq\eta})$ supported in $[r^2,\infty)\times\R^3$. Hence,
\begin{equs}
	\VERT{\vartheta(\bigcdot/t)F(\boud)}_{\gamma+|\boud|,|\boutd|;\bar\fK}
	&\lesssim
	t^{-|\boutz|/2}\VERT{\vartheta(\bigcdot/t)F(\boud)}_{\gamma+|\boud|,|\boud|;\bar\fK}
	\\
	&\lesssim
	t^{-|\boutz|/2}(1+\|(\hat\Pi,\hat\Gamma)\|_{\bar\fK})^2\,
	\VERT{F(\boud)}_{\gamma+|\boud|,|\boud|;\bar\fK}\,,	
\end{equs}
where in the second line we used the estimate for the product of singular modelled distributions~\cite[Lemma~A.7]{CCHS22}. Thus, we have
\begin{equs}
	r^{|\boutd|}\,
	\VERT{\vartheta(\bigcdot/t)F^{t,x}(\boutd)}_{\gamma+|\boud|,|\boutd|; \bar\fK}
	&=
	r^{|\boutd|}\,\VERT{\vartheta(\bigcdot/t)F(\boud)\,(\kK^+ F(\bout) - \outz_S(t,x)\oneb)}_{\gamma+|\boud|,|\boutd|; \bar\fK}
	\\
	&\lesssim
	r^{|\boutd|}(1+\|(\hat\Pi,\hat\Gamma)\|_{\bar\fK})^2\,\VERT{F(\boud)\,\kK^+ F(\bout)}_{\gamma+|\boud|,|\boutd|; \bar\fK}
	\\
	&\quad +
	r^{|\boutd|}\VERT{F(\boud)}_{\gamma+|\boud|,|\boud|;\bar\fK} ~t^{-|\boutz|/2}|\outz_S(t,x)|\,.
\end{equs}
Using the fact that $\outz_S(t,x)=0$ for $t\leq 0$ we obtain
\begin{equ}
	t^{-|\boutz|/2}|\outz_S(t,x)|\leq [\outz_S]_{|\boutz|,\fK}\,.
\end{equ}
Plugging the above estimates into~\eqref{eq:stochastic_proof_boudd} one concludes the bound~\eqref{eq:stochastic_proof_bound} for $\tau=\boutd$ with $N(\boutd)=16$. To prove the bound for $\tau=\boudd$ we use the same argument with $\boutd$, $\bout$, $\boutz$, $\outz_S$ replaced by $\boudd$, $\boud$, $\boudz$, $\oudz_S$. Finally, in the case $\tau=\boutu$ we replace $\boutd$, $\boud$, by $\boutu$, $\bou$.
\end{proof}

\begin{lem}\label{lem:stochastic_spacetime}
	Recall Definitions~\ref{def:smoothed_noise},~\ref{def:enhanced_noise} and~\ref{def:size_enhanced_noise}. For every $p>0$ there exists $C>0$ such that
	\begin{equ}
		\EE (\tilde\fX_{\eps,\ell,s,t})^p \leq C
	\end{equ}
	for all $\eps\in(0,1]$, $\ell\in\N_+$, $s\in\R$, $t>s$.
\end{lem}
\begin{proof}
	Note that we have
	\begin{equ}
		\tilde\fX_{\eps,\ell,s,t,z}=\tilde\fX(\ou_{\eps,\ell,s},\tilde\oud_{\eps,\ell,s},\tilde\out_{\eps,\ell,s},\tilde\oudz_{\eps,\ell,s},\tilde\outz_{\eps,\ell,s},\one_{(s,\infty)}\Ctwo,[s,t]\times\R^3,w_z)\,.
	\end{equ}
	By translation invariance, we may assume without loss of generality that $s=0$ and $z=0$. We apply Lemma~\ref{lem:convergence_model_phi_deterministic} with the canonical model $(\Pi,\Gamma)$ constructed with the use of $\xi_{\eps,\ell}$ and $S=K\ast \one_> \lL\ou^-_{\eps,\ell}$. 
	The claim then follows by applying Lemma~\ref{lem:S_phi_estimate} with $\phi=0$, $h=\lL\ou^-_{\eps,\ell}$, $w=\bracket{\bigcdot}^{-a}$, together with Lemmas~\ref{lem:convergence_stationary_model} and~\ref{lem:stochastic_paracontrolled}, and choosing $a>0$ small enough.
\end{proof}

\begin{lem}\label{le:renormalisation_diff}
	Recall that 
	\begin{equs}
		\Cone &\eqdef \EE |\ou_{\eps,\ell}(t, x)|^2\,,
		\qquad&
		\Ctwo&\eqdef\EE\oudz_{\eps,\ell}(t,x)\oud_{\eps,\ell}(t,x)\,,
		\\
		\Cones(t) &\eqdef \EE |\ou_{\eps,\ell,s}(t, x)|^2\,,
		\qquad&
		\Ctwos(t)&\eqdef\EE|\nabla\oudz_{\eps,\ell,s}(t, x)|^2\,.
	\end{equs}
	Let
	\begin{equ}\label{eq:C_hat}
		\bar C^{(1)}_{\eps}\eqdef \int_{\R^4} K^+_\eps(t,x)\,\md t\md x\,,
		\qquad
		\bar C^{(2)}_{\eps}\eqdef 2\int_{\R^4} ((K^+_\eps\ast K^+_\eps)(t,x))^2 K^+(t,x)\,\md t\md x\,,		
	\end{equ}
	where $K^+_\eps\eqdef M_\eps\star K^+$, $M_\eps$ is the mollifier used in the definition of $\xi_{\eps,\ell}$ and $\star$ denotes the convolution over $\R^3$. There exists $C>0$ such that
	\begin{equ}
		\label{eq:counterterm-bound}
		|\Cones(t) - \Cone| \leq C\,(t-s)^{-1/2}\;,
		\qquad
		|\Ctwos(t) - \Ctwo | \leq C\, (t - s)^{-\kappa}\;,
	\end{equ}
	for all $\eps\in(0,1], \ell\in\N_+,s\in\R,t\in(s,s+1]$. Moreover, there exist $\bar C^{(1)},\bar C^{(2)}\in\R$ such that
	\begin{equ}\label{eq:C_diff_convergence}
		\lim_{\ell\to\infty}\lim_{\eps\searrow0}\Cone-\bar C^{(1)}_{\eps}=\bar C^{(1)}\,,
		\qquad
		\lim_{\ell\to\infty}\lim_{\eps\searrow0}\Ctwo-\bar C^{(2)}_{\eps}=\bar C^{(1)}\,.
	\end{equ}
\end{lem}
\begin{proof}
	By translational invariance without loss of generality we can restrict attention to the case $s=0$. A direct computation yields
	\begin{equs}
		0\leq\Cone-C^{(1)}_{\eps,\ell,0}(t)
		&\lesssim
		\frac{1}{\ell^3}\sum_{k\in (2\pi\Z/\ell)^3} \frac{e^{-2t\langle k\rangle^2}}{2\langle k\rangle^2} \\
		&
		\lesssim
		\frac{1}{\ell^3}+ t^{-1/2} \biggl(\frac{t^{3/2}}{\ell^3}\sum_{k\in (2\pi t^{1/2}\Z/\ell)^3\setminus\{0\}}\frac{e^{-2|k|^2}}{2|k|^2}\biggr)\;,
	\end{equs}
	which implies the first of the bounds~\eqref{eq:counterterm-bound}. Next, by stationarity and integration by parts, we observe that
	\begin{equs}
		\Ctwo-C^{(2)}_{\eps,\ell,0} 
		& = \frac{1}{|\T_\ell^3|}\int_{\T_\ell^3} \EE \oudz_{\eps, \ell}(t, x) \big(\d_t - \Laplace + 1 \big) \oudz_{\eps, \ell}(t, x) \md x - \EE |\nabla\oudz_{\eps,\ell,0}(t,0)|^2
		\\
		& = 
		\EE(\nabla\oudz_{\eps,\ell}(t,0)-\nabla\oudz_{\eps,\ell,0}(t,0))
		(\nabla\oudz_{\eps,\ell}(t,0)+\nabla\oudz_{\eps,\ell,0}(t,0)) 
		+ 
		\EE \oudz_{\eps,\ell}(t,0)^2\,.
	\end{equs}
	Let $S_{\eps,\ell}(t)=K(t)\star\ou_{\eps,\ell}(0)$ for $t\geq0$ and $S_{\eps,\ell}(t)=0$ for $t<0$. We have
	\begin{equ}
		\nabla(\oudz_{\eps,\ell}-\oudz_{\eps,\ell,0})(t) 
		= 
		(\nabla K\ast(\one_>\oud_{\eps,\ell}-\oud_{\eps,\ell,0}))(t)
		+
		\nabla K(t)\ast\oudz_{\eps,\ell}(0)
	\end{equ}
	and
	\begin{equ}
		\ou_{\eps,\ell,0}
		=
		\one_>\ou_{\eps,\ell}-S_{\eps,\ell}\,,
		\qquad
		\oud_{\eps,\ell,0}
		=
		\one_>\oud_{\eps,\ell}-2\ou_{\eps,\ell}\rpara S_{\eps,\ell}-2\ou_{\eps,\ell}\pareq S_{\eps,\ell}+S^2_{\eps,\ell}\,.
	\end{equ}
	Using estimates for paraproducts and regularising effect of the heat kernel we obtain
	\begin{equ}
		t^{\kappa/2}\|\nabla(\oudz_{\eps,\ell}-\oudz_{\eps,\ell,0})(t)\|_{\cC^{\f\kappa4}(w)}
		\lesssim 
		\sup_{u\in[0,1]}\|\ou_{\eps,\ell}(u)\|^2_{\cC^{-\f12-\f\kappa8}(w)}
		+
		\|\oudz_{\eps,\ell}(0)\|_{\cC^{1-\f\kappa4}(w)} \,.
	\end{equ}
	By the standard stochastic estimates for the stationary trees $\ou_{\eps,\ell}$ and $\oudz_{\eps,\ell}$, see for example~\cite[Theorem~3.4]{GH19}, for any $p>1$ the expressions
	\begin{equ}
		\sup_{u\in[0,1]}\|\ou_{\eps,\ell}(u)\|_{\cC^{-\f12-\f\kappa8}(w)}\,,
		\qquad
		\sup_{u\in[0,1]}\|\oudz_{\eps,\ell}(u)\|_{\cC^{1-\f\kappa8}(w)}
	\end{equ}
	are bounded in $L^p(\Omega)$ uniformly over $\eps\in(0,1]$ and $\ell\in\N_+$. Thus, using the formula for $\Ctwo-C^{(2)}_{\eps,\ell,0}$ given above and the multiplication theorem in Besov spaces we obtain the second of the bounds~\eqref{eq:counterterm-bound}. To prove~\eqref{eq:C_diff_convergence} we note that
	\begin{equ}
		\Cone=\int_{\R^4} K_{\eps,\ell}(t,x)\,\md t\md x\,,
		\qquad
		\Ctwo= 2\int_{\R^4} ((K_{\eps,\ell}\ast K_{\eps,\ell})(t,x))^2 K_{\ell}(t,x)\,\md t\md x\,,		
	\end{equ}
	where $K_\ell$ coincides with the periodisation in space of the heat kernel $K$ with unit mass and $K_{\eps,\ell}=M_\eps\star K_\ell$, and use standard properties of the heat kernel.
\end{proof}

\begin{lem}\label{lem:kolmogorov}
	Let $d,n \in \N_+$, $p\geq 1$, $\beta < \alpha$ and $\hat w=\bracket{\bigcdot}^{-a}\in C(\R^d)$, $a>0$.	We have
	\begin{equ}
		\EE\left(\sup_{t\in[0,1]}\|\tau(t)\|^p_{\cC^\beta(\hat w)}\right) \lesssim C^p
	\end{equ}
	uniformly over $\ell\in\N_+$ and stationary in space stochastic processes $\tau \in C([0,1],C(\T_\ell^d))$ in the Wiener chaos of order $n$ such that 
	\begin{equ}\label{eq:covariance_kolmogorov}
		\EE|\langle\tau(t),e_k\rangle|^2
		\vee 
		|t_1-t_2|^{-2\kappa}\EE|\langle\tau(t_1)-\tau(t_2),e_k\rangle|^2 
		\leq
		\ell^d \,C^2 \bracket{k}^{-d-2\alpha}
	\end{equ}
	with $C>0$ for all $t,t_1,t_2\in[0,1]$, $t_1\neq t_2$, and all Fourier modes $e_k\in C(\T_\ell^d)$.
\end{lem}
\begin{proof}
	See~\cite[Proposition~5]{MWX17}.
\end{proof}

\begin{lem}\label{lem:stochastic_paracontrolled}
	Let $p\geq1$, $\hat w=\bracket{\bigcdot}^{-a}\in C(\R^3)$, $a>0$ and $\xi_{\eps,\ell}$ be constructed from the space-time white noise $\xi$ as specified in Definition~\ref{def:smoothed_noise}. Recall Definition~\ref{def:enhanced_noise} and set $\ou^\pm_{\eps,\ell}\eqdef K^\pm\ast \xi_{\eps,\ell}$ and $\ou^\pm\eqdef K^\pm\ast \xi$. The random variable
	\begin{equs}
		\sup_{t\in[s,s+1]}\Bigl(&\|\ou_{\eps,\ell,s}(t)\|_{\cC^{-\f12-\kappa}(\hat w)}
		\vee
		\|\oudz_{\eps,\ell,s}(t)\|_{\cC^{1-2\kappa}(\hat w)}
		\vee
		\|\outz_{\eps,\ell,s}(t)\|_{\cC^{1-3\kappa}(\hat w)}
		\\
		&\vee\,\|\ou_{\eps,\ell,s}(t)\reso \outz_{\eps,\ell,s}(t)\|_{\cC^{-4\kappa}(\hat w)}
		\vee
		\|(\nabla\oudz_{\eps,\ell,s}(t))^2-\Ctwos(t)\|_{\cC^{-4\kappa}(\hat w)}\Bigr)
	\end{equs}
	is bounded in $L^p(\Omega)$ uniformly in $\eps\in(0,1]$, $\ell\in\N_+$ and $s\in\R$. Moreover, for all $T>0$ we have
	\begin{equs}
		\lim_{\ell\to\infty}\lim_{\eps\searrow0}\sup_{t\in[0,T]}\|\ou^+(t)-\ou^+_{\eps,\ell}(t)\|_{\cC^{-\f12-\f\kappa2}(\hat w)}&=0\;,
		\\
		\lim_{\ell\to\infty}\lim_{\eps\searrow0}\sup_{t\in[0,T]}\|\lL\ou^-(t)-\lL\ou^-_{\eps,\ell}(t)\|_{L^\infty(\hat w)}&=0
	\end{equs}
	almost surely and in $L^p(\Omega)$.
\end{lem}

\begin{proof}
By translational invariance  we can restrict our attention to the case $s=0$. In order to 
prove the first part of the lemma it is enough to verify the covariance 
condition~\eqref{eq:covariance_kolmogorov} and to apply the Kolmogorov type estimate from 
Lemma~\ref{lem:kolmogorov}. To this end, one studies separately components $\tau^{(n)}_{\eps,\ell}(t)$ of the stochastic processes
	\begin{equ}
		\ou_{\eps,\ell,0}(t)\,,\quad
		\oudz_{\eps,\ell,0}(t)\,,\quad
		\outz_{\eps,\ell,0}(t)\,,\quad
		\ou_{\eps,\ell,0}(t)\reso \outz_{\eps,\ell,0}(t)\,,\quad
		(\nabla\oudz_{\eps,\ell,0}(t))^2-C^{(2)}_{\eps,\ell,0}(t)
	\end{equ} 
	in the $n$th Wiener chaos. The case $n=0$ is trivial as the expected values of the above processes vanish by definition. For $n\in\N_+$ the bounds for the covariances of the components of the first four processes from the list are quite standard and follow, for example, by a straightforward generalisation of the argument in \cite{MWX17} to infinite volume and trees with zero initial data. As argued in~\cite[Lemma~A.1]{JP23}, the proof of the bounds for $|\grad \oudz_{\eps,\ell,0}|^2-C^{(2)}_{\eps,\ell,0} $ is very similar to $\oud_{\eps,\ell} \reso \oudz_{\eps,\ell} -\Ctwo$, which was also discussed in \cite{MWX17}.
\end{proof}

\endappendix

\bibliographystyle{Martin}
\bibliography{phi43_mixing}

\begin{thebibliography}{AHKZ89b}
\def\myhref#1#2{\href{#2}{\nolinkurl{#1}}}
\ifdefined\urlprefix\else\def\urlprefix{}\fi

\bibitem[ADC21]{AD21}
\textsc{M.~Aizenman} and \textsc{H.~Duminil-Copin}.
\newblock Marginal triviality of the scaling limits of critical 4{D} {I}sing
  and {$\phi{}_4^4$} models.
\newblock \emph{Ann. of Math. (2)} \textbf{194}, no.~1, (2021), 163--235.
\newblock
  \myhref{doi:10.4007/annals.2021.194.1.3}{https://doi.org/10.4007/annals.2021.194.1.3}.

\bibitem[AHKZ89a]{AHZ89b}
\textsc{S.~Albeverio}, \textsc{R.~H{\o}egh-Krohn}, and
  \textsc{B.~Zegarli{\'n}ski}.
\newblock Uniqueness and global {M}arkov property for {E}uclidean fields: the
  case of general polynomial interactions.
\newblock \emph{Comm. Math. Phys.} \textbf{123}, no.~3, (1989), 377--424.
\newblock \urlprefix\url{http://projecteuclid.org/euclid.cmp/1104178888}.

\bibitem[AHKZ89b]{AHZ89a}
\textsc{S.~Albeverio}, \textsc{R.~H{\o}egh-Krohn}, and
  \textsc{B.~Zegarli{\'n}ski}.
\newblock Uniqueness of {G}ibbs states for general {$P(\varphi{})_2$}-weak
  coupling models by cluster expansion.
\newblock \emph{Comm. Math. Phys.} \textbf{121}, no.~4, (1989), 683--697.
\newblock \urlprefix\url{http://projecteuclid.org/euclid.cmp/1104178254}.

\bibitem[Aiz82]{Aiz82}
\textsc{M.~Aizenman}.
\newblock Geometric analysis of {$\varphi\sp{4}$}\ fields and {I}sing models.
  {I}, {II}.
\newblock \emph{Comm. Math. Phys.} \textbf{86}, no.~1, (1982), 1--48.
\newblock \urlprefix\url{http://projecteuclid.org/euclid.cmp/1103921614}.

\bibitem[AK20]{AK20}
\textsc{S.~Albeverio} and \textsc{S.~Kusuoka}.
\newblock The invariant measure and the flow associated to the
  {$\Phi^4_3$}-quantum field model.
\newblock \emph{Ann. Sc. Norm. Super. Pisa Cl. Sci. (5)} \textbf{20}, no.~4,
  (2020), 1359--1427.
\newblock \myhref{arXiv:1711.07108}{https://arxiv.org/abs/1711.07108}.

\bibitem[AKR97]{albeverio1997ergodicity}
\textsc{S.~Albeverio}, \textsc{Y.~G. Kondratiev}, and \textsc{M.~R\"ockner}.
\newblock Ergodicity for the stochastic dynamics of quasi-invariant measures
  with applications to {G}ibbs states.
\newblock \emph{J. Funct. Anal.} \textbf{149}, no.~2, (1997), 415--469.
\newblock
  \myhref{doi:10.1006/jfan.1997.3099}{https://doi.org/10.1006/jfan.1997.3099}.

\bibitem[BCD11]{BCD11}
\textsc{H.~Bahouri}, \textsc{J.-Y. Chemin}, and \textsc{R.~Danchin}.
\newblock \emph{Fourier analysis and nonlinear partial differential equations},
  vol. 343 of \emph{Grundlehren der mathematischen Wissenschaften [Fundamental
  Principles of Mathematical Sciences]}.
\newblock Springer, Heidelberg, 2011,  xvi+523.
\newblock
  \myhref{doi:10.1007/978-3-642-16830-7}{https://doi.org/10.1007/978-3-642-16830-7}.

\bibitem[BD24]{RolandLogSob}
\textsc{R.~Bauerschmidt} and \textsc{B.~Dagallier}.
\newblock Log-{S}obolev inequality for the {$\varphi{}^4_2$} and
  {$\varphi{}^4_3$} measures.
\newblock \emph{Comm. Pure Appl. Math.} \textbf{77}, no.~5, (2024), 2579--2612.
\newblock \myhref{doi:10.1002/cpa.22173}{https://doi.org/10.1002/cpa.22173}.

\bibitem[BDH95]{BDH95}
\textsc{D.~Brydges}, \textsc{J.~Dimock}, and \textsc{T.~R. Hurd}.
\newblock The short distance behavior of {$(\phi^4)_3$}.
\newblock \emph{Comm. Math. Phys.} \textbf{172}, no.~1, (1995), 143--186.
\newblock \urlprefix\url{http://projecteuclid.org/euclid.cmp/1104273962}.

\bibitem[BDW25]{RolandHendrik}
\textsc{R.~Bauerschmidt}, \textsc{B.~Dagallier}, and \textsc{H.~Weber}.
\newblock {H}olley--{S}troock uniqueness method for the $\varphi{}^4_2$
  dynamics.
\newblock \emph{arXiv preprint} (2025).
\newblock \myhref{arXiv:2504.08606}{https://arxiv.org/abs/2504.08606}.

\bibitem[BFS83]{brydges1983new}
\textsc{D.~C. Brydges}, \textsc{J.~Fr\"ohlich}, and \textsc{A.~D. Sokal}.
\newblock A new proof of the existence and nontriviality of the continuum
  {$\varphi \sp{4}\sb{2}$}\ and {$\varphi \sp{4}\sb{3}$}\ quantum field
  theories.
\newblock \emph{Comm. Math. Phys.} \textbf{91}, no.~2, (1983), 141--186.
\newblock \urlprefix\url{http://projecteuclid.org/euclid.cmp/1103940528}.

\bibitem[BG25]{BG25}
\textsc{N.~Barashkov} and \textsc{T.~S. Gunaratnam}.
\newblock The {M}arkov property for $\varphi^4_3$ on the cylinder.
\newblock \emph{arXiv preprint} (2025).
\newblock \myhref{arXiv:2506.21466}{https://arxiv.org/abs/2506.21466}.

\bibitem[BHZ19]{BHZ19}
\textsc{Y.~Bruned}, \textsc{M.~Hairer}, and \textsc{L.~Zambotti}.
\newblock Algebraic renormalisation of regularity structures.
\newblock \emph{Invent. Math.} \textbf{215}, no.~3, (2019), 1039--1156.
\newblock
  \myhref{doi:10.1007/s00222-018-0841-x}{https://doi.org/10.1007/s00222-018-0841-x}.

\bibitem[BRW04]{BRW04}
\textsc{V.~I. Bogachev}, \textsc{M.~R\"ockner}, and \textsc{F.-Y. Wang}.
\newblock Invariance implies {G}ibbsian: some new results.
\newblock \emph{Comm. Math. Phys.} \textbf{248}, no.~2, (2004), 335--355.
\newblock
  \myhref{doi:10.1007/s00220-004-1096-5}{https://doi.org/10.1007/s00220-004-1096-5}.

\bibitem[CC18]{CC18}
\textsc{R.~Catellier} and \textsc{K.~Chouk}.
\newblock Paracontrolled distributions and the 3-dimensional stochastic
  quantization equation.
\newblock \emph{Ann. Probab.} \textbf{46}, no.~5, (2018), 2621--2679.
\newblock \myhref{doi:10.1214/17-AOP1235}{https://doi.org/10.1214/17-AOP1235}.

\bibitem[CCHS22]{CCHS22}
\textsc{A.~Chandra}, \textsc{I.~Chevyrev}, \textsc{M.~Hairer}, and
  \textsc{H.~Shen}.
\newblock Langevin dynamic for the 2{D} {Y}ang-{M}ills measure.
\newblock \emph{Publ. Math. Inst. Hautes {\'E}tudes Sci.} \textbf{136}, (2022),
  1--147.
\newblock
  \myhref{doi:10.1007/s10240-022-00132-0}{https://doi.org/10.1007/s10240-022-00132-0}.

\bibitem[CGW22]{CGW22}
\textsc{A.~Chandra}, \textsc{T.~S. Gunaratnam}, and \textsc{H.~Weber}.
\newblock Phase transitions for {$\phi{}_3^4$}.
\newblock \emph{Comm. Math. Phys.} \textbf{392}, no.~2, (2022), 691--782.
\newblock
  \myhref{doi:10.1007/s00220-022-04353-6}{https://doi.org/10.1007/s00220-022-04353-6}.

\bibitem[DDJ25]{DDJ25}
\textsc{P.~Duch}, \textsc{W.~Dybalski}, and \textsc{A.~Jahandideh}.
\newblock Stochastic quantization of two-dimensional {$P(\Phi)$} quantum field
  theory.
\newblock \emph{Ann. Henri Poincar\'e} \textbf{26}, no.~3, (2025), 1055--1086.
\newblock
  \myhref{doi:10.1007/s00023-024-01447-w}{https://doi.org/10.1007/s00023-024-01447-w}.

\bibitem[DGR24]{DGR23}
\textsc{P.~Duch}, \textsc{M.~Gubinelli}, and \textsc{P.~Rinaldi}.
\newblock Parabolic stochastic quantisation of the fractional $\phi^4_3$ model
  in the full subcritical regime.
\newblock \emph{arXiv preprint} (2024).
\newblock \myhref{arXiv:2303.18112}{https://arxiv.org/abs/2303.18112}.

\bibitem[DPD03]{DPD03}
\textsc{G.~Da~Prato} and \textsc{A.~Debussche}.
\newblock Strong solutions to the stochastic quantization equations.
\newblock \emph{Ann. Probab.} \textbf{31}, no.~4, (2003), 1900--1916.
\newblock
  \myhref{doi:10.1214/aop/1068646370}{https://doi.org/10.1214/aop/1068646370}.

\bibitem[DR79]{doss1978processus}
\textsc{H.~Doss} and \textsc{G.~Royer}.
\newblock Processus de diffusion associ{\'e} aux mesures de {G}ibbs sur {${\bf
  R}\sp{{\bf Z}\sp{{\bf d}}}$}.
\newblock \emph{Zeitschrift f{\"u}r Wahrscheinlichkeitstheorie und Verwandte
  Gebiete} \textbf{46}, no.~1, (1978/79), 107--124.
\newblock \myhref{doi:10.1007/BF00535690}{https://doi.org/10.1007/BF00535690}.

\bibitem[Fel74]{Fel74}
\textsc{J.~Feldman}.
\newblock The {$\lambda{}\varphi\sp{4}\sb{3}$} field theory in a finite volume.
\newblock \emph{Comm. Math. Phys.} \textbf{37}, (1974), 93--120.
\newblock \urlprefix\url{http://projecteuclid.org/euclid.cmp/1103859849}.

\bibitem[FO76]{FO76}
\textsc{J.~S. Feldman} and \textsc{K.~Osterwalder}.
\newblock The {W}ightman axioms and the mass gap for weakly coupled
  $\phi{}^4_3$ quantum field theories.
\newblock \emph{Annals of Physics} \textbf{97}, no.~1, (1976), 80--135.
\newblock
  \myhref{doi:10.1016/0003-4916(76)90223-2}{https://doi.org/10.1016/0003-4916(76)90223-2}.

\bibitem[Fri82]{FritzGibbs}
\textsc{J.~Fritz}.
\newblock Stationary measures of stochastic gradient systems, infinite lattice
  models.
\newblock \emph{Z. Wahrsch. Verw. Gebiete} \textbf{59}, no.~4, (1982),
  479--490.
\newblock \myhref{doi:10.1007/BF00532804}{https://doi.org/10.1007/BF00532804}.

\bibitem[Fr{\"o}82]{Fro82}
\textsc{J.~Fr{\"o}hlich}.
\newblock On the triviality of {$\lambda{}\varphi\sp{4}\sb{d}$}\ theories and
  the approach to the critical point in {$d \geq{}4$}\ dimensions.
\newblock \emph{Nuclear Phys. B} \textbf{200}, no.~2, (1982), 281--296.
\newblock
  \myhref{doi:10.1016/0550-3213(82)90088-8}{https://doi.org/10.1016/0550-3213(82)90088-8}.

\bibitem[FSS76]{FSS76}
\textsc{J.~Fr\"ohlich}, \textsc{B.~Simon}, and \textsc{T.~Spencer}.
\newblock Infrared bounds, phase transitions and continuous symmetry breaking.
\newblock \emph{Comm. Math. Phys.} \textbf{50}, no.~1, (1976), 79--95.
\newblock \urlprefix\url{http://projecteuclid.org/euclid.cmp/1103900151}.

\bibitem[Fun91]{Fun91}
\textsc{T.~Funaki}.
\newblock The reversible measures of multi-dimensional {G}inzburg-{L}andau type
  continuum model.
\newblock \emph{Osaka J. Math.} \textbf{28}, no.~3, (1991), 463--494.
\newblock \urlprefix\url{https://projecteuclid.org/euclid.ojm/1200783221}.

\bibitem[GH19]{GH19}
\textsc{M.~Gubinelli} and \textsc{M.~Hofmanov{\'a}}.
\newblock Global solutions to elliptic and parabolic {$\Phi{}^4$} models in
  {E}uclidean space.
\newblock \emph{Comm. Math. Phys.} \textbf{368}, no.~3, (2019), 1201--1266.
\newblock
  \myhref{doi:10.1007/s00220-019-03398-4}{https://doi.org/10.1007/s00220-019-03398-4}.

\bibitem[GH21]{GH21}
\textsc{M.~Gubinelli} and \textsc{M.~Hofmanov{\'a}}.
\newblock A {PDE} construction of the {E}uclidean {$\phi{}_3^4$} quantum field
  theory.
\newblock \emph{Comm. Math. Phys.} \textbf{384}, no.~1, (2021), 1--75.
\newblock
  \myhref{doi:10.1007/s00220-021-04022-0}{https://doi.org/10.1007/s00220-021-04022-0}.

\bibitem[GHR25]{GHR25}
\textsc{M.~Gubinelli}, \textsc{M.~Hofmanová}, and \textsc{N.~Rana}.
\newblock Decay of correlations in stochastic quantization: the exponential
  {E}uclidean field in two dimensions.
\newblock \emph{Stochastic Partial Differential Equations. Analysis and
  Computations} \textbf{13}, no.~1, (2025), 107–145.
\newblock
  \myhref{doi:10.1007/s40072-024-00328-x}{https://doi.org/10.1007/s40072-024-00328-x}.

\bibitem[GIP15]{GIP15}
\textsc{M.~Gubinelli}, \textsc{P.~Imkeller}, and \textsc{N.~Perkowski}.
\newblock Paracontrolled distributions and singular {PDE}s.
\newblock \emph{Forum Math. Pi} \textbf{3}, (2015), e6, 75.
\newblock \myhref{doi:10.1017/fmp.2015.2}{https://doi.org/10.1017/fmp.2015.2}.

\bibitem[GJ73]{GJ73}
\textsc{J.~Glimm} and \textsc{A.~Jaffe}.
\newblock Positivity of the {$\phi \sp{4}\sb{3}$} {H}amiltonian.
\newblock \emph{Fortschr. Physik} \textbf{21}, (1973), 327--376.
\newblock
  \myhref{doi:10.1002/prop.19730210702}{https://doi.org/10.1002/prop.19730210702}.

\bibitem[GJ87]{GJ87}
\textsc{J.~Glimm} and \textsc{A.~Jaffe}.
\newblock \emph{Quantum physics}.
\newblock Springer-Verlag, New York, second ed., 1987,  xxii+535.
\newblock A functional integral point of view.
\newblock
  \myhref{doi:10.1007/978-1-4612-4728-9}{https://doi.org/10.1007/978-1-4612-4728-9}.

\bibitem[GJS75]{GJS75}
\textsc{J.~Glimm}, \textsc{A.~Jaffe}, and \textsc{T.~Spencer}.
\newblock Phase transitions for {$\phi\sb{2}\sp{4}$} quantum fields.
\newblock \emph{Comm. Math. Phys.} \textbf{45}, no.~3, (1975), 203--216.
\newblock \urlprefix\url{http://projecteuclid.org/euclid.cmp/1103899492}.

\bibitem[GJS76a]{GJS76a}
\textsc{J.~Glimm}, \textsc{A.~Jaffe}, and \textsc{T.~Spencer}.
\newblock A convergent expansion about mean field theory. {I}. {T}he expansion.
\newblock \emph{Ann. Physics} \textbf{101}, no.~2, (1976), 610--630.
\newblock
  \myhref{doi:10.1016/0003-4916(76)90026-9}{https://doi.org/10.1016/0003-4916(76)90026-9}.

\bibitem[GJS76b]{GJS76b}
\textsc{J.~Glimm}, \textsc{A.~Jaffe}, and \textsc{T.~Spencer}.
\newblock A convergent expansion about mean field theory. {I}. {T}he expansion.
\newblock \emph{Ann. Physics} \textbf{101}, no.~2, (1976), 610--630.
\newblock
  \myhref{doi:10.1016/0003-4916(76)90026-9}{https://doi.org/10.1016/0003-4916(76)90026-9}.

\bibitem[GMW25]{GMW25}
\textsc{P.~Grazieschi}, \textsc{K.~Matetski}, and \textsc{H.~Weber}.
\newblock The dynamical {I}sing-{K}ac model in {$3D$} converges to
  {$\Phi{}^4_3$}.
\newblock \emph{Probab. Theory Related Fields} \textbf{191}, no. 1-2, (2025),
  671--778.
\newblock
  \myhref{doi:10.1007/s00440-024-01316-x}{https://doi.org/10.1007/s00440-024-01316-x}.

\bibitem[GT20]{GT20}
\textsc{B.~Gess} and \textsc{P.~Tsatsoulis}.
\newblock Synchronization by noise for the stochastic quantization equation in
  dimensions 2 and 3.
\newblock \emph{Stoch. Dyn.} \textbf{20}, no.~6, (2020), 2040006, 17.
\newblock
  \myhref{doi:10.1142/S0219493720400067}{https://doi.org/10.1142/S0219493720400067}.

\bibitem[Hai14]{Hai14}
\textsc{M.~Hairer}.
\newblock A theory of regularity structures.
\newblock \emph{Invent. Math.} \textbf{198}, no.~2, (2014), 269--504.
\newblock
  \myhref{doi:10.1007/s00222-014-0505-4}{https://doi.org/10.1007/s00222-014-0505-4}.

\bibitem[HI18]{HI18}
\textsc{M.~Hairer} and \textsc{M.~Iberti}.
\newblock Tightness of the {I}sing-{K}ac model on the two-dimensional torus.
\newblock \emph{J. Stat. Phys.} \textbf{171}, no.~4, (2018), 632--655.
\newblock
  \myhref{doi:10.1007/s10955-018-2033-x}{https://doi.org/10.1007/s10955-018-2033-x}.

\bibitem[HL15]{HL15}
\textsc{M.~Hairer} and \textsc{C.~Labb{\'e}}.
\newblock A simple construction of the continuum parabolic {A}nderson model on
  {${\bf R}^2$}.
\newblock \emph{Electron. Commun. Probab.} \textbf{20}, (2015), no. 43, 11.
\newblock
  \myhref{doi:10.1214/ECP.v20-4038}{https://doi.org/10.1214/ECP.v20-4038}.

\bibitem[HL18]{HL18}
\textsc{M.~Hairer} and \textsc{C.~Labb{\'e}}.
\newblock Multiplicative stochastic heat equations on the whole space.
\newblock \emph{J. Eur. Math. Soc. (JEMS)} \textbf{20}, no.~4, (2018),
  1005--1054.
\newblock \myhref{doi:10.4171/JEMS/781}{https://doi.org/10.4171/JEMS/781}.

\bibitem[HM18a]{HM18}
\textsc{M.~Hairer} and \textsc{K.~Matetski}.
\newblock Discretisations of rough stochastic {PDE}s.
\newblock \emph{Ann. Probab.} \textbf{46}, no.~3, (2018), 1651--1709.
\newblock \myhref{doi:10.1214/17-AOP1212}{https://doi.org/10.1214/17-AOP1212}.

\bibitem[HM18b]{HM18b}
\textsc{M.~Hairer} and \textsc{J.~Mattingly}.
\newblock The strong {F}eller property for singular stochastic pdes.
\newblock \emph{Annales de l'Institut Henri Poincar{\'e}, Probabilit{\'e}s et
  Statistiques} \textbf{54}, no.~3, (2018), 1314 -- 1340.
\newblock \myhref{doi:10.1214/17-AIHP840}{https://doi.org/10.1214/17-AIHP840}.

\bibitem[HS77]{HS77}
\textsc{R.~A. Holley} and \textsc{D.~W. Stroock}.
\newblock In one and two dimensions, every stationary measure for a stochastic
  {I}sing model is a {G}ibbs state.
\newblock \emph{Communications in Mathematical Physics} \textbf{55}, no.~1,
  (1977), 37--45.
\newblock \urlprefix\url{http://projecteuclid.org/euclid.cmp/1103900928}.

\bibitem[HS22a]{HS22b}
\textsc{M.~Hairer} and \textsc{P.~Sch{\"o}nbauer}.
\newblock The support of singular stochastic partial differential equations.
\newblock \emph{Forum of Mathematics, Pi} \textbf{10}, (2022), e1.
\newblock
  \myhref{doi:10.1017/fmp.2021.18}{https://doi.org/10.1017/fmp.2021.18}.

\bibitem[HS22b]{HS22}
\textsc{M.~Hairer} and \textsc{R.~Steele}.
\newblock The {$\Phi{}_3^4$} measure has sub-{G}aussian tails.
\newblock \emph{J. Stat. Phys.} \textbf{186}, no.~3, (2022), Paper No. 38, 25.
\newblock
  \myhref{doi:10.1007/s10955-021-02866-3}{https://doi.org/10.1007/s10955-021-02866-3}.

\bibitem[HS24]{HS24}
\textsc{M.~Hairer} and \textsc{R.~Steele}.
\newblock The {BPHZ} theorem for regularity structures via the spectral gap
  inequality.
\newblock \emph{Arch. Ration. Mech. Anal.} \textbf{248}, no.~1, (2024), Paper
  No. 9, 81.
\newblock
  \myhref{doi:10.1007/s00205-023-01946-w}{https://doi.org/10.1007/s00205-023-01946-w}.

\bibitem[JP23]{JP23}
\textsc{A.~Jagannath} and \textsc{N.~Perkowski}.
\newblock A simple construction of the dynamical {$\Phi{}^4_3$} model.
\newblock \emph{Trans. Amer. Math. Soc.} \textbf{376}, no.~3, (2023),
  1507--1522.
\newblock \myhref{doi:10.1090/tran/8724}{https://doi.org/10.1090/tran/8724}.

\bibitem[KT25]{KT22}
\textsc{F.~Kunick} and \textsc{P.~Tsatsoulis}.
\newblock Gradient-type estimates for the dynamic {$\varphi_2^4$}-model.
\newblock \emph{Stochastic Process. Appl.} \textbf{181}, (2025), Paper No.
  104548, 21.
\newblock
  \myhref{doi:10.1016/j.spa.2024.104548}{https://doi.org/10.1016/j.spa.2024.104548}.

\bibitem[MS76]{MS76}
\textsc{J.~Magnen} and \textsc{R.~Seneor}.
\newblock The infinite volume limit of the $\varphi{}^4_3$ model.
\newblock \emph{Annales de l'institut Henri Poincar{\'e}. Section A, Physique
  Th{\'e}orique} \textbf{24}, no.~2, (1976), 95--159.
\newblock \urlprefix\url{https://www.numdam.org/item/AIHPA_1976__24_2_95_0/}.

\bibitem[MW17a]{MW17c}
\textsc{J.-C. Mourrat} and \textsc{H.~Weber}.
\newblock Convergence of the two-dimensional dynamic {I}sing-{K}ac model to
  {$\Phi{}^4_2$}.
\newblock \emph{Comm. Pure Appl. Math.} \textbf{70}, no.~4, (2017), 717--812.
\newblock \myhref{doi:10.1002/cpa.21655}{https://doi.org/10.1002/cpa.21655}.

\bibitem[MW17b]{MW17b}
\textsc{J.-C. Mourrat} and \textsc{H.~Weber}.
\newblock The dynamic {$\Phi{}^4_3$} model comes down from infinity.
\newblock \emph{Comm. Math. Phys.} \textbf{356}, no.~3, (2017), 673--753.
\newblock
  \myhref{doi:10.1007/s00220-017-2997-4}{https://doi.org/10.1007/s00220-017-2997-4}.

\bibitem[MW17c]{MW17a}
\textsc{J.-C. Mourrat} and \textsc{H.~Weber}.
\newblock Global well-posedness of the dynamic {$\Phi{}^4$} model in the plane.
\newblock \emph{Ann. Probab.} \textbf{45}, no.~4, (2017), 2398--2476.
\newblock \myhref{doi:10.1214/16-AOP1116}{https://doi.org/10.1214/16-AOP1116}.

\bibitem[MW20]{MW20}
\textsc{A.~Moinat} and \textsc{H.~Weber}.
\newblock Space-time localisation for the dynamic {$\Phi{}^4_3$} model.
\newblock \emph{Comm. Pure Appl. Math.} \textbf{73}, no.~12, (2020),
  2519--2555.
\newblock \myhref{doi:10.1002/cpa.21925}{https://doi.org/10.1002/cpa.21925}.

\bibitem[MWX17]{MWX17}
\textsc{J.-C. Mourrat}, \textsc{H.~Weber}, and \textsc{W.~Xu}.
\newblock Construction of {$\Phi{}^4_3$} diagrams for pedestrians.
\newblock In \emph{From particle systems to partial differential equations},
  vol. 209 of \emph{Springer Proc. Math. Stat.},  1--46. Springer, Cham, 2017.
\newblock
  \myhref{doi:10.1007/978-3-319-66839-0_1}{https://doi.org/10.1007/978-3-319-66839-0_1}.

\bibitem[OS75]{OS75}
\textsc{K.~Osterwalder} and \textsc{R.~Schrader}.
\newblock Axioms for {E}uclidean {G}reen's functions. {II}.
\newblock \emph{Comm. Math. Phys.} \textbf{42}, (1975), 281--305.
\newblock With an appendix by Stephen Summers.
\newblock \urlprefix\url{http://projecteuclid.org/euclid.cmp/1103899050}.

\bibitem[Par75]{MR418721}
\textsc{Y.~M. Park}.
\newblock Lattice approximation of the {$(\lambda \phi\sp{4}-\mu \phi)\sb{3}$}\
  field theory in a finite volume.
\newblock \emph{J. Mathematical Phys.} \textbf{16}, (1975), 1065--1075.
\newblock \myhref{doi:10.1063/1.522661}{https://doi.org/10.1063/1.522661}.

\bibitem[PW81]{PW81}
\textsc{G.~Parisi} and \textsc{Y.~S. Wu}.
\newblock Perturbation theory without gauge fixing.
\newblock \emph{Sci. Sinica} \textbf{24}, no.~4, (1981), 483--496.

\bibitem[Tri06]{Tri06}
\textsc{H.~Triebel}.
\newblock \emph{Theory of function spaces. {III}}, vol. 100 of \emph{Monographs
  in Mathematics}.
\newblock Birkh\"auser Verlag, Basel, 2006,  xii+426.

\bibitem[Wat89]{Wat89}
\textsc{H.~Watanabe}.
\newblock Block spin approach to {$\phi^4_3$} field theory.
\newblock \emph{J. Statist. Phys.} \textbf{54}, no. 1-2, (1989), 171--190.
\newblock \myhref{doi:10.1007/BF01023477}{https://doi.org/10.1007/BF01023477}.

\bibitem[Wig76]{Wig76}
\textsc{A.~S. Wightman}.
\newblock Hilbert's sixth problem: mathematical treatment of the axioms of
  physics.
\newblock In \emph{Mathematical developments arising from {H}ilbert problems
  ({P}roc. {S}ympos. {P}ure {M}ath., {N}orthern {I}llinois {U}niv., {D}e
  {K}alb, {I}ll., 1974)}, vol. Vol. XXVIII of \emph{Proc. Sympos. Pure Math.},
  147--240. Amer. Math. Soc., Providence, RI, 1976.

\end{thebibliography}

\end{document}